\documentclass[11pt, letterpaper]{amsart}

\usepackage{amssymb,euscript,tikz,units}
\usepackage[colorlinks,citecolor=blue,linkcolor=red]{hyperref}
\usepackage{verbatim}
\usepackage[nameinlink]{cleveref}
\usepackage{enumitem}
\usepackage{hyperref}
\usepackage{url}

\newcommand{\N}{\mathbb{N}}
\newcommand{\calT}{\mathcal{T}}
\newcommand{\T}{\calT}
\newcommand{\M}{\mathcal{M}}
\newcommand{\sM}{\mathcal{M}}

\newcommand{\R}{\mathbb{R}}
\newcommand{\Z}{\mathbb{Z}}
\renewcommand{\H}{\mathbb{H}}

\newcommand{\eps}{\epsilon}
\newcommand{\wep}{Weil-Petersson}

\newcommand{\sbs}{\subset}

\newcommand{\sT}{\mathcal{T}}
\newcommand{\dvol}{d\mathit{vol}}

\newcommand{\ve}{\boldsymbol}

\newcommand{\eg}{\textit{e.g.\@ }}

\newcommand{\ie}{\textit{i.e.\@ }}

\def\area{\mathop{\rm Area}}
\def\arcsinh{\mathop{\rm arcsinh}}

\def\Vol{\mathop{\rm Vol}}
\def\Supp{\mathop{\rm Supp}}

\def\inj{\mathop{\rm inj}}
\def\Prob{\mathop{\rm Prob}\nolimits_{\rm WP}^g}
\def\Mod{\mathop{\rm Mod}}
\def\Sym{\mathop{\rm Sym}}

\newcommand{\lmdx}{\lambda_1(X)}

\DeclareMathOperator{\Diff}{Diff}

\DeclareMathOperator{\WP}{WP}
\DeclareMathOperator{\RHS}{RHS}
\DeclareMathOperator{\LHS}{LHS}
\DeclareMathOperator{\diam}{diam}

\theoremstyle{plain}
\newtheorem{theorem}{Theorem}

\newtheorem{proposition}[theorem]{Proposition}
\newtheorem{lemma}[theorem]{Lemma}

\newtheorem{question}[theorem]{Question}
\newtheorem{remark}[theorem]{Remark}

\newcommand{\be}{\begin{equation}}
\newcommand{\ene}{\end{equation}}
\newcommand{\br}{\begin{remark}}
\newcommand{\er}{\end{remark}}
\newcommand{\bl}{\begin{lemma}}
\newcommand{\el}{\end{lemma}}
\newcommand{\bcor}{\begin{cor}}
\newcommand{\ecor}{\end{cor}}
\newcommand{\bpro}{\begin{pro}}
\newcommand{\epro}{\end{pro}}
\newcommand{\ben}{\begin{enumerate}}
\newcommand{\een}{\end{enumerate}}
\newcommand{\bp}{\begin{proof}}
\newcommand{\ep}{\end{proof}}
\newcommand{\bpo}{\begin{pro}}
\newcommand{\epo}{\end{pro}}
\newcommand{\beq}{\begin{equation*}}
\newcommand{\eeq}{\end{equation*}}
\newcommand{\bear}{\begin{eqnarray}}
\newcommand{\eear}{\end{eqnarray}}
\newcommand{\beqar}{\begin{eqnarray*}}
\newcommand{\eeqar}{\end{eqnarray*}}
\newcommand{\bt}{\begin{theorem}}
\newcommand{\et}{\end{theorem}}
\newcommand{\bex}{\begin{excer}}
\newcommand{\eex}{\end{excer}}

\theoremstyle{definition}

\theoremstyle{remark}

\newtheorem*{rem*}{Remark}
\newtheorem*{def*}{Definition}
\newtheorem*{con*}{Construction}
\newtheorem*{thm*}{\bf Theorem}
\newtheorem*{definition*}{Definition}
\newtheorem*{assum*}{Assumption $(\star)$}
\newtheorem*{obs*}{Observation}

\makeatletter
\@namedef{subjclassname@2020}{\textup{2020} Mathematics Subject Classification}
\makeatother


\begin{document}

\title[Large first eigenvalues for Random surfaces]{Random hyperbolic surfaces of large genus have first eigenvalues greater than $\frac{3}{16}-\epsilon$}
\author{Yunhui Wu and Yuhao Xue}

\address{Department of Mathematics and Yau Mathematical Sciences Center, Tsinghua University, Beijing, China}
\email[(Y.~W.)]{yunhui\_wu@tsinghua.edu.cn}
\email[(Y.~X.)]{xueyh18@mails.tsinghua.edu.cn}

\date{}
\maketitle

\begin{abstract}
Let $\sM_g$ be the moduli space of hyperbolic surfaces of genus $g$ endowed with the Weil-Petersson metric. In this paper, we show that for any $\eps>0$, as genus $g$ goes to infinity, a generic surface $X\in \M_g$ satisfies that the first eigenvalue $\lambda_1(X)>\frac{3}{16}-\epsilon$. As an application, we also show that a generic surface $X\in \M_g$ satisfies that the diameter $\diam(X)<(4+\epsilon)\ln(g)$ for large genus.
\end{abstract}

\section{Introduction}

Let $X_g$ be a hyperbolic surface of genus $g\geq 2$ and $\lambda_1(X_g)$ be the first eigenvalue of the Laplacian operator on $X_g$. For large genus, it is known (\eg see  \cite{Hub74} or \cite{Cheng75}) that  $$\limsup \limits_{g\to \infty} \lambda_1(X_g)\leq \frac{1}{4}$$ for any sequence of hyperbolic surfaces $\{X_g\}_{g\geq 2}$. The motivation of this work is whether random objects have large, or even optimal, first eigenvalues. Brooks and Makover in \cite{BM04} showed that there exists a uniform positive lower bound for the first eigenvalues of random surfaces in their discrete model by gluing ideal hyperbolic triangles. In this work we view the first eigenvalue as a random variable with respect to the probability measure $\Prob$ on moduli pace $\M_g$ of Riemann surfaces of genus $g$ given by the Weil-Petersson metric, which was initiated by Mirzakhani in \cite{Mirz13}. Based on her celebrated thesis works \cite{Mirz07,Mirz07-int}, Mirzakhani proved the following result via the Cheeger inequality by showing \cite[Theorem 4.8]{Mirz13} that the Cheeger constant of a generic surface $X\in \M_g$ is greater than or equal to $\frac{\ln(2)}{2\pi+\ln(2)}$ as $g$ tends to infinity. More precisely, she showed that
\[\lim \limits_{g\to \infty}\Prob\left(X\in \M_g; \ \lmdx\geq \frac{1}{4} \left( \frac{\ln(2)}{2\pi+\ln(2)}\right)^2\sim 0.00247\right)=1.\]

The main result of this paper is the following.
\bt \label{main}
For any $\eps>0$, we have
\[\lim \limits_{g\to \infty}\Prob\left(X\in \M_g; \ \lmdx>\frac{3}{16}-\eps \right)=1.\]
\et

\noindent It is \emph{unknown} whether Theorem \ref{main} also holds with $\frac{3}{16}$ replaced by $\frac{1}{4}$ (\eg see \cite[Problem 10.4]{Wright-tour}). We remark here that it is very recently proved by Hide and Magee \cite{HM21} that there exists a sequence of hyperbolic surfaces $\{X_{g_n}\}$ with genus $g_n$ going to infinity such that $\lambda_1(X_{g_n})$ tends to $\frac{1}{4}$ (\eg see \cite{BBD88}, \cite[Conjecture 1.2]{Mondal15},  \cite[Conjecture 5]{WX18} and \cite[Problem 10.3]{Wright-tour}).

\begin{rem*}
 Mysteriously,
\begin{enumerate}
\item the number $\frac{3}{16}$ is the lower bound in a celebrated theorem of Selberg \cite{Sel65} saying that congruence covers of the moduli surface $\H/\rm{SL(2,\Z)}$ have first eigenvalues $\geq \frac{3}{16}$. Meanwhile in \cite{Sel65} the \emph{Selberg eigenvalue conjecture} was also proposed, which asserts that the lower bound $\frac{3}{16}$ can be improved to be $\frac{1}{4}$. Kim and Sarnak \cite{KS03} proved the congruence covers of $\H/\rm{SL(2,\Z)}$ have first eigenvalues $\geq \frac{975}{4096}$. One may also see \eg \cite{GJ78, Iwa89, LRS95, Sar95-survey, Iwa96, KS02} for intermediate results.\\

\item The number $\frac{3}{16}$ also appears in a recent work \cite{MNP20} of Magee-Naud-Puder showing that for any closed hyperbolic surface $X$, then as the covering degree tends to infinity, it holds asymptotically almost surely that a generic covering surface $X_\phi$ of $X$ satisfies that for any $\eps>0$,
\[\rm{spec(\Delta_{X_{\phi}})}\cap [0,\frac{3}{16}-\eps]=\rm{spec(\Delta_{X})}\cap [0,\frac{3}{16}-\eps]\]
where $\rm{spec(\Delta_{\bullet})}$ is the spectrum of the Laplacian operator of $\bullet$. They also \emph{conjecture} in \cite{MNP20} that the number $\frac{3}{16}$ can be replaced by $\frac{1}{4}$. In particular, for the case that $S=\mathcal{B}_2$ is the Bolza surface which is known \cite{Jenni84} that the first eigenvalue $\lambda_1(\mathcal{B}_2)\sim 3.838>\frac{3}{16}$, the result of Magee-Naud-Puder above implies that as the covering degree tends to infinity, it holds asymptotically almost surely that a generic covering surface $S'$ of $\mathcal{B}_2$ satisfies that the first eigenvalues $\lambda_1(S')> \frac{3}{16}-\epsilon$ for any $\epsilon>0$. 
\end{enumerate}
\end{rem*}

\begin{rem*}
Our proof of Theorem \ref{main} in this paper is completely different from the proof of \cite[Theorem 4.8]{Mirz13} that one may also see \cite[Page 1142]{Mirz10} of Mirzakhani's 2010 ICM report for similar results. We will use Selberg's trace formula as a tool, and then find an effective way to compute the integral of Selberg's trace formula over moduli space $\sM_g$ endowed with the \wep \ measure for large genus. 
\end{rem*}

\subsection*{Recent related works.} Recently there are several important developments on related works. For example: Theorem \ref{main} is independently proved by Lipnowski and Wright in \cite{LW21} by an alternative method; Hide in \cite{Hide21} extends Theorem \ref{main} to surfaces with punctures. Hide and Magee in \cite{HM21} show that $\rm{spec(\Delta_{X_{\phi}})}\cap [0,\frac{1}{4}-\eps]=\rm{spec(\Delta_{X})}\cap [0,\frac{1}{4}-\eps]$ when $X$ is a hyperbolic surface with punctures, which together with \cite{BBD88} shows $\lim \limits_{g\to \infty}\sup\limits_{X\in \sM_g}\lambda_1(X)=\frac{1}{4}$, which in particular shows the existence of a sequence of hyperbolic surfaces $\{X_{g}\}$ with genus $g$ going to infinity such that $\lambda_1(X_{g})$ tends to $\frac{1}{4}$.\\

Our method also yields the following estimate on the density of eigenvalues below $\frac{1}{4}$ of random surfaces for large genus, which is a also a generalization of Theorem \ref{main}.
\begin{theorem}\label{main-2}
Let $X\in \sM_g$ be a hyperbolic surface of genus $g$ and denote
\[0=\lambda_0(X)<\lambda_1(X)\leq\lambda_2(X)\cdots \leq\lambda_{s(X)}(X)\leq \frac{1}{4}\]
the collection of eigenvalues of $X$ at most $\frac{1}{4}$ counted with multiplicity. For any $\sigma \in (\frac{1}{2},1)$, set
\[N_\sigma(X)=\#\{1\leq i \leq s(X); \ \lambda_i(X)<\sigma(1-\sigma)\}.\]
Then for any $\eps>0$, we have
\[\lim \limits_{g\to \infty} \Prob\left( X\in \M_g; \ N_\sigma(X) \leq g^{3-4\sigma+\epsilon}\right)=1.\] 
\end{theorem}

\begin{rem*}
Choose $\sigma=\frac{3}{4}+\eps$, Theorem \ref{main-2} implies Theorem \ref{main}. 
\end{rem*}

\begin{rem*}
One may see similar estimates in \cite[Theorem 11.7]{Iwa02} by Iwaniec for congruence covers of the moduli surface $\H/\rm{SL(2,\Z)}$, and in \cite[Theorem 1.7]{MNP20} by Magee-Naud-Puder for random cover surfaces.
\end{rem*}

Let $X\in \sM_g$ be a hyperbolic surface of genus $g$. A simple area argument tells that the \emph{diameter} $\diam(X)$ of $X$ satisfies $\diam(X)\geq \ln(4g-2)$. Surprisingly, Mirzakhani proved in \cite[Theorem 4.10]{Mirz13} that
\[\lim \limits_{g\to \infty}\Prob\left(X\in \M_g; \ \diam(X)< 40\ln(g)\right)=1.\]
Combine Theorem \ref{main} and a recent observation of Magee \cite{Magee-20}, we extend the above bound of Mirzakhani as following.
\bt\label{mt-diam}
For any $\eps>0$, we have
\[\lim \limits_{g\to \infty}\Prob\left(X\in \M_g; \ \diam(X)<(4+\eps)\ln(g) \right)=1.\]
\et

\subsection*{Strategy on the proof of Theorem \ref{main}.} The proofs of Theorem \ref{main} and \ref{main-2} are almost the same. We just briefly introduce the idea and novelty in the proof of Theorem \ref{main}. We will use Selberg's trace formula as a tool, and then combine similar ideas in \cite{Mirz07, NWX20} to find an effective way to compute the integral of  Selberg's trace formula over moduli space $\sM_g$. In the procedure of resolving intersections of non-simple closed geodesics, which is also the most difficult part in the proof of Theorem \ref{main}, we will prove a new counting result Theorem \ref{sec-count} for filling closed geodesics to control the multiplicity occurring in the resolution procedure. More precisely, let $X\in \sM_g$ be a closed hyperbolic surface of genus $g$, we rewrite Selberg's trace formula in the following form (\eg see \eqref{sel-split}) 

\small{\bear \label{sel-split-sec}
&  \sum_{k=0}^{\infty}\widehat{\phi_T}(r_k(X))=\underbrace{(g-1)\int_{-\infty}^\infty r \widehat{\phi_T}(r)\tanh(\pi r)dr}_{\rm I}\\
&+\underbrace{\sum_{\gamma \in \mathcal P(X)}\sum_{k=2}^\infty\frac{\ell_{\gamma}(X)}{2\sinh \left(\frac{k \ell_{\gamma}(X)}{2} \right)}\phi_T(k \ell_{\gamma}(X))}_{\rm{II}} +\underbrace{\sum_{\gamma \in \mathcal{P}^{s}_{sep}(X)} \frac{\ell_{\gamma}(X)}{2\sinh \left(\frac{\ell_{\gamma}(X)}{2} \right)}\phi_T( \ell_{\gamma}(X))}_{\rm{III}}\nonumber \\
&+\underbrace{\sum_{\gamma \in \mathcal{P}^{s}_{nsep}(X)} \frac{\ell_{\gamma}(X)}{2\sinh \left(\frac{\ell_{\gamma}(X)}{2} \right)}\phi_T( \ell_{\gamma}(X))}_{\rm{IV}} +\underbrace{\sum_{\gamma \in \mathcal{P}^{ns}(X)} \frac{\ell_{\gamma}(X)}{2\sinh \left(\frac{\ell_{\gamma}(X)}{2} \right)}\phi_T( \ell_{\gamma}(X))}_{\rm{V}} \nonumber 
\eear}
where
\begin{eqnarray*}
r_k(X)\overset{\text{def}}{=}
\begin{cases}
\sqrt{\lambda_k(X)-\frac{1}{4}}, &\text{if $\lambda_k(X)>\frac{1}{4}$};\\
\textbf{i}\cdot \sqrt{-\lambda_k(X)+\frac{1}{4}}, & \text{if $\lambda_k(X)\leq\frac{1}{4}$}.
\end{cases}
\end{eqnarray*}

\noindent Here we briefly introduce the five terms on the $\mathrm{RHS}$ above. One may see Section \ref{sec-proof-1} for more details. Term-I only depends on the genus $g$; Term -II is a summation over all non-primitive closed geodesics in $X$; Term-III is a summation over all primitive simple separating closed geodesics; Term-IV is a summation over all primitive simple non-separating closed geodesics; the last one Term-V is a summation over all primitive non-simple closed geodesics. We will give some efficient upper bounds for these five terms case by case.

First one may choose a suitable even function $\phi_T(\cdot)$ as shown in \cite{MNP20} (or see Section \ref{section trace}), where $T=4\ln(g)$, such that 
\begin{enumerate}
\item[(a)] $\Supp(\phi_T)=(-T,T)$;

\item[(b)] $\widehat{\phi_T}\geq 0$ on $\R\cup \textbf{i}\R$;

\item[(c)] for any $\eps>0$ and $X\in \sM_g$ with $\lmdx \leq \frac{3}{16}-\eps$, then there exists a constant $C_\eps>0$ independent of $g$ such that 
\be
\widehat{\phi_{T}}(r_1(X))\geq C_\eps g^{1+C_\eps}  \ln(g).
\ene
\end{enumerate}

 Next we take an integral of \eqref{sel-split-sec} over $\sM_g$. Let $V_g$ be the \wep \ volume of $\sM_g$. It is easy to see that (see Proposition \ref{I})
\be
\frac{\int_{\sM_g}\rm{I} \,dX}{V_g} \prec \frac{g}{\ln(g)}.
\ene 
Split $\sM_g$ into thick and thin parts, and then use a result in \cite{Mirz13} it is not hard to show that (see Proposition \ref{II})
\be
\frac{\int_{\sM_g}\rm{II} \,dX}{V_g} \prec g \ln(g)^2.
\ene 
Applying the Integration Formula of Mirzakhani \cite{Mirz07}, one may show that (see Proposition \ref{III})
\be
\frac{\int_{\sM_g}\rm{III} \,dX}{V_g} \prec g.
\ene 
Again by applying the Integration Formula of Mirzakhani \cite{Mirz07}, one may also show that (see Proposition \ref{IV})
\be
\left|\frac{\int_{\sM_g}\rm{IV} \,dX}{V_g}- \widehat{\phi_{T}}(r_0(X))\right|=\left|\frac{\int_{\sM_g}\rm{IV} \,dX}{V_g}- \widehat{\phi_{T}}(\frac{\textbf{i}}{2})\right|\prec g\ln(g)^2.
\ene 
Term \rm{V} in \eqref{sel-split-sec} is the most difficult part to study in the proof of Theorem \ref{main}, which is also the essential part of this work. Similar as in \cite{MP19, NWX20} we will first resolve intersections of non-simple closed geodesics where we will encounter a new essential multiplicity issue. Then we will show that (see Theorem \ref{V}) for any $\eps_1>0$ there exists a constant $c(\eps_1)>0$ only depending on $\eps_1$ such that
\be  \label{sec-ns-7}
\frac{\int_{\sM_g}\rm{V} \,dX}{V_g} \prec \left(g\ln(g)^3+c_1(\eps_1)g^{1+\eps_1}\ln(g)^{67} \right).
\ene

Finally after taking an integral of \eqref{sel-split-sec} over $\sM_g$, we only keep the first two terms in the \rm{LHS} of \eqref{sel-split-sec} and drop the remaining terms in the \rm{LHS} of \eqref{sel-split-sec}, and then combine all the equations \eqref{sel-split-sec}--\eqref{sec-ns-7} above to get
\be
\quad \ \Prob\left(X\in \M_g; \ \lmdx\leq\frac{3}{16}-\eps \right) \prec \frac{g(\ln(g))^3+c_1(\eps_1)g^{1+4\eps_1}(\ln(g))^{67}  }{C_\eps \ln(g) g^{1+C_\eps}}.\nonumber
\ene
Let $g\to \infty$, then Theorem \ref{main} follows by choosing $\eps_1>0$ with $4\eps_1<C_\eps$.\\

We enclose this introduction by the following new counting result for filling closed geodesics on compact hyperbolic surfaces with non-empty geodesic boundaries, which is essential in resolving the multiplicity issue in the proof of \eqref{sec-ns-7} and also interesting by itself. The proof will be postponed until Section \ref{section-new count}.
\begin{theorem}[\bf Key Counting]\label{sec-count}
For any $\eps_1>0$ and $m=2g-2+n\geq 1$, there exists a constant $c(\eps_1,m)>0$ only depending on $m$ and $\eps_1$ such that for all $T>0$ and any compact hyperbolic surface $X$ of genus $g$ with $n$ boundary simple closed geodesics, we have
\begin{equation*}
\#_f(X,T)\leq c(\eps_1,m) \cdot e^{T-\frac{1-\eps_1}{2}\ell(\partial X)}.
\end{equation*}
Where $\#_f(X,T)$ is the number of filling closed geodesics in $X$ of length $\leq T$ and $\ell(\partial X)$ is the total length of the boundary closed geodesics of $X$.
\end{theorem}

\begin{rem*}
A filling closed geodesic in $X$ always has length greater than  $\frac{\ell(\partial X)}{2}$. The importance of Theorem \ref{sec-count} above is that the boundary length $\ell(\partial X)$ is allowed to depend on $T$. In particular, if $\ell(\partial X)$ is closed to $2T$, then as $T\to \infty$, the growth rate of the number $\#_f(X,T)$ is no more than $e^{\eps_1 T }$, which is much less than the general bound $e^T$. To our best knowledge, Theorem \ref{sec-count} is new even for the case $X \cong S_{0,3}$ is a pair of pants in which a non-trivial closed geodesic is always filling. 
\end{rem*}

\subsection*{Notations.} For any two nonnegative functions $f$ and $h$ (may be of multi-variables), we say $f\prec h$ if there exists a uniform constant $C>0$ such that $f\leq Ch$. And we say $f\asymp h$ if $f\prec h$ and $h\prec f$. For any $r>0$, we denote by $[r]$ the largest integer part of $r$.

\subsection*{Plan of the paper.} In Sections \ref{section preliminaries}, \ref{section count}, \ref{section wp volume} and \ref{section trace}, we review the backgrounds, introduce some notations, and prove several lemmas. We prove the relative easy parts in the proofs of Theorem \ref{main} and \ref{main-2} in Section \ref{sec-proof-1}, i.e., the upper bounds for integrals of Term \rm{I}---Term \rm{IV} over the moduli space $\sM_g$. Then we prove the difficult part, i.e., the upper bound for $\int_{\sM_g} \rm{V}\,dX$, and then complete the proofs of Theorem \ref{main}, \ref{main-2} and \ref{mt-diam} in Section \ref{sec-main-1}, assuming the essential new counting result Theorem \ref{sec-count} which is proved in Section \ref{section-new count}. 

\subsection*{Acknowledgements.}The authors would like to thank Yang Shen for helpful discussions on the key counting result Theorem \ref{sec-count}. They also would like to thank Prof. S. T. Yau for his interests on this work. Both authors are supported by the NSFC grant No. $12171263$, and the first named author is also partially supported by a grant from Tsinghua University. We are also grateful to the referees for helpful comments and suggestions which improve this article.

\setcounter{tocdepth}{1}
\tableofcontents

%%%%%%%%%%%%%%%%%%%%%%%%%%%%%%%%%%%%%%%%%%%%%%%%%%%%%%%%%%%%%%%%%%%%%%
\section{Preliminaries}\label{section preliminaries}
In this section, we set our notations and review the relevant background material about moduli space of Riemann surfaces, \wep \ metric and Mirzakhani's Integration Formula.
%%%%%%%%%%%%%%%%%%%%%%%%%%%

\subsection{Riemann surfaces.} \label{sec:wp background}
We denote by $S_{g,n}$ an oriented surface of genus $g$ with $n$ punctures or boundaries where $2g+n\geq 3$. Let $\sT_{g,n}$ be the Teichm\"uller space of surfaces of genus $g$ with $n$ punctures or boundaries, which we consider as the equivalence classes under the action of the group $\Diff_0(S_{g,n})$ of diffeomorphisms isotopic to the identity of the space of hyperbolic surfaces $X=(S_{g,n},\sigma(z)|dz|^2)$. The moduli space of Riemann surfaces $\sM_{g,n}$ is defined as $\sT_{g,n}/\Mod_{g,n}$ where $\Mod_{g,n}\overset{\text{def}}{=}\Diff^+(S_{g,n})/\Diff_0(S_{g,n})$ is the so-called mapping class group of $S_{g,n}$. If $n=0$, we write $\sM_g = \sM_{g,0}$ for simplicity. Given $\ve{L}=(L_1,\cdots, L_n)\in\R_{\geq0}^n$, the weighted Teichm\"uller space $\T_{g,n}(\ve{L})$ parametrizes hyperbolic surfaces $X$ marked by $S_{g,n}$ such that for each $i=1,\cdots,n$,
\begin{itemize}
	\item if $L_i=0$, the $i^\text{th}$ puncture of  $X$ is a cusp;
	\item if $L_i>0$, one can attach a circle to the $i^\text{th}$ puncture of $X$ to form a geodesic boundary loop of length $L_i$. 
\end{itemize}
The weighted moduli space $\M_{g,n}(\ve{L})\overset{\text{def}}{=}\T_{g,n}(\ve{L})/\Mod_{g,n}$ then parametrizes unmarked such surfaces.
%%%%%%%%%%%%%%%%%%%%%%%%%
\subsection{The Weil-Petersson metric}
Associated to a pants decomposition of $S_{g,n}$, the \emph{Fenchel-Nielsen coordinates}, given by $X \mapsto (\ell_{\alpha_i}(X),\tau_{\alpha_i}(X))_{i=1}^{3g-3+n}$, are global coordinates for the Teichm\"uller space $\sT_{g,n}$ of $S_{g,n}$. Where $\{\alpha_i\}_{i=1}^{3g-3+n}$ are disjoint simple closed geodesics, $\ell_{\alpha_i}(X)$ is the length of $\alpha_i$ on $X$ and $\tau_{\alpha_i}(X)$ is the twist along $\alpha_i$ (measured in length). Wolpert in \cite{Wolpert82} showed that the \wep \ symplectic structure has a natural form in Fenchel-Nielsen coordinates:
\bt[Wolpert]\label{wol-wp}
The \wep \ symplectic form $\omega_{\WP}$ on $\sT_{g,n}$ is given by
\[\omega_{\WP}=\sum_{i=1}^{3g-3+n}d\ell_{\alpha_i}\wedge d\tau_{\alpha_i}.\]
\et

We mainly work with the \emph{Weil-Petersson volume form}
$$
\dvol_{\WP}\overset{\text{def}}{=}\tfrac{1}{(3g-3+n)!}\underbrace{\omega_{\WP}\wedge\cdots\wedge\omega_{\WP}}_{\text{$3g-3+n$ copies}}~.
$$
It is a mapping class group invariant measure on $\T_{g,n}$, hence is the lift of a measure on $\M_{g,n}$, which we also denote by $\dvol_{\WP}$. The total volume of $\M_{g,n}$ is finite and we denote it by $V_{g,n}$. The Weil-Petersson volume form is also well-defined on the weighted moduli space $\M_{g,n}(\ve{L})$ and its total volume, denoted by $V_{g,n}(\ve{L})$, is finite.  

Following \cite{Mirz13}, we view a quantity $f:\M_g\to\R$ as a random variable on $\M_g$ with respect to the probability measure $\Prob$ defined by normalizing $\dvol_{\WP}$. Namely,
$$
\Prob(\mathcal{A})\overset{\text{def}}{=}\frac{1}{V_g}\int_{\M_g}\mathbf{1}_{\mathcal{A}}dX
$$
where $\mathcal{A}\subset\M_g$ is any Borel subset, $\mathbf{1}_\mathcal{A}:\M_g\to\{0,1\}$ is its characteristic function, and where $dX$ is short for $\dvol_{\WP}(X)$. One may see the book \cite{Wolpert-book} for recent developments on \wep \ geometry, and see the recent survey \cite{Wright-tour} for works of Mirzakhani including her coworkers on random surfaces in the \wep \ model. 

In this paper, we view the first eigenvalue function as a random variable on $\sM_g$, and study its asymptotic behavior as $g\to \infty$. One may also see \cite{DGZZ20-multi, GMST19, GPY11, Monk20, MP19, NWX20, PWX20} for related interesting topics.

%%%%%%%%%%%
\subsection{Mirzakhani's Integration Formula}

In this subsection we recall an integration formula in \cite{Mirz07, Mirz13}, which is essential in the study of random surfaces in the \wep \ model.

Given any non-peripheral closed curve $\gamma$ on a topological surface $S_{g,n}$ and $X\in\T_{g,n}$, we denote by $\ell_\gamma(X)$ the hyperbolic length of the unique closed geodesic in the homotopy class $\gamma$ on $X$. We also write $\ell(\gamma)$ for simplicity if we do not need to emphasize the surface $X$. Let $\Gamma=(\gamma_1,\cdots,\gamma_k)$ be an ordered k-tuple where the $\gamma_i$'s are distinct disjoint homotopy classes of nontrivial, non-peripheral, simple closed curves on $S_{g,n}$. We consider the orbit containing $\Gamma$ under $\Mod_{g,n}$ action
\begin{equation*}
\mathcal O_{\Gamma} = \{(h\cdot\gamma_1,\cdots,h\cdot\gamma_k) ; \ h\in\Mod\nolimits_{g,n}\}.
\end{equation*}
Given a function $F:\R^k_{\geq0} \rightarrow \R_{\geq0}$ one may define a function on $\M_{g,n}$
\begin{eqnarray*}
F^\Gamma:\M_{g,n} &\rightarrow& \R \\
X &\mapsto& \sum_{(\alpha_1,\cdots,\alpha_k)\in \mathcal O_\Gamma} F(\ell_{\alpha_1}(X),\cdots,\ell_{\alpha_k}(X)).
\end{eqnarray*}

\noindent Assume $S_{g,n}-\cup\gamma_j = \cup_{i=1}^s S_{g_i,n_i}$. For any given $\boldsymbol{x}=(x_1,\cdots,x_k)\in \R^k_{\geq0}$, we consider the moduli space $\M(S_{g,n}(\Gamma); \ell_{\Gamma}=\boldsymbol{x})$ of hyperbolic Riemann surfaces (possibly disconnected) homeomorphic to $S_{g,n}-\cup\gamma_j$ with $\ell_{\gamma_i^1} = \ell_{\gamma_i^2} =x_i$ for $i=1,\cdots,k$, where $\gamma_i^1$ and $\gamma_i^2$ are the two boundary components of $S_{g,n}-\cup\gamma_j$ given by cutting along $\gamma_i$. We consider the volume
\begin{equation*}
V_{g,n}(\Gamma,\boldsymbol{x}) = \Vol\nolimits_{\rm WP}\big(\M(S_{g,n}(\Gamma); \ell_{\Gamma}=\boldsymbol{x})\big).
\end{equation*}
In general
\begin{equation*}
V_{g,n}(\Gamma,\boldsymbol{x}) = \prod_{i=1}^s V_{g_i,n_i}(\boldsymbol{x}^{(i)})
\end{equation*}
where $\boldsymbol{x}^{(i)}$ is the list of those coordinates $x_j$ of $\boldsymbol{x}$ such that $\gamma_j$ is a boundary component of $S_{g_i,n_i}$. And $V_{g_i,n_i}(\boldsymbol{x}^{(i)})$ is the Weil-Petersson volume of the moduli space $\M_{g_i,n_i}(\boldsymbol{x}^{(i)})$. Mirzakhani used Theorem \ref{wol-wp} of Wolpert to get the following integration formula. One may refer to \cite[Theorem 7.1]{Mirz07} or \cite[Theorem 2.2]{MP19} or \cite[Theorem 4.1]{Wright-tour}.

\begin{theorem}\label{Mirz int formula}
For any $\Gamma=(\gamma_1,\cdots,\gamma_k)$, the integral of $F^\Gamma$ over $\M_{g,n}$ with respect to Weil-Petersson metric is given by
\begin{equation*}
\int_{\M_{g,n}} F^\Gamma(X)dX =
C_\Gamma\int_{\R^k_{\geq0}} F(x_1,\cdots,x_k)V_{g,n}(\Gamma,\boldsymbol{x}) \boldsymbol{x}\cdot d\boldsymbol{x}
\end{equation*}
where $\boldsymbol{x}\cdot d\boldsymbol{x} = x_1\cdots x_k dx_1\wedge\cdots\wedge dx_k$ and the constant $C_\Gamma \in (0,1]$ only depends on $\Gamma$. Moreover, $C_\Gamma=\frac{1}{2}$ if $g>2$ and $\Gamma$ is a simple non-separating closed curve. 
\end{theorem}

\begin{rem*}
In \cite[Section 4]{Wright-tour} it has a detailed argument for $C_\Gamma=\frac{1}{2}$ when $g>2$ and $\Gamma$ is a simple non-separating closed curve. 
\end{rem*}

\begin{rem*}
Given an unordered multi-curve $\gamma=\sum_{i=1}^k c_i \gamma_i$ where $\gamma_i's$ are distinct disjoint homotopy classes of nontrivial, non-peripheral, simple closed curves on $S_{g,n}$, when $F$ is a symmetric function, we can define
\begin{eqnarray*}
F_\gamma:\M_{g,n} &\rightarrow& \R \\
X &\mapsto& \sum_{\sum_{i=1}^k c_i\alpha_i \in \Mod_{g,n}\cdot \gamma} F(c_1\ell_{\alpha_1}(X),\cdots,c_k\ell_{\alpha_k}(X)).
\end{eqnarray*}

\noindent It is easy to check that
\begin{equation*}
F^\Gamma(X) = |\Sym(\gamma)| \cdot F_\gamma(X)
\end{equation*}
where $\Gamma=(c_1\gamma_1,\cdots,c_k\gamma_k)$ and $\Sym(\gamma)$ is the symmetry group of $\gamma$ defined by
\begin{equation*}
\Sym(\gamma) = \mathop{\rm Stab}(\gamma) / \cap_{i=1}^k \mathop{\rm Stab}(\gamma_i).
\end{equation*}
%\noindent Actually we consider the integration of $F_\gamma$ for most times in this paper.
\end{rem*}

%%%%%%%%%%%%%%%%%%%%%%%%%%%%%%%%%%%%%%%%%%%%%%%%%%%%%%%%%%%%%%%%%%%%%%%%%%%%%%%%%%%%%%%%%%%%%%%%%
%%%%%%%%%%%%%%%%%%%%%%%%%%%%%%%%%%%%%%%%%%%%%%%%%%%%%%%%%%%%%%%%%%%%%%%%%%%%%%%%
\section{Counting closed geodesics}\label{section count}
In this section we first briefly introduce a geodesic subsurface for a non-simple closed geodesic, and then provide several useful counting results for closed geodesics.

First we recall the following construction as in \cite{MP19, NWX20}. 

\begin{con*}
Let $X\in \sM_g$ be a hyperbolic surface and $\gamma' \subset X$ be a non-simple closed geodesic. Consider the $\eps$-neighborhood $\mathcal{N}_{\eps}(\gamma')$ of $\gamma'$ where $\eps>0$ is small enough such that $\mathcal{N}_{\eps}(\gamma')$ is homotopic to $\gamma'$ in $X$. Now we obtain a \emph{subsurface $X(\gamma')$ of geodesic boundary} by deforming each of its boundary components $\xi\subset\partial(\mathcal{N}_{\eps}(\gamma'))$ as follows: 
\begin{itemize}
	\item if $\xi$ is homotopically trivial, we fill the disc bounded by $\xi$ into $\mathcal{N}_{\eps}(\gamma')$;
	\item otherwise, we deform $\mathcal{N}_{\eps}(\gamma')$ by shrinking $\xi$ to the unique simple closed geodesic homotopic to it. 
\end{itemize}
We remark here that if two components of $\partial \mathcal{N}_{\eps}(\gamma')$ deforms to the same simple closed geodesic, we do not glue them together, \ie, one may view $X(\gamma')$ as an open subsurface of $X$ (\eg see Figure \ref{X-gamma}).
\end{con*}

\begin{figure}[ht]
	\centering
	\includegraphics[width=12cm]{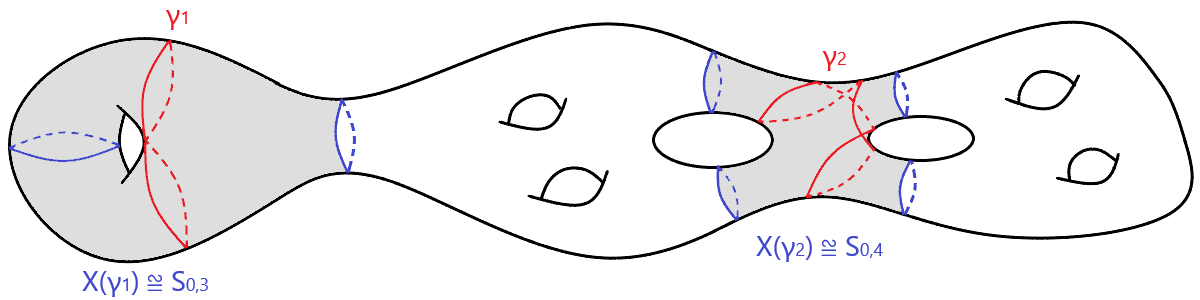}
	\caption{examples for $X(\gamma')$}
	\label{X-gamma}
\end{figure}

For a surface $X$ with possibly non-empty boundary, recall that a closed curve $\gamma \subset X$ is \emph{filling} if the complement $X\setminus \gamma$ of $\gamma$ in $X$ is a disjoint union of disks and cylinders such that each cylinder is homotopic to a boundary component of $X$. 

The following result is proved in \cite{NWX20}.

\begin{proposition}\cite[Proposition 47]{NWX20}\label{ns-sub}
Let $X\in \sM_g$ and $\gamma' \subset X$ be a non-simple closed geodesic. Then the connected subsurface $X(\gamma')$ of $X$ constructed above satisfies that,
\begin{itemize}
\item[(1)]  $\gamma' \subset X(\gamma')$ is filling;  

\item[(2)]  the possibly empty boundary $\partial X(\gamma')$ of $X(\gamma')$ consists of simple closed multi-geodesics with
\[\ell\left(\partial X(\gamma')\right) \leq 2\ell_{\gamma'}(X);\]

\item[(3)] the area $\area(X(\gamma')) \leq 4\ell_{\gamma'}(X)$. In particular, if $\ell_{\gamma'}(X) \prec \ln(g)$, then for large enough $g>1$, $X(\gamma')$ is a proper subsurface of $X$.
\end{itemize}
\end{proposition}
\bp
We only outline a proof here for completeness. One may see \cite{NWX20} for more details.

For $(1)$: by construction the subsurface $X(\gamma')$ is freely homotopic to $\mathcal{N}_{\eps}(\gamma')$ in $X$. So it is also freely homotopic to $\gamma'$ in $X$. Since $\gamma'$ is the unique closed geodesic representing the free homotopy class $\gamma'$ and $X(\gamma')\subset X$ is a subsurface of geodesic boundary, we have \[\gamma'\subset X(\gamma').\]
By construction we know that $\gamma'$ is filling in $X(\gamma')$.

For (2): by construction clearly we have
\[\ell(\partial X(\gamma'))\leq 2\ell_{\gamma'}(X).\]

For (3): by construction we know that the complement $X(\gamma')\setminus \gamma'= (\sqcup D_i) \sqcup (\sqcup C_j)$ where the subsets are setwisely disjoint, the $D_i's$ are disjoint discs and the $C_j's$ are disjoint cylinders. By elementary Isoperimetric Inequality (\eg see \cite{Buser10, WX18}) we know that
\[\area(D_i)\leq \ell(\partial D_i) \quad \text{and} \quad \area(C_j)\leq \ell(\partial C_j).\]
Thus, we have
\begin{eqnarray*}
\area(X(\gamma'))&=&\sum(\area(D_i))+\sum(\area(C_j))\\
&\leq & \sum (\ell(\partial D_i))+\sum (\ell(\partial C_j))\\
&\leq & 4\ell_{\gamma'}(X).
\end{eqnarray*}
If $\ell_{\gamma'}(X) \prec \ln(g)$, we have $\area(X(\gamma'))\prec \ln(g)$. Then the conclusion clearly follows because $\area(X)=4\pi(g-1)$ by Gauss-Bonnet.
 
The proof is complete.
\ep

\begin{rem*}
It is not hard to see that the inequality $\area(X(\gamma')) \leq 4\ell_{\gamma'}(X)$ can be improved to be $\area(X(\gamma')) \leq 2\ell_{\gamma'}(X)$. For our purpose the constant $4$ in this bound is enough in this paper.
\end{rem*}

\begin{def*}
For any $L>0$, we define 
\begin{enumerate}
\item $\#(X,L)$ is the number of closed geodesics of length $\leq L$ on $X$;

\item $\#_f(X,L)$ is the number of filling closed geodesics of length $\leq L$ on $X$;

\item $\#_0(X,L)$ is the number of closed geodesics of length $\leq L$ on $X$ which are not iterates of any closed geodesic of length $\leq 2 \arcsinh 1$.
\end{enumerate}

\end{def*}

The map $\gamma' \mapsto X(\gamma')$ in the construction above is infinite-to-one. Indeed for any connected subsurface $Y\subset X$ of geodesic boundary, $X(\gamma_1)=X(\gamma_2)=Y$ for any two filling curves $\gamma_1, \gamma_2\subset Y$. However, the multiplicity of the map $\gamma' \mapsto X(\gamma')$ is always bounded if restricting the length of $\gamma'$ to be bounded. That is, $\#_f(Y,L)<\infty$ for any $L>0$. In this paper we prove the following counting result for compact hyperbolic surfaces of geodesic boundaries, which is essential in the proofs of Theorem \ref{main} and \ref{main-2} when dealing with primitive non-simple closed geodesics. Here a closed geodesic is called \emph{primitive} if it is not an iterate of any other closed geodesic at least twice. Since the proof is technical, we postpone the proof until a single section \ref{section-new count}.

\begin{theorem}[=Theorem \ref{sec-count}]\label{count}
For any $\eps_1>0$ and $m=2g-2+n\geq 1$, there exists a constant $c(\eps_1,m)$ only depending on $m$ and $\eps_1$ such that for any hyperbolic surface $X\in \T_{g,n}(x_1,...,x_n)$ we have
\begin{equation*}
\#_f(X,T)\leq c(\eps_1,m) \cdot e^{T-\frac{1-\eps_1}{2}\sum_{i=1}^n x_i}.
\end{equation*}
\end{theorem}

\begin{rem*}
Lalley in \cite{Lall89} showed for a compact hyperbolic surface $X$ of geodesic boundary, the number $\#(X,L)\sim \frac{1}{2} \frac{e^{\delta L}}{\delta L}$ as $L\to \infty$ where $\delta$ is the Hausdorff dimension of the limit set of the Fuchsian group of $X$ in the boundary of the upper half plane (the factor $\frac{1}{2}$ disappears if counting oriented closed geodesics). The upper bound in Theorem \ref{count} contains explicit information on the boundary length of $X$ and is uniform in $X$, which will play an essential role in the proofs of Theorem \ref{main} and \ref{main-2}.
\end{rem*}

Now we conclude this section by several general and soft counting results. First we recall the following bound (\eg see \cite[Lemma 6.6.4]{Buser10}).
\begin{lemma}\label{count-1}
For any $X\in \sM_g$ and $L>0$, we have
\[\#_0(X,L)\leq (g-1)e^{L+6}.\]
\end{lemma}

The following result is a direct consequence of Lemma \ref{count-1}.
\begin{lemma}\label{count-3}
Let $X$ be a compact hyperbolic surface of non-empty geodesic boundary. Then
\[\#_f(X,L)\leq \frac{\area(X)}{4\pi}e^{L+6}.\]
\end{lemma}
\bp
We first double two $X$'s to get a closed hyperbolic surface $2X$. Since $\area(2X)=2\area(X)$, by Gauss-Bonnet the genus of $2X$ is equal to $\frac{\area(X)}{2\pi}+1$. By symmetry, each closed geodesic counted in $\#_f(X,L)$ gives two curves counted in $\#_0(2X,L)$. By the Collar Lemma (\eg see \cite[Theorem 4.1.6]{Buser10}) it is known that a closed geodesic in $2X$ of length $\leq 2 \arcsinh 1$ is always simple. So each filling curve in $X$ is clearly not an iterate of any closed geodesic in $2X$ of length $\leq 2 \arcsinh 1$. Then it follows by Lemma \ref{count-1} that
\beqar
2\#_f(X,L)&\leq& \#_0(2X,L)\\
&\leq & \frac{\area(X)}{2\pi} e^{L+6}
\eeqar
which completes the proof.
\ep

Recall that
\[\sM_g^{\geq 1}\overset{\text{def}}{=}\{X\in \sM_g; \ell_{sys}(X)\geq 1\}\]
where $\ell_{sys}(X)$ is the length of shortest closed geodesic in $X$. Another direct consequence of Lemma \ref{count-1} is as follows which will be applied later to bound Term-\rm{II}.
\begin{lemma}\label{count-2}
For any $X\in \sM_g^{\geq 1}$ and $L>0$, we have
\[\#(X,L)\leq 2(g-1)e^{L+6}.\]
\end{lemma}
\bp
Let $\#_1(X,L)$ be the number of closed geodesics of length $\leq L$ on $X$ which are iterates of closed geodesics of length $\leq 2 \arcsinh 1$. First by the Collar Lemma (\eg see \cite[Theorem 4.1.6]{Buser10}) we know that there are at most $(3g-3)$ simple closed geodesics of length $\leq 2 \arcsinh 1\sim 1.7627$. Moreover, they are mutually disjoint, and we denote them by $\Gamma$. Recall $X\in \sM_g^{\geq 1}$, so $\ell_{sys}(X)\geq 1$. Then each closed geodesic counted in $\#_1(X,L)$ is an iterate of some curve in $\Gamma$ at most $L$ times. Thus we have
\[\#_1(X,L)\leq (3g-3)(L+1)\]
which together with Lemma \ref{count-1} imply that
\beqar
\#(X,L)&=&\#_0(X,L)+\#_1(X,L)\\
&\leq &(g-1)e^{L+6}+(3g-3)(L+1) \nonumber\\
&\leq & 2(g-1)e^{L+6}\nonumber
\eeqar
completing the proof.
\ep

%%%%%%%%%%%%%%%%%%%%%%%%%%%%%%%%%%%%%%%%%%%%%%%%%%%%%%

\section{Weil-Petersson volume}\label{section wp volume}

In this section we list some results on \wep \ volumes of moduli spaces which will be applied later in the proofs of Theorem \ref{main} and \ref{main-2}. All of them are already known results and presented in \cite{NWX20}. We denote $V_{g,n}(x_1,\cdots,x_n)$ to be the \wep \ volume of $\M_{g,n}(x_1,\cdots,x_n)$ and $V_{g,n}= V_{g,n}(0,\cdots,0)$. 

First we recall several results of Mirzakhani and her coauthors.

\begin{theorem} \cite[Theorem 1.1]{Mirz07} \label{Mirz vol lemma 0}
The volume $V_{g,n}(x_1,\cdots,x_n)$ is a polynomial in $x_1^2,\cdots,x_n^2$ with degree $3g-3+n$. Namely we have
\begin{equation*}
V_{g,n}(x_1,\cdots,x_n) = \sum_{\alpha; |\alpha|\leq 3g-3+n} C_\alpha \cdot x^{2\alpha}
\end{equation*}
where $C_\alpha>0$ lies in $\pi^{6g-6+2n-|2\alpha|} \cdot \mathbb Q$. Here $\alpha=(\alpha_1,\cdots,\alpha_n)$ is a multi-index and $|\alpha|=\alpha_1+\cdots+\alpha_n$, $x^{2\alpha}= x_1^{2\alpha_1}\cdots x_n^{2\alpha_n}$.
\end{theorem}

\begin{lemma}\label{Mirz vol lemma 1}
\begin{enumerate}

  \item \cite[Lemma 3.2]{Mirz13}
\begin{equation*}
V_{g,n} \leq V_{g,n}(x_1,\cdots,x_n) \leq e^{\frac{x_1+\cdots+x_n}{2}} V_{g,n}.
\end{equation*}

  \item \cite[Theorem 3.5]{Mirz13}
For fixed $n\geq 0$, as $g\rightarrow \infty$ we have
\begin{equation*}
\frac{V_{g,n}}{V_{g-1,n+2}} = 1 + O(\frac{1}{g}).
\end{equation*}
Where the implied constants are related to $n$ and independent of $g$.
\end{enumerate}

\end{lemma}
\begin{rem*}
For Part $(2)$, one may also see the following Theorem \ref{MZ vol thm} of Mirzakhani-Zograf.
\end{rem*}

\begin{lemma}\cite[Corollary 3.7]{Mirz13} \label{Mirz vol lemma 2}
For fixed $b,k,r\geq 0$ and $C<C_0= 2\ln 2$,
\begin{equation*}
\sum_{\begin{array}{c}
        g_1+g_2=g+1-k  \\
        r+1\leq g_1\leq g_2
      \end{array}}
e^{Cg_1} \cdot g_1^b \cdot V_{g_1,k} \cdot V_{g_2,k} \asymp \frac{V_g}{g^{2r+k}}
\end{equation*}
as $g\rightarrow\infty$. The implied constants are related to $b,k,r,C$ and independent of $g$. 

\end{lemma}

The following several useful bounds for \wep \ volumes are proved in \cite{NWX20}. And their proofs highly rely on works in \cite{Mirz13} and the following asymptotic property of $V_{g,n}$ which dues to Mirzakhani-Zograf.

\begin{theorem}\cite[Theorem 1.2]{MZ15} \label{MZ vol thm}
There exists a universal constant $\alpha>0$ such that for any given $n\geq0$,
\begin{equation*}
V_{g,n} = \alpha \frac{1}{\sqrt{g}} (2g-3+n)! (4\pi^2)^{2g-3+n} \big(1+O(\frac{1}{g}) \big)
\end{equation*}
as $g\rightarrow\infty$. The implied constant is related to $n$ and independent of $g$.
\end{theorem}

The first one is as following which is motivated by \cite[Proposition 3.1]{MP19} where the error term in the lower bound is different.
\begin{lemma} \cite[Lemma 20]{NWX20} \label{MP vol lemma}
Let $g,n\geq 1$ and $x_1,\cdots,x_n\geq 0$, then there exists a constant $c=c(n)>0$ independent of $g,x_1,\cdots,x_n$ such that
\begin{equation*}
\prod_{i=1}^n \frac{\sinh(x_i/2)}{x_i/2} \big(1- c(n)\frac{\sum_{i=1}^n x_i^2}{g}\big)
\leq \frac{V_{g,n} (x_1,\cdots,x_n)}{ V_{g,n}}
\leq \prod_{i=1}^n \frac{\sinh(x_i/2)}{x_i/2}.
\end{equation*}
\end{lemma}

\begin{rem*} In the lemma above,
\begin{enumerate}
\item for the lower bound, the $x_i's$ may be related to $g$ but $n$ is independent of $g$ as $g\rightarrow\infty$;
\item for the upper bound, both the $x_i's$ and $n$ may be related to $g$ as $g\rightarrow\infty$.
\end{enumerate}
\end{rem*}

As in \cite{NWX20}, for $r\geq1$ one may define 
$$
W_{r}\overset{\text{def}}{=}
\begin{cases}
V_{\frac{r}{2}+1}&\text{if $r$ is even},\\[5pt]
V_{\frac{r+1}{2},1}&\text{if $r$ is odd}.
\end{cases}
$$

The proof of the following result relies on Lemma \ref{Mirz vol lemma 1} and Lemma \ref{Mirz vol lemma 2}.
\begin{lemma}\cite[Lemma 21]{NWX20}\label{Wr-prop}
	\begin{enumerate}
		\item For any $g,n\geq 0$, we have
		$$V_{g,n}  \leq c \cdot W_{2g-2+n}$$
		for some universal constant $c>0$.
		\item For any $r\geq1$ and $m_0\leq \frac{1}{2}r$, we have
		$$\sum_{m=m_0}^{[\frac{r}{2}]} W_m W_{r-m} \leq c(m_0) \frac{1}{r^{m_0}}W_r$$
		for some constant $c(m_0)>0$ only depending on $m_0$.
	\end{enumerate}
\end{lemma}

The proof of the following lemma relies on Theorem \ref{MZ vol thm}. Which is also a generalization of \cite[Lemma 3.2]{MP19} and \cite[Lemma 6.3]{GMST19}. Here we allow the $n_i's$ and $q$ depend on $g$ as $g\to \infty$.

\begin{lemma}\cite[Lemma 22]{NWX20}\label{sum vol lemma}
Assume $q\geq 1$, $n_1,\cdots,n_q\geq 0$, $r\geq2$. Then there exists two universal constants $c,D>0$ such that
\begin{equation*}
\sum_{\{g_i\}} V_{g_1,n_1}\cdots V_{g_q,n_q} \leq c \big(\frac{D}{r}\big)^{q-1} W_r
\end{equation*}
where the sum is taken over all $\{g_i\}_{i=1}^q \sbs \N$ such that $2g_i-2+n_i \geq 1$ for all $i=1,\cdots,q$, and $\sum_{i=1}^q (2g_i-2+n_i) = r$. 
\end{lemma}

We conclude this section by the following useful property whose proof relies on Lemma \ref{Mirz vol lemma 1} and Lemma \ref{sum vol lemma}.
\begin{proposition}\cite[Lemma 23]{NWX20}\label{1 over gm} 
	Given $m\geq 1$, for any $g\geq m+1$, $q\geq 1$, $n_1,...,n_q\geq 1$, there exists a constant $c(m)>0$ only depending on $m$ such that
	\begin{equation*}
	\sum_{\{g_i\}} V_{g_1,n_1}\cdots V_{g_q,n_q} \leq c(m)\frac{1}{g^m} V_g
	\end{equation*}
    where the sum is taken over all $\{g_i\}_{i=1}^q \sbs \N$ such that $2g_i-2+n_i \geq 1$ for all $i=1,\cdots,q$, and $\sum_{i=1}^q (2g_i-2+n_i) = 2g-2-m$.
\end{proposition}

%%%%%%%%%%%%%%%%%%%%%%%%%%%%%%%%%%%%%%%%%%%%%%%%%%%%%%%%%%%%%%%%%%%%%%%%%%%%%%%%%%

\section{Selberg's trace formula} \label{section trace}
In this section we describe the Selberg trace formula for closed hyperbolic surfaces, which is a main tool in the proofs of Theorem \ref{main} and \ref{main-2}. 

Let $C_c^\infty(\R)$ denote the set of all smooth functions on $\R$ with compact support. Given a function $\phi \in C_c^\infty(\R)$, its Fourier transform is defined as
\[\widehat{\phi}(z)\overset{\text{def}}{=}\int_{-\infty}^{\infty}\phi(x)e^{-ixz}dx\]
for any $z\in \mathbb{C}$. For $\phi \in  C_c^\infty(\R)$, the above integral is an entire function over $\mathbb C$. In particular, it converges for any $z\in \mathbb{C}$.

Recall that a closed geodesic is primitive if it is not an iterate of any other closed geodesic at least twice. For any hyperbolic surface $X$, we let $\mathcal{P}(X)$ denote the set of all oriented primitive closed geodesics on $X$. Now we recall Selberg's trace formula in the form of \cite[Theorem 5.6]{B-book} or \cite[Theorem 9.5.3]{Buser10}. One may also see \eg \cite{Selb56, Hej76} for more details.
\bt[Selberg's trace formula] \label{sel}
Let $X$ be a closed hyperbolic surface of genus $g$ and let 
\[0=\lambda_0(X)<\lambda_1(X)\leq \lambda_2(X)\leq \cdots \to \infty\]
denote the spectrum of the Laplacian on $X$. For $k\in \Z^{\geq 0}$, let 
\begin{eqnarray*}
r_k(X)\overset{\text{def}}{=}
\begin{cases}
\sqrt{\lambda_k(X)-\frac{1}{4}}, &\text{if $\lambda_k(X)>\frac{1}{4}$};\\
\textbf{i}\cdot \sqrt{-\lambda_k(X)+\frac{1}{4}}, & \text{if $\lambda_k(X)\leq\frac{1}{4}$}.
\end{cases}
\end{eqnarray*} 
Then for any even function $\phi \in C_c^\infty(\R)$ we have
\begin{eqnarray*}
\sum_{k=0}^{\infty}\widehat{\phi}(r_k(X))=&&(g-1)\int_{-\infty}^\infty r \widehat{\phi}(r)\tanh(\pi r)dr\\
&&+\sum_{\gamma \in \mathcal P(X)}\sum_{k=1}^\infty\frac{\ell_{\gamma}(X)}{2\sinh \left(\frac{k \ell_{\gamma}(X)}{2} \right)}\phi(k \ell_{\gamma}(X)).
\end{eqnarray*}    
\noindent Both sides of the formula above are absolutely convergent.
\et 

\noindent \textbf{Choice of $\phi$.} In this paper, we apply the same function in \cite{MNP20} to Selberg's trace formula. For completeness, we briefly introduce such a function. One may see \cite[Section 2]{MNP20} for more details.

First one may let $\psi:\R \to \R^{\geq 0}$ be a smooth and even function whose support is exactly $(-\frac{1}{2},\frac{1}{2})$. Then we define
\be
\phi_0(x)\overset{\text{def}}{=}\int_{\R}\psi(x-t)\psi(t)dt.\nonumber
\ene

It is not hard to see that
\begin{lemma}\label{phi0}
The function $\phi_0$ satisfies that
\begin{enumerate}
\item $\phi_0$ is non-negative and even.  
\item $\Supp(\phi_0)=(-1,1)$.
\item The Fourier transform $\widehat{\phi_0}$ satisfies $\widehat{\phi_0}(\xi)\geq 0$ for all $\xi \in \R \cup \textbf{i}\R$.
\end{enumerate}
\end{lemma}

Now for any $T>0$ we define
\be
\phi_T(x)\overset{\text{def}}{=} \phi_0\left(\frac{x}{T}\right).  \nonumber
\ene

The following property \cite[Lemma 2.4]{MNP20} will be applied later. 
\bl\label{1+eps}
For any small enough $\eps>0$, then there exists a constant $C_\eps>0$ depending on $\eps$ and $\phi_0$ such that for $t\geq 0$,
\[\widehat{\phi_{4\ln(g)}}(\textbf{i} t)\geq C_\eps g^{4(1-\eps)t}  \ln(g).\]
In particular, for any hyperbolic surface $X\in \sM_g$ with $\lmdx \leq \frac{3}{16}-\eps$ we have
\[\widehat{\phi_{4\ln(g)}}(r_1(X))\geq C_\eps g^{1+C_\eps}  \ln(g)\]
where $r_1(X)=\rm{\textbf{i}}\cdot \sqrt{-\lambda_1(X)+\frac{1}{4}}$ in Selberg's trace formula. 
\el

\bp
We outline a proof here for completeness. Since $\phi_0\geq 0$ and $\Supp(\phi_0)=(-1,1)$, we have that for any $\eps>0$ near $0$,
\begin{eqnarray*}
\widehat{\phi_{4\ln(g)}}(\textbf{i} t)&=&4\ln(g)\int_\R \phi_0(x)e^{\left(4\ln(g) t \cdot x\right)}dx \\
&\geq & 4\ln(g)\int_{1-\eps}^1 \phi_0(x) e^{\left(4\ln(g) t \cdot x\right) }dx \nonumber\\
&\geq & B_\eps  g^{4(1-\eps)t}  \ln(g) \nonumber
\end{eqnarray*}
where $B_{\eps}=4\int_{1-\eps}^1 \phi_0(x)dx>0$ depends on $\eps$ and $\phi_0$. For any hyperbolic surface $X\in \sM_g$ with $\lmdx \leq \frac{3}{16}-\eps$, then $r_1(X)= \rm{\textbf{i}}\cdot \sqrt{-\lambda_1(X)+\frac{1}{4}}$ where $ \sqrt{-\lambda_1(X)+\frac{1}{4}}\geq \sqrt{\frac{1}{16}+\eps}$. Then the conclusion follows by choosing
\[C_\eps=\min\{B_{\eps},\left(\sqrt{1+16\eps}(1-\eps)-1\right)\}.\]

The proof is complete.
\ep

%%%%%%%%%%%%%%%%%%%%%%%%%%%%%%%%%%%%%%%%
\section{Proofs of Theorem \ref{main} and \ref{main-2}--relatively easy parts}\label{sec-proof-1}
In the following two sections we finish the proofs of Theorem \ref{main} and \ref{main-2}. In this section we study the relative easy parts.

Let $X\in \sM_g$ be a hyperbolic surface of genus $g$. A closed geodesic $\gamma \subset X$ is called \emph{non-simple} if $\gamma$ intersects itself; otherwise it is called \emph{simple}. A simple closed geodesic $\alpha \subset X$ is called \emph{non-separating} if the complement $X\setminus \alpha$ is connected; otherwise it is called \emph{separating}. Recall that $\mathcal{P}(X)$ is the set of all oriented primitive closed geodesics on $X$. Now we split it as the following cases:
\begin{enumerate}
\item $\mathcal{P}^{s}_{sep}(X)\overset{\text{def}}{=}\{\gamma \in \mathcal{P}(X) \ \text{is simple and separating}\}.$
 
\item $\mathcal{P}^{s}_{nsep}(X)\overset{\text{def}}{=}\{\gamma \in \mathcal{P}(X) \ \text{is simple and non-separating}\}.$

\item $\mathcal{P}^{ns}(X)\overset{\text{def}}{=}\{\gamma \in \mathcal{P}(X) \ \text{is non-simple}\}.$
\end{enumerate}
Clearly we have $$\mathcal{P}(X)=\mathcal{P}^{s}_{sep}(X) \cup \mathcal{P}^{s}_{nsep}(X) \cup \mathcal{P}^{ns}(X).$$

Let $\phi_T(x)$ be the function in Section \ref{section trace}. We plug $\phi_T$ into Selberg's trace formula Theorem \ref{sel} and rewrite it as
\bear \label{sel-split}
&& \quad \quad  \sum_{k=0}^{\infty}\widehat{\phi_T}(r_k(X))=(g-1)\int_{-\infty}^\infty r \widehat{\phi_T}(r)\tanh(\pi r)dr+\\
&& \sum_{\gamma \in \mathcal P(X)} \frac{\ell_{\gamma}(X)}{2\sinh \left(\frac{\ell_{\gamma}(X)}{2} \right)}\phi_T( \ell_{\gamma}(X))+\sum_{\gamma \in \mathcal P(X)}\sum_{k=2}^\infty\frac{\ell_{\gamma}(X)}{2\sinh \left(\frac{k \ell_{\gamma}(X)}{2} \right)}\phi_T(k \ell_{\gamma}(X))\nonumber\\
&&=\underbrace{(g-1)\int_{-\infty}^\infty r \widehat{\phi_T}(r)\tanh(\pi r)dr}_{\rm I}+\underbrace{\sum_{\gamma \in \mathcal P(X)}\sum_{k=2}^\infty\frac{\ell_{\gamma}(X)}{2\sinh \left(\frac{k \ell_{\gamma}(X)}{2} \right)}\phi_T(k \ell_{\gamma}(X))}_{\rm{II}} \nonumber\\
&&+\underbrace{\sum_{\gamma \in \mathcal{P}^{s}_{sep}(X)} \frac{\ell_{\gamma}(X)}{2\sinh \left(\frac{\ell_{\gamma}(X)}{2} \right)}\phi_T( \ell_{\gamma}(X))}_{\rm{III}}+\underbrace{\sum_{\gamma \in \mathcal{P}^{s}_{nsep}(X)} \frac{\ell_{\gamma}(X)}{2\sinh \left(\frac{\ell_{\gamma}(X)}{2} \right)}\phi_T( \ell_{\gamma}(X))}_{\rm{IV}} \nonumber \\
&& +\underbrace{\sum_{\gamma \in \mathcal{P}^{ns}(X)} \frac{\ell_{\gamma}(X)}{2\sinh \left(\frac{\ell_{\gamma}(X)}{2} \right)}\phi_T( \ell_{\gamma}(X))}_{\rm{V}}. \nonumber 
\eear 

Next we will take an integral of Equation \eqref{sel-split} over the moduli space $\sM_g$ endowed with the \wep \ metric. Recall that $\widehat{\phi_T}(r_k(X))\geq 0$ for all $k\geq 0$. For the $\LHS$ of \eqref{sel-split}, we will only keep the first two terms $\widehat{\phi_T}(r_0(X))$ and $\widehat{\phi_T}(r_1(X))$. For the $\RHS$ of \eqref{sel-split}, we will bound the five terms case by case: Term-\rm{I} can be bounded by using an elementary observation; we will combine an argument in \cite{MNP20} and a result in \cite{Mirz13} to bound Term-\rm{II}; for Term-\rm{III} and Term-\rm{IV} we will apply Mirzakhani's integration formula \cite{Mirz07} to bound them; the last Term-\rm{V} is the most difficult case on which we will apply the new counting result Theorem \ref{count} and combine similar ideas in \cite{Mirz07, NWX20} to get the desired bound. We postpone the study of Term-\rm{V} until the next section.   
%%%%%%%%%%%%%%%
\subsection{An upper bound for $\int_{\sM_g}\rm{I} dX$} 
In this subsection we provide the following bound on Term \rm{I} in the $\rm RHS$ of \eqref{sel-split}.
\begin{proposition}\label{I}
Let $\phi_T$ be the function in Section \ref{section trace}. Then we have for all $T>1$ and as $g\to \infty$,
\[\left|\frac{1}{V_g}\int_{\sM_g}\left((g-1)\int_{-\infty}^\infty r \widehat{\phi_T}(r)\tanh(\pi r)dr \right)dX \right|\prec \frac{g}{T}.\]
\end{proposition}
\bp
Since $\phi_0$ is compactly supported, its Fourier transform $\widehat{\phi_0}$ is a Schwartz function which decays faster than any polynomial. In particular there exists a constant $C>0$ such that
\[\int_0^\infty |x| |\widehat{\phi_0}(x) |dx \leq C<\infty.\]
Recall that $\phi_T$ is an even function and $\left|\tanh(\pi r)\right|\leq 1$. Since $\widehat{\phi_T}(r)=T\widehat{\phi_0}(Tr)$,
\begin{eqnarray*}
\left|\int_{-\infty}^\infty r \widehat{\phi_T}(r)\tanh(\pi r)dr\right|&=&T \left|\int_{-\infty}^\infty r\widehat{\phi_0}(Tr) \tanh(\pi r)dr \right|\\
&\leq & \frac{2}{T} \int_0^\infty |x| |\widehat{\phi_0}(x) |dx \nonumber\\
&\leq & \frac{2C}{T} \nonumber
\end{eqnarray*}
which clearly implies the conclusion.
\ep
%%%%%%%%%%%%%%
\subsection{An upper bound for $\int_{\sM_g}\rm{II} dX$} 
In this subsection we prove the following bound on Term \rm{II} in the $\rm RHS$ of \eqref{sel-split}.
\begin{proposition}\label{II}
Let $\phi_T$ be the function in Section \ref{section trace}. Then we have for all $T>1$ and as $g\to \infty$,
\[\frac{1}{V_g}\int_{\sM_g}\sum_{\gamma \in \mathcal P(X)}\sum_{k=2}^\infty\frac{\ell_{\gamma}(X)}{2\sinh \left(\frac{k \ell_{\gamma}(X)}{2} \right)}\phi_T(k \ell_{\gamma}(X)) dX\prec T^2 g.\]
\end{proposition}

We split the proof into several lemmas.
\bl \label{II-1}
For any $X\in \sM_g^{\geq 1}$ and $T> 1$, then we have
\[\sum_{\gamma \in \mathcal P(X)}\sum_{k=2}^\infty\frac{\ell_{\gamma}(X)}{2\sinh \left(\frac{k \ell_{\gamma}(X)}{2} \right)}\phi_T(k \ell_{\gamma}(X)) \prec T^2 g.\]
\el

\bp
Since $X\in \sM_g^{\geq 1}$, we have that for all $\gamma \in \mathcal P(X)$,
\[\sinh \left(\frac{k \ell_{\gamma}(X)}{2} \right) \asymp e^{\frac{k\ell_{\gamma}(X)}{2}}.\]
Now we follow \cite{MNP20}. Recall that $\phi_T$ is bounded and $\Supp(\phi_T)=(-T,T)$. So we have
\bear \label{ii-k-1-0}
&&\sum_{\gamma \in \mathcal P(X)}\sum_{k=2}^\infty\frac{\ell_{\gamma}(X)}{2\sinh \left(\frac{k \ell_{\gamma}(X)}{2} \right)}\phi_T(k \ell_{\gamma}(X))\\ &&\asymp \sum_{\gamma \in \mathcal P(X)}\sum_{k=2}^\infty \ell_{\gamma}(X) e^{- \frac{k\ell_{\gamma}(X)}{2}}\phi_T(k \ell_{\gamma}(X)) \nonumber\\
&&\prec  \sum_{\gamma \in \mathcal P(X), \ \ell_{\gamma}(X)\leq \frac{T}{2}}\ell_{\gamma}(X) e^{-\ell_{\gamma}(X)} \nonumber\\
&&\leq \sum_{m=1}^{[T]}me^{-m} \cdot \#\{\gamma \in \mathcal{P}(X);  m\leq \ell_{\gamma}(X) <m+1\}. \nonumber
\eear
It follows by Lemma \ref{count-2} that
\bear\label{ii-k-k-1}
 \#\{\gamma \in \mathcal{P}(X);  m\leq \ell_{\gamma}(X) <m+1\} \prec g e^m
\eear
which together with \eqref{ii-k-1-0} imply that
\bear
 \sum_{\gamma \in \mathcal P(X)}\sum_{k=2}^\infty\frac{\ell_{\gamma}(X)}{2\sinh \left(\frac{k \ell_{\gamma}(X)}{2} \right)}\phi_T(k \ell_{\gamma}(X)) &\prec& \sum_{m=1}^{[T]}me^{-m} \cdot g\cdot e^m \nonumber\\
& \asymp& T^2g. \nonumber
\eear

The proof is complete.
\ep

\bl \label{II-2}
For any $X\in \sM_g^{< 1}$ and $T>1$, then we have
\[\sum_{\gamma \in \mathcal P(X)}\sum_{k=2}^\infty\frac{\ell_{\gamma}(X)}{2\sinh \left(\frac{k \ell_{\gamma}(X)}{2} \right)}\phi_T(k \ell_{\gamma}(X)) \prec \left(T^2 g+ \sum_{\substack{\alpha \in \mathcal{P}(X);\\ \ell_\alpha(X)< 1}}\frac{T}{\ell_\alpha(X)}\right).\]
\el

\bp
First we rewrite
\bear \label{2-1-2}
&&\sum_{\gamma \in \mathcal P(X)}\sum_{k=2}^\infty\frac{\ell_{\gamma}(X)}{2\sinh \left(\frac{k \ell_{\gamma}(X)}{2} \right)}\phi_T(k \ell_{\gamma}(X))\\
&&=\sum_{\substack{\alpha \in \mathcal{P}(X);\\ \ell_\alpha(X)< 1}}\sum_{k=2}^\infty\frac{\ell_{\gamma}(X)}{2\sinh \left(\frac{k \ell_{\gamma}(X)}{2} \right)}\phi_T(k \ell_{\gamma}(X)) \nonumber\\
&&+\sum_{\substack{\alpha \in \mathcal{P}(X);\\ \ell_\alpha(X)\geq 1}}\sum_{k=2}^\infty\frac{\ell_{\gamma}(X)}{2\sinh \left(\frac{k \ell_{\gamma}(X)}{2} \right)}\phi_T(k \ell_{\gamma}(X)). \nonumber
\eear

\noindent For the second term of the $\RHS$ above, one may apply the same argument in the proof of Lemma \ref{II-1} to get
\be\label{2-2-2}
\sum_{\substack{\alpha \in \mathcal{P}(X);\\ \ell_\alpha(X)\geq 1}}\sum_{k=2}^\infty\frac{\ell_{\gamma}(X)}{2\sinh \left(\frac{k \ell_{\gamma}(X)}{2} \right)}\phi_T(k \ell_{\gamma}(X))\prec T^2 g.
\ene

\noindent For the first term of the $\RHS$ of \eqref{2-1-2} above, recall that $\phi_T$ is bounded and $\Supp(\phi_T)=(-T,T)$. Since $\sinh(kx)\geq kx$ for all $x\geq 0$ and $k\geq 0$, we have
\bear
\sum_{\substack{\alpha \in \mathcal{P}(X);\\ \ell_\alpha(X)< 1}}\sum_{k=2}^\infty\frac{\ell_{\gamma}(X)}{2\sinh \left(\frac{k \ell_{\gamma}(X)}{2} \right)}\phi_T(k \ell_{\gamma}(X)) \prec \sum_{\substack{\alpha \in \mathcal{P}(X);\\ \ell_\alpha(X)< 1}} \sum_{k=2}^{[\frac{T}{\ell_\gamma(X)}]+1}\frac{1}{k}  \nonumber
\eear 
For all $n\geq 2$, it is not hard to see that
\[\frac{1}{2}+\frac{1}{3}+\cdots+\frac{1}{n}<4 \ln(n).\]
Thus, we have
\bear\label{2-3-2}
&&\sum_{\substack{\alpha \in \mathcal{P}(X);\\ \ell_\alpha(X)< 1}}\sum_{k=2}^\infty\frac{\ell_{\gamma}(X)}{2\sinh \left(\frac{k \ell_{\gamma}(X)}{2} \right)}\phi_T(k \ell_{\gamma}(X)) \\
&& \prec \sum_{\substack{\alpha \in \mathcal{P}(X);\\ \ell_\alpha(X)< 1}} \ln \left([\frac{T}{\ell_\gamma(X)}]+1  \right) \nonumber \\
&&\leq  \sum_{\substack{\alpha \in \mathcal{P}(X);\\ \ell_\alpha(X)< 1}} \ln \left(\frac{2T}{\ell_\gamma(X)}  \right) \nonumber \\
&& \leq 2\sum_{\substack{\alpha \in \mathcal{P}(X);\\ \ell_\alpha(X)< 1}}\frac{T}{\ell_\alpha(X)}  \nonumber
\eear 
where in the last inequality we apply the fact that $e^x\geq x$ for all $x\geq 0$.

Then the conclusion clearly follows by \eqref{2-1-2}, \eqref{2-2-2} and \eqref{2-3-2}.
\ep

The Collar Lemma \cite{Buser10} implies that $\alpha$ is always simple if $\ell_{\alpha}(X)<1$. As in \cite[Page 292]{Mirz13} we define $f:\sM_g \to \R^{\geq 0}$ as
\[f(X)\overset{\text{def}}{=} \sum_{\ell_\alpha(X)\leq 1}\frac{1}{\ell_\alpha(X)}.\]
By using her Integration Formula (see Theorem \ref{Mirz int formula}), Mirzakhani \cite[Page 292]{Mirz13} showed that
\bl \label{II-3}
\[\int_{\sM_g}f(X)dX \asymp V_g.\]
\el

Now we are ready to prove Proposition \ref{II}.
\bp [Proof of Proposition \ref{II}]
It follows by Lemma \ref{II-1}, \ref{II-2} and \ref{II-3} that
\beqar
&&\frac{1}{V_g}\int_{\sM_g}\sum_{\gamma \in \mathcal P(X)}\sum_{k=2}^\infty\frac{\ell_{\gamma}(X)}{2\sinh \left(\frac{k \ell_{\gamma}(X)}{2} \right)}\phi_T(k \ell_{\gamma}(X)) dX\\
&&=\frac{1}{V_g}\int_{\sM_g^{\geq 1}}\sum_{\gamma \in \mathcal P(X)}\sum_{k=2}^\infty\frac{\ell_{\gamma}(X)}{2\sinh \left(\frac{k \ell_{\gamma}(X)}{2} \right)}\phi_T(k \ell_{\gamma}(X)) dX\nonumber\\
&&+\frac{1}{V_g}\int_{\sM_g^{< 1}}\sum_{\gamma \in \mathcal P(X)}\sum_{k=2}^\infty\frac{\ell_{\gamma}(X)}{2\sinh \left(\frac{k \ell_{\gamma}(X)}{2} \right)}\phi_T(k \ell_{\gamma}(X)) dX\nonumber\\
&& \prec T^2g\cdot  \frac{\Vol(\sM_g^{\geq 1})}{V_g}+\left( T^2g\cdot  \frac{\Vol(\sM_g^{< 1})}{V_g}+\frac{T}{V_g}\int_{\sM_g^{<1}}f(X)dX \right)\nonumber\\
&&\prec T^2 g. \nonumber
\eeqar

The proof is complete.
\ep

%%%%%%%%%%%%%%%%%%
\subsection{An upper bound for $\int_{\sM_g}\rm{III} dX$} 
In this subsection we apply the Integration Formula of Mirzakhani (see Theorem \ref{Mirz int formula}) as a tool to prove the following bound on Term \rm{III} in the $\rm RHS$ of \eqref{sel-split}.
\begin{proposition}\label{III}
Let $\phi_T$ be the function in Section \ref{section trace}. Then we have for all $T>1$ and as $g\to \infty$,
\[\frac{1}{V_g}\int_{\sM_g}\sum_{\gamma \in \mathcal P^s_{sep}(X)}\frac{\ell_{\gamma}(X)}{2\sinh \left(\frac{ \ell_{\gamma}(X)}{2} \right)}\phi_T( \ell_{\gamma}(X)) dX\prec \frac{e^{\frac{T}{2}}}{g}.\]
\end{proposition}

\bp
For each $1\leq k \leq [\frac{g}{2}]$, we let $\alpha_k \subset X$ be an unoriented simple closed geodesic separating $X$ into $X_{k,1}\cup X_{g-k,1}$ where $X_{k,1}$ is a subsurface in $X$ of $k$ genus with one boundary curve $\alpha_k$. Recall that $\mathcal P^s_{sep}(X)$ is the set of oriented simple and separating closed geodesics. So for $\gamma \in \mathcal P^s_{sep}(X)$, $\gamma \neq \gamma^{-1}\in \mathcal P^s_{sep}(X)$. But they have the same lengths. By symmetry we have
\[\sum_{\gamma \in \mathcal P^s_{sep}(X)}\frac{\ell_{\gamma}(X)}{2\sinh \left(\frac{ \ell_{\gamma}(X)}{2} \right)}\phi_T( \ell_{\gamma}(X))=2\sum_{k=1}^{[\frac{g}{2}]} \sum_{\gamma \in \Mod_g\cdot \alpha_k}\frac{\ell_{\gamma}(X)}{2\sinh \left(\frac{ \ell_{\gamma}(X)}{2} \right)}\phi_T( \ell_{\gamma}(X)).\]
Now one may apply the Integration Formula of Mirzakhani (see Theorem \ref{Mirz int formula}) to get
\bear \label{3-1-0}
&&\frac{1}{V_g}\int_{\sM_g}\sum_{\gamma \in \mathcal P^s_{sep}(X)}\frac{\ell_{\gamma}(X)}{2\sinh \left(\frac{ \ell_{\gamma}(X)}{2} \right)}\phi_T( \ell_{\gamma}(X)) dX\\
&&\leq \frac{2}{V_g}\sum_{k=1}^{[\frac{g}{2}]} \int_0^{\infty} \frac{x}{2\sinh\left(\frac{x}{2}\right)} \phi_T(x) V_{k,1}(x) V_{g-k,1}(x)xdx.\nonumber
\eear
Recall that Lemma \ref{MP vol lemma} tells that for all $1\leq k \leq [\frac{g}{2}]$,
\[V_{k,1}(x)\leq \frac{\sinh(x/2)}{x/2}V_{k,1} \ \text{and} \ V_{g-k,1}(x)\leq \frac{\sinh(x/2)}{x/2}V_{g-k,1}.\]
Thus, combine \eqref{3-1-0} and the two inequalities above we get
\bear \label{3-1-0-2}
&&\frac{1}{V_g}\int_{\sM_g}\sum_{\gamma \in \mathcal P^s_{sep}(X)}\frac{\ell_{\gamma}(X)}{2\sinh \left(\frac{ \ell_{\gamma}(X)}{2} \right)}\phi_T( \ell_{\gamma}(X)) dX\\
&&\leq 2 \int_0^{\infty} 2\sinh\left(\frac{x}{2}\right) \phi_T(x) dx\cdot \left(\frac{\sum_{k=1}^{[\frac{g}{2}]}V_{k,1}V_{g-k,1}}{V_g} \right) \nonumber \\
&&\leq 2 \int_0^{\infty} e^{\frac{x}{2}} \phi_T(x) dx\cdot \left(\frac{\sum_{k=1}^{[\frac{g}{2}]}V_{k,1}V_{g-k,1}}{V_g} \right). \nonumber
\eear
By Lemma \ref{Mirz vol lemma 2} we know that
\bear\label{sum-k-1}
\frac{\sum_{k=1}^{[\frac{g}{2}]}V_{k,1}V_{g-k,1}}{V_g} \asymp \frac{1}{g}.
\eear
Plug \eqref{sum-k-1} into \eqref{3-1-0-2}, since $\phi_T$ is bounded and $\Supp(\phi_T)=(-T,T)$ we get
\bear
\frac{1}{V_g}\int_{\sM_g}\sum_{\gamma \in \mathcal P^s_{sep}(X)}\frac{\ell_{\gamma}(X)}{2\sinh \left(\frac{ \ell_{\gamma}(X)}{2} \right)}\phi_T( \ell_{\gamma}(X)) dX &\prec& \frac{1}{g} \cdot \int_0^T e^{\frac{x}{2}} dx  \nonumber\\
&\asymp& \frac{e^{\frac{T}{2}}}{g}.\nonumber
\eear

The proof is complete.
\ep
%%%%%%%%%%%%%%%%%%%%%%%%

\subsection{A bound for $\int_{\sM_g}\rm{IV} dX$} 
In this subsection we also apply the Integration Formula of Mirzakhani (see Theorem \ref{Mirz int formula}) as a tool to prove the following bound to link Term \rm{IV} in the $\rm RHS$ of \eqref{sel-split} and $\widehat{\phi_T}(\frac{\textbf{i}}{2})$.
\begin{proposition}\label{IV}
Let $\phi_T$ be the function in Section \ref{section trace}. Then we have for all $T>1$ and as $g\to \infty$,
\[\left|\frac{1}{V_g}\int_{\sM_g}\sum_{\gamma \in \mathcal P^{s}_{nsep}(X)}\frac{\ell_{\gamma}(X)}{2\sinh \left(\frac{ \ell_{\gamma}(X)}{2} \right)}\phi_T( \ell_{\gamma}(X)) dX-\widehat{\phi_T}(\frac{\textbf{i}}{2})\right|\prec \left(\frac{T^2e^{\frac{T}{2}}}{g}+1\right).\]
\end{proposition}

We first list the following elementary properties for $\phi_T$:
\begin{enumerate}
\item[(a)] $\int_0^\infty e^{-\frac{x}{2}}\phi_T(x)dx \prec 1$.

\item[(b)] $\int_{0}^\infty e^{\frac{x}{2}}\phi_T(x)dx\prec e^{\frac{T}{2}}.$

\item[(c)] $\widehat{\phi_T}(\frac{\textbf{i}}{2})=\int_{-\infty}^\infty e^{\frac{x}{2}}\phi_T(x)dx= \int_{0}^\infty 2\cosh\left(\frac{x}{2}\right)\phi_T(x)dx\prec e^{\frac{T}{2}}.$
\end{enumerate}

Now we prove Proposition \ref{IV}.
\bp [Proof of Proposition \ref{IV}]
Let $\alpha_0 \subset X$ be an unoriented simple non-separating closed geodesic. Similar as in the proof of Proposition \ref{III}, for $\gamma \in \mathcal P^s_{nsep}(X)$, $\gamma \neq \gamma^{-1}\in \mathcal P^s_{nsep}(X)$. But they have the same lengths. So by symmetry we have
\[\sum_{\gamma \in \mathcal P^{s}_{nsep}(X)}\frac{\ell_{\gamma}(X)}{2\sinh \left(\frac{ \ell_{\gamma}(X)}{2} \right)}\phi_T( \ell_{\gamma}(X))= 2 \left(\sum_{\gamma \in \Mod_g\cdot \alpha_0}\frac{\ell_{\gamma}(X)}{2\sinh \left(\frac{ \ell_{\gamma}(X)}{2} \right)}\phi_T( \ell_{\gamma}(X))\right).\]
Since $\alpha_0$ is simple and non-separating and $g>2$, one may apply the Integration Formula of Mirzakhani (see Theorem \ref{Mirz int formula}) to get
\bear
&&\int_{\sM_g}\sum_{\gamma \in \Mod_g\cdot \alpha_0}\frac{\ell_{\gamma}(X)}{2\sinh \left(\frac{ \ell_{\gamma}(X)}{2} \right)}\phi_T( \ell_{\gamma}(X))dX\\
&&=\frac{1}{2}\int_0^{\infty} \frac{x}{2\sinh\left(\frac{x}{2}\right)} \phi_T(x) V_{g-1,2}(x,x)xdx. \nonumber
\eear
Thus, we have
\bear\label{4-1-1}
&& \int_{\sM_g}\sum_{\gamma \in \mathcal P^{s}_{nsep}(X)}\frac{\ell_{\gamma}(X)}{2\sinh \left(\frac{ \ell_{\gamma}(X)}{2} \right)}\phi_T( \ell_{\gamma}(X)) dX\\
&&=\int_0^{\infty} \frac{x}{2\sinh\left(\frac{x}{2}\right)} \phi_T(x) V_{g-1,2}(x,x)xdx. \nonumber
\eear
By Lemma \ref{MP vol lemma} we know that
\[\frac{V_{g-1,2}(x,x)}{V_{g-1,2}}=\left( \frac{\sinh(x/2)}{x/2} \right)^2 \cdot \left(1+O\left(\frac{x^2}{g}\right) \right)\]
where the implied constant is uniform. By Lemma \ref{Mirz vol lemma 1} we have
\[V_{g-1,2}=V_g \cdot \left(1+O\left(\frac{1}{g}\right)\right)\]
where the implied constant is also uniform. So we have
\be \label{4-1-1-0}
\frac{V_{g-1,2}(x,x)}{V_{g}}=\left( \frac{\sinh(x/2)}{x/2} \right)^2 \cdot \left(1+O\left(\frac{1+x^2}{g}\right) \right).
\ene
Plug \eqref{4-1-1-0} into \eqref{4-1-1} we get
\bear\label{4-2-1}
&& \frac{1}{V_g}\int_{\sM_g}\sum_{\gamma \in \mathcal P^{s}_{nsep}(X)}\frac{\ell_{\gamma}(X)}{2\sinh \left(\frac{ \ell_{\gamma}(X)}{2} \right)}\phi_T( \ell_{\gamma}(X)) dX \\
&&=\int_0^{\infty} 2\sinh\left(\frac{x}{2}\right) \phi_T(x) \cdot \left(1+O\left(\frac{1+x^2}{g}\right) \right)dx. \nonumber\\
&&=\int_0^{\infty} e^{\frac{x}{2}} \phi_T(x) \cdot \left(1+O\left(\frac{1+x^2}{g}\right) \right)dx \nonumber \\
&&  -\int_0^{\infty} e^{\frac{-x}{2}} \phi_T(x) \cdot \left(1+O\left(\frac{1+x^2}{g}\right) \right)dx. \nonumber
\eear
Recall that we have $\widehat{\phi_T}(\frac{\textbf{i}}{2})=\int_{0}^\infty 2\cosh\left(\frac{x}{2}\right)\phi_T(x)dx $, $\int_{0}^\infty \phi_T(x)e^{\frac{x}{2}}dx\prec e^{\frac{T}{2}}$ and $\int_0^\infty e^{-\frac{x}{2}}\phi_T(x)dx \prec 1$. Now it follows by \eqref{4-2-1} that
\bear\label{4-3-1}
 &&   \left|\frac{1}{V_g}\int_{\sM_g}\sum_{\gamma \in \mathcal P^{s}_{nsep}(X)}\frac{\ell_{\gamma}(X)}{2\sinh \left(\frac{ \ell_{\gamma}(X)}{2} \right)}\phi_T( \ell_{\gamma}(X)) dX-\widehat{\phi_T}(\frac{\textbf{i}}{2})\right| \nonumber \\
&&\prec \int_0^{\infty} e^{\frac{x}{2}} \phi_T(x) \cdot \left(\frac{1+x^2}{g}\right) dx  \nonumber \\
&&  + \int_0^{\infty} e^{\frac{-x}{2}} \phi_T(x) \cdot \left(2+\frac{1+x^2}{g} \right)dx \nonumber \\
&& \prec \frac{T^2e^{\frac{T}{2}}}{g}+\left(2+\frac{T^2}{g} \right) \nonumber\\
&& \prec \frac{T^2e^{\frac{T}{2}}}{g}+1. \nonumber
\eear

The proof is complete.
\ep
%%%%%%%%%%%%%%%%%%%%%%%

\section{Proofs of Theorem \ref{main}, \ref{main-2} and \ref{mt-diam}}\label{sec-main-1}
In this section we resolve intersections of non-simple closed geodesics, and apply the new counting result Theorem \ref{count} and similar ideas in \cite{Mirz07, NWX20} to study the most difficult case: Term \rm{V} in the $\rm RHS$ of \eqref{sel-split}. Then we will combine the new desired bound on Term \rm{V} and the results in the previous section to complete the proofs of Theorem \ref{main} and \ref{main-2}.

\subsection{An upper bound for Term \rm{V}} 
In this subsection we apply the new counting result Theorem \ref{count} to give an effective upper bound for Term \rm{V} of \eqref{sel-split}.

Throughout this section we always assume that $g>1$ is large enough and \be T=a\ln(g)\ene where $a>0$ is any fixed constant. 

Let $X\in \sM_g$ be a hyperbolic surface and $\gamma \in \mathcal{P}^{ns}(X)$ be a non-simple closed geodesic of length $\ell_\gamma(X)\leq T$. By Proposition \ref{ns-sub} one may assume that $X(\gamma)\subset X$ is a connected subsurface of geodesic boundary (we warn here that two distinct simple closed geodesics on the boundary of $X(\gamma)$ may correspond to a single simple closed geodesic in $X$) such that
\ben
\item $\gamma\subset X(\gamma)$ is filing;
\item $\ell(\partial X(\gamma))\leq 2\ell_\gamma(X)\leq 2T$;
\item $\area(X(\gamma))\leq 4 \ell_\gamma(X) \leq 4T$.
\een

\begin{def*} For $T=a\ln(g)>0$ and $X\in \sM_g$, we define
$$
\rm{Sub_T(X)}\overset{\text{def}}{=}\left\{ \parbox[1]{8.5cm}{$Y\subset X$ is a connected subsurface of geodesic boundary such that $\ell(\partial Y)\leq 2T$ and $\area(Y)\leq 4T$
}\right\}
$$ 
where we allow two distinct simple closed geodesics on the boundary of $Y$ to be a single simple closed geodesic in $X$.
\end{def*}
For large enough $g>1$, we have that the map
\[Y \mapsto \partial Y\]
is injective; indeed if $\partial Y_1=\partial Y_2$ for $Y_1\neq Y_2 \in \rm{Sub_T(X)}$, then  $Y_1$ and $Y_2$ lie on the two different sides of $\partial Y_1=\partial Y_2$ and we have $Y_1\cup Y_2=X$ implying that 
\[4\pi(g-1)=\area(X)\leq\area(Y_1)+\area(Y_2)\leq 8a \ln(g)\] 
which is impossible for large enough $g>1$. 

Let $\gamma \in \mathcal{P}^{ns}(X)$ be of length $\ell_\gamma(X)\leq T$ and consider the composition $\mathcal{F}$ of the following two maps
\[\gamma \mapsto X(\gamma) \mapsto \partial X(\gamma),\]
then we have that for any $\gamma \in \mathcal{P}^{ns}(X)$ of length $\ell_\gamma(X)\leq T$,
\be \label{up-mul}
\#\left\{\gamma' \in \mathcal{P}^{ns}(X); \ \ell_{\gamma'}(X)\leq T \ \textit{and} \ \mathcal{F}(\gamma')=\mathcal{F}(\gamma) \right\}\leq 2\#_f\left(X(\gamma), T\right)
\ene
where $\#_f\left(X(\gamma), T\right)$ is the number of filling (unoriented) closed geodesics in $X(\gamma)$ of length less than or equal to $T$.

We prove the following upper bound for Term \rm{V} in the $\rm RHS$ of \eqref{sel-split}.
\begin{proposition}\label{V-upp}
Let $\phi_T$ be the function in Section \ref{section trace}. For any $\eps_1>0$, there exists a constant $c(\eps_1)>0$ only depending on $\eps_1$ such that as $g\to \infty$,
\bear \label{eq-v-upp}
&&\sum_{\gamma \in \mathcal{P}^{ns}(X)} \frac{\ell_{\gamma}(X)}{2\sinh \left(\frac{\ell_{\gamma}(X)}{2} \right)}\phi_T( \ell_{\gamma}(X))\\
&&\prec T^2e^T\left(\sum_{\substack{Y\in \rm{Sub_T(X)};\\ |\chi(Y)|\geq 17}} e^{-\frac{\ell(\partial Y)}{4}}\textbf{1}_{[0,2T]}(\ell(\partial Y))\right) \nonumber\\
&&+c(\eps_1)T \left(\sum_{\substack{Y\in \rm{Sub_T(X)};\\ 1\leq |\chi(Y)|\leq 16}} e^{\frac{T}{2}-\frac{1-\eps_1}{2}\ell(\partial Y)} \mathrm{\textbf{1}_{[0,2T]}}(\ell(\partial Y))\right). \nonumber
\eear
\end{proposition}
\bp
Recall that every non-simple closed geodesic has length at least $4\arcsinh(1)$ (\eg see \cite[4.2.2]{Buser10}). Since $\phi_T\geq 0$ is bounded and $\Supp(\phi_T)=(-T,T)$, we get
\bear \label{V-e-0-1}
&&\sum_{\gamma \in \mathcal{P}^{ns}(X)} \frac{\ell_{\gamma}(X)}{2\sinh \left(\frac{\ell_{\gamma}(X)}{2} \right)}\phi_T( \ell_{\gamma}(X))\asymp \sum_{\gamma \in \mathcal{P}^{ns}(X)} \ell_{\gamma}(X)e^{-\frac{\ell_{\gamma}(X)}{2} }\phi_T( \ell_{\gamma}(X))\\
&&\prec \sum_{\gamma \in \mathcal{P}^{ns}(X)} \ell_{\gamma}(X)e^{-\frac{\ell_{\gamma}(X)}{2} } \textbf{1}_{[0,T]}(\ell_{\gamma}(X)) \nonumber \\
&&= 2\times \left( \sum_{Y\in \rm{Sub_T(X)}}\sum_{\gamma\subset Y \ \textit{is filling}} \ell_{\gamma}(X)e^{-\frac{\ell_{\gamma}(X)}{2} } \textbf{1}_{[0,T]}(\ell_{\gamma}(X))\right) \nonumber \\
&&= 2\times  \sum_{\substack{Y\in \rm{Sub_T(X)};\\ |\chi(Y)|\geq 17}} \ \sum_{\gamma\subset Y \ \textit{is filling}} \ell_{\gamma}(X)e^{-\frac{\ell_{\gamma}(X)}{2} } \textbf{1}_{[0,T]}(\ell_{\gamma}(X)) \nonumber\\
&&+2\times\sum_{\substack{Y\in \rm{Sub_T(X)};\\ 1\leq|\chi(Y)|\leq 16}}\ \sum_{\gamma\subset Y \ \textit{is filling}} \ell_{\gamma}(X)e^{-\frac{\ell_{\gamma}(X)}{2} } \textbf{1}_{[0,T]}(\ell_{\gamma}(X)). \nonumber
\eear
Where the factor $2$ is from the multiplicity because the curves in $\mathcal{P}^{ns}(X)$ are oriented.
Now we consider the first term in the \rm{RHS} of \eqref{V-e-0-1}. Since $\gamma \subset Y$ is filling, 
\be \label{in-to-b}
\frac{\ell(\partial Y)}{2}\leq \ell_\gamma(Y)=\ell_{\gamma}(X).
\ene 
For $Y\in \rm{Sub_T(X)}$ we know that $\area(Y)\leq 4 T$. So by Lemma \ref{count-3} we have 
\be\label{c-g-b}
\#_f\left(Y, T\right)\prec Te^T.
\ene
By \eqref{in-to-b} and \eqref{c-g-b} we have    
\bear \label{V-e-0-2}
&&\sum_{\substack{Y\in \rm{Sub_T(X)};\\ |\chi(Y)|\geq 17}} \ \sum_{\gamma\subset Y \ \textit{is filling}} \ell_{\gamma}(X)e^{-\frac{\ell_{\gamma}(X)}{2} } \textbf{1}_{[0,T]}(\ell_{\gamma}(X))\\
&& \leq \sum_{\substack{Y\in \rm{Sub_T(X)};\\ |\chi(Y)|\geq 17}} \#_f\left(Y, T\right)T e^{-\frac{\ell(\partial Y)}{4} } \textbf{1}_{[0,2T]}(\ell(\partial Y)) \nonumber \\
&& \prec T^2e^T\left(\sum_{\substack{Y\in \rm{Sub_T(X)};\\ |\chi(Y)|\geq 17}} e^{-\frac{\ell(\partial Y)}{4}}\textbf{1}_{[0,2T]}(\ell(\partial Y))\right).\nonumber
\eear

Now we consider the second term in the \rm{RHS} of \eqref{V-e-0-1}. Let $Y\in \rm{Sub_T(X)}$ with $m=|\chi(Y)|\in[1,16]$. For any $\eps_1>0$, let $c(\eps_1)=\max\limits_{1\leq m \leq 16} c(m,\eps_1)>0$ where $c(m,\eps_1)>0$ is the constant in Theorem \ref{count}. Then we have that for large enough $g>1$,
\bear\label{f-u-s-1}
&&\sum_{\gamma\subset Y \ \textit{is filling}} \ell_{\gamma}(X)e^{-\frac{\ell_{\gamma}(X)}{2} } \textbf{1}_{[0,T]}(\ell_{\gamma}(X))\\
&&\leq \sum_{n=0}^{[T]}(n+1)e^{-\frac{n}{2}}\cdot \#\{\gamma\subset Y \ \textit{is filling}; \ n< \ell_\gamma(Y)\leq n+1\}\cdot \textbf{1}_{[0,2T]}(\ell(\partial Y)).\nonumber
\eear
By the new counting result Theorem \ref{count} we know that for each $n\geq 0$,
\be\label{n-c-f-1}
\#\{\gamma\subset Y \ \textit{is filling}; \ n< \ell_\gamma(Y)\leq n+1\}\leq c(\eps_1)e^{n+1-\frac{1-\eps_1}{2}\ell(\partial Y)}.
\ene
By \eqref{f-u-s-1} and \eqref{n-c-f-1} we have
\beqar
&& \sum_{\gamma\subset Y \ \textit{is filling}} \ell_{\gamma}(X)e^{-\frac{\ell_{\gamma}(X)}{2} } \textbf{1}_{[0,T]}(\ell_{\gamma}(X))\\
&&\leq c(\eps_1)e^{1-\frac{1-\eps_1}{2}\ell(\partial Y)}\sum_{n=0}^{[T]}(n+1)e^{\frac{n}{2}}\textbf{1}_{[0,2T]}(\ell(\partial Y))\\
&& \prec c(\eps_1) T e^{\frac{T}{2}-\frac{1-\eps_1}{2}\ell(\partial Y)} \textbf{1}_{[0,2T]}(\ell(\partial Y))
\eeqar
which implies that
\bear \label{V-e-0-3}
&&\sum_{\substack{Y\in \rm{Sub_T(X)};\\ 1\leq|\chi(Y)|\leq 16}}\ \sum_{\gamma\subset Y \ \textit{is filling}} \ell_{\gamma}(X)e^{-\frac{\ell_{\gamma}(X)}{2} } \textbf{1}_{[0,T]}(\ell_{\gamma}(X))\\
&&\prec c(\eps_1)T \left(\sum_{\substack{Y\in \rm{Sub_T(X)};\\ 1\leq |\chi(Y)|\leq 16}} e^{\frac{T}{2}-\frac{1-\eps_1}{2}\ell(\partial Y)} \textbf{1}_{[0,2T]}(\ell(\partial Y))\right).\nonumber
\eear

Then the conclusion follows by \eqref{V-e-0-1}, \eqref{V-e-0-2} and \eqref{V-e-0-3}. 
\ep

\begin{rem*}
The critical value $16$ for Euler characteristic in the Proposition above is not the unique choice. Actually the remaining argument also works if replacing $16$ by any positive integer larger than $16$.
\end{rem*}

\subsection{An upper bound for $\int_{\sM_g}\rm{V}dx$} In this subsection we prove our desired upper bound for $\int_{\sM_g}\rm{V}dx$. Recall that the map $Y\mapsto \partial Y$ is injective for $Y\in \rm{Sub_T(X)}$ and large enough $g>1$. We take an integral of \eqref{eq-v-upp} in Proposition \ref{V-upp} over $\sM_g$, and then apply the Integration Formula of Mirzakhani (see Theorem \ref{Mirz int formula}) to get the desired upper bounds. 

First we restrict the argument to a single orbit $\Mod_g \cdot Y \subset \rm{Sub_T(X)}$.

\begin{assum*} \label[$\star$]{star} let $Y_0\in \rm{Sub_T(X)}$ satisfying
	\begin{itemize}
           \item $Y_0$ is homeomorphic to $S_{g_0,k}$ for some $g_0\geq 0$ and $k>0$ with $m=|\chi(Y_0)|= 2g_0-2+k\geq 1$;
		\item the boundary $\partial Y_0$ is a simple closed multi-geodesics in $X$ consisting of $k$ simple closed geodesics which has $n_0$ pairs of simple closed geodesics for some $n_0\geq 0$ such that each pair corresponds to a single simple closed geodesic in $X$;
           \item the interior of its complement $X\setminus S_{g_0,k}$ consists of $q$ components $S_{g_1,n_1},\cdots,$ $S_{g_q,n_q}$ for some $q\geq 1$ where $\sum_{i=1}^q n_i=k-2n_0$.
	\end{itemize}
    (\eg see Figure \ref{Y0}). 
\end{assum*}

\begin{figure}[ht]
	\centering
	\includegraphics[width=7cm]{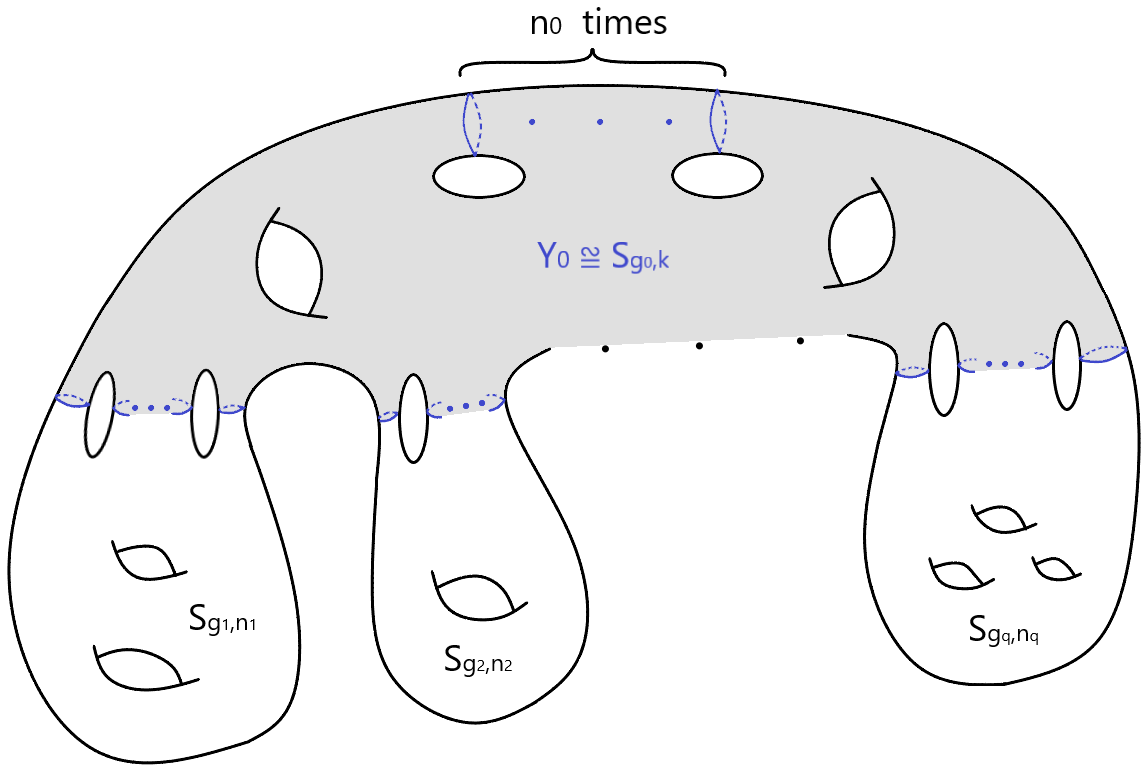}
	\caption{an example for $Y_0$}
	\label{Y0}
\end{figure}

Let $h:\R^{\geq 0}\to \R^{\geq 0}$ be a continuous function. Next we compute the integral $\int_{\sM_g}\sum_{Y\in \Mod_g \cdot Y_0}h(\ell(\partial Y))\textbf{1}_{[0,2T]}(\ell(\partial Y)) dX$.

Recall that the map $Y\mapsto \partial Y$ is injective for $Y\in \rm{Sub_T(X)}$ and large enough $g>1$, and $\ell(\partial Y)\leq 2T$. It follows by the Integration Formula of Mirzakhani (see Theorem \ref{Mirz int formula}) that
\bear \label{int-h-0}
&&\int_{\sM_g}\sum_{Y\in \Mod_g \cdot Y_0}h(\ell(\partial Y))\textbf{1}_{[0,2T]}(\ell(\partial Y)) dX \nonumber\\
&&= \int_{\sM_g}\sum_{ \partial Y\in \Mod_g \cdot \partial Y_0}h(\ell(\partial Y))\textbf{1}_{[0,2T]}(\ell(\partial Y)) dX \nonumber\\
&&\leq \frac{1}{|\Sym|} \int_{\R_{\geq0}^{k-n_0}}h\left(\sum_{j=1}^{n_0}2x_j'+\sum_{i=1}^q(x_{i,1} + \cdots +x_{i,n_i})\right) \nonumber \\
&&\mathbf 1_{[0,2T]}\left(\sum_{j=1}^{n_0}2x_j'+\sum_{i=1}^q(x_{i,1} + \cdots +x_{i,n_i})\right)  \nonumber \\
& & V_{g_0,k}(x_1',x_1',\cdots, x_{n_0}',x_{n_0}',x_{1,1},\cdots,x_{1,n_1}, x_{2,1}, \cdots,x_{q,n_q})\nonumber \\
& & V_{g_1,n_1}(x_{1,1},\cdots,x_{1,n_1})\cdots V_{g_q,n_q}(x_{q,1},\cdots,x_{q,n_q})\nonumber \\
& & x_1'\cdots x_{n_0}'\cdot x_{1,1}\cdots x_{q,n_q} dx_1'\cdots dx_{n_0}' dx_{1,1}\cdots dx_{q,n_q}.  \nonumber
\eear

\noindent Recall that the symmetry is given by $\Sym(\partial Y_0) = \mathop{\rm Stab}(\partial Y_0) / \cap_{\gamma\in\partial Y_0} \mathop{\rm Stab}(\gamma)$. For each $1\leq i \leq q$, the set of all permutations of $n_i$ boundary geodesics of $S_{g_i,n_i}$ gives $n_i!$ elements in $\Sym(\partial Y_0)$. Meanwhile, the set of all permutations of $n_0$ pairs of geodesics defined in Assumption ($\star$) gives $n_0!$ elements. So we have

\begin{equation*}
|\Sym| \geq n_0!n_1!\cdots n_q!.
\end{equation*}
By Lemma \ref{Mirz vol lemma 1}, we have
\begin{equation*}
V_{g,n}(x_1,\cdots,x_n) \leq e^{\frac{x_1+\cdots +x_n}{2}} V_{g,n},
\end{equation*}
Set the condition
\begin{equation*}
Cond:=\{0\leq x_{j}', \ 0\leq x_{i,j}, \  \sum_{j=1}^{n_0}2x_j'+\sum_{i=1}^q\sum_{j=1}^{n_i} x_{i,j} \leq 2T  \}.
\end{equation*}
Put all these equations together we get
\bear \label{int-h}
&& \quad  \int_{\sM_g}\sum_{Y\in \Mod_g \cdot Y_0}h(\ell(\partial Y))\textbf{1}_{[0,2T]}(\ell(\partial Y)) dX  \\
&& \leq \frac{1}{n_0!n_1!\cdots n_q!} V_{g_0,k} V_{g_1,n_1}\cdots V_{g_q,n_q} \times \nonumber\\
&&   \int_{Cond}  h\left(\sum_{j=1}^{n_0}2x_j'+\sum_{i=1}^q(x_{i,1} + \cdots +x_{i,n_i})\right)\cdot e^{\sum_{j=1}^{n_0}x_j'+\sum_{i=1}^q(x_{i,1} + \cdots +x_{i,n_i})} \nonumber\\
& & x_1'\cdots x_{n_0}'\cdot x_{1,1}\cdots x_{q,n_q} dx_1'\cdots dx_{n_0}' dx_{1,1}\cdots dx_{q,n_q}.  \nonumber
\eear

Now we apply \eqref{int-h} to bound the integral of the first term in the \rm{RHS} of \eqref{eq-v-upp} in Proposition \ref{V-upp}.

\begin{proposition}\label{upp-orbit}
Let $Y_0\in \rm{Sub_T(X)}$ satisfying Assumption ($\star$). Then we have that as $g\to \infty$,
\beqar
\int_{\sM_g}\sum_{Y\in \Mod_g \cdot Y_0}e^{-\frac{\ell(\partial Y)}{4}}\textbf{1}_{[0,2T]}(\ell(\partial Y))dX \prec  e^{\frac{7}{2}T} V_{g_0,k} \frac{ V_{g_1,n_1}\cdots V_{g_q,n_q}}{n_0!n_1!\cdots n_q!}.
\eeqar
\end{proposition}
\bp
We apply \eqref{int-h} for the case that $h(x)=e^{-\frac{x}{4}}$ to get
\bear \label{single-e-1}
&&\int_{\sM_g}\sum_{Y\in \Mod_g \cdot Y_0}e^{-\frac{\ell(\partial Y)}{4}}\textbf{1}_{[0,2T]}(\ell(\partial Y))dX\\
&&\prec \frac{1}{n_0!n_1!\cdots n_q!} V_{g_0,k} V_{g_1,n_1}\cdots V_{g_q,n_q} \nonumber \\
&&  \times \int_{Cond} e^{\frac{1}{2}\cdot(\sum_{j=1}^{n_0}x_j')+\frac{3}{4}\cdot(\sum_{i=1}^q(x_{i,1} + \cdots +x_{i,n_i}))} \nonumber\\
&& x_1'\cdots x_{n_0}'\cdot x_{1,1}\cdots x_{q,n_q} dx_1'\cdots dx_{n_0}' dx_{1,1}\cdots dx_{q,n_q}.  \nonumber
\eear
Recall that
\be
\int_{x_i\geq 0,\ \sum_{j=1}^n x_i \leq s}x_1\cdots x_n dx_1\cdots dx_n=\frac{(s)^{2n}}{(2n)!}<e^{s}. \nonumber
\ene
Since $Cond \subset \{x_i'\geq 0, x_{i,j}\geq 0,  \  \sum_{j=1}^{n_0}x_j'+\sum_{i=1}^q\sum_{j=1}^{n_i} x_{i,j} \leq 2T \}$, 
\bear\label{s-e-1-0}
\int_{Cond} x_1'\cdots x_{n_0}'\cdot x_{1,1}\cdots x_{q,n_q} dx_1'\cdots dx_{n_0}' dx_{1,1}\cdots dx_{q,n_q}< e^{2T}.  
\eear

\noindent Since $\frac{1}{2}\cdot(\sum_{j=1}^{n_0}x_j')+\frac{3}{4}\cdot(\sum_{i=1}^q(x_{i,1} + \cdots +x_{i,n_i}))\leq \frac{3}{4} \ell(\partial Y)\leq \frac{3}{2}T$ on $Cond $, we combine \eqref{single-e-1} and \eqref{s-e-1-0} to get
\bear
&&\int_{\sM_g}\sum_{Y\in \Mod_g \cdot Y_0}e^{-\frac{\ell(\partial Y)}{4}}\textbf{1}_{[0,2T]}(\ell(\partial Y))dX \nonumber\\
&& \prec \frac{1}{n_0!n_1!\cdots n_q!} V_{g_0,k} V_{g_1,n_1}\cdots V_{g_q,n_q}  \times (e^{\frac{3}{2}T} e^{2T})\nonumber \\
&&=   e^{\frac{7}{2}T} \frac{1}{n_0!n_1!\cdots n_q!} V_{g_0,k} V_{g_1,n_1}\cdots V_{g_q,n_q} \nonumber
\eear
as desired.
\ep

Next we bound the integral of $e^{-\frac{\ell(\partial Y)}{4}}\textbf{1}_{[0,2T]}(\ell(\partial Y))$ over $\sM_g$ when the topological type of $Y$ is fixed. Assume that $Y\cong S_{g_0,k}$ meaning that $Y$ is homeomorphic to $S_{g_0,k}$ (\eg see Figure \ref{S03}).
\begin{figure}[ht]
	\centering
	\includegraphics[width=12cm]{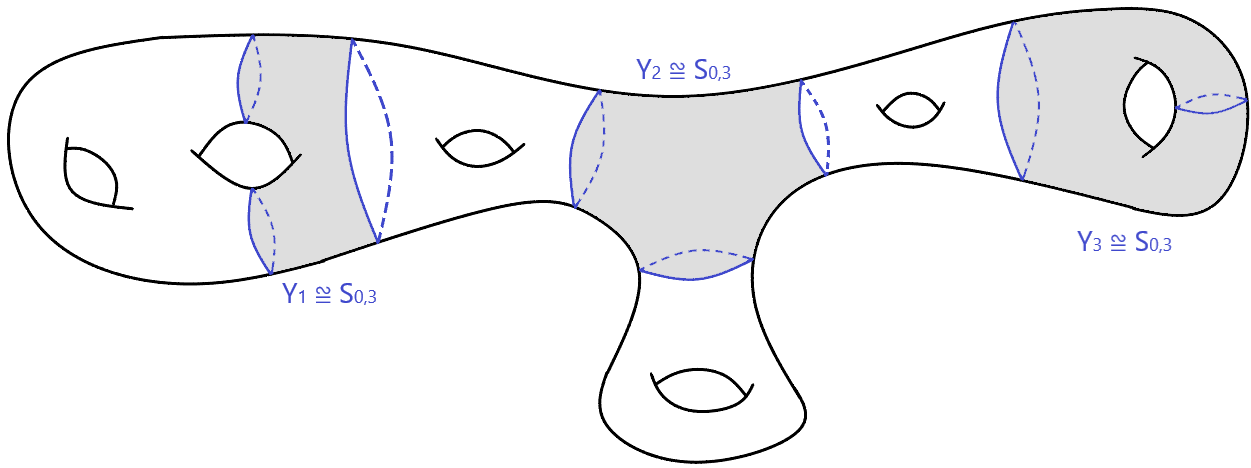}
	\caption{$Y_1\cong Y_2 \cong Y_3 \cong S_{0,3}$, but they are in different orbits}
	\label{S03}
\end{figure}

\noindent Then we let $g_0$ and $k$ be fixed in Assumption ($\star$), and let $q\geq 1$, $\{g_i\}_{1\leq i \leq q}$ and $\{n_i\}_{1\leq i \leq q}$ vary. We list several facts for $Y_0\in \rm{Sub_T(X)}$ satisfying Assumption ($\star$) which will be applied later.
	\begin{itemize}
       \item $Y_0$ is homeomorphic to $S_{g_0,k}$ for some fixed $g_0\geq 0$ and $k>0$ with $m=|\chi(Y_0)|= 2g_0-2+k \geq 1$;
       \item by Gauss-Bonnet $2\pi m=2\pi (2g_0-2+k)\leq 4T$;
		\item $\sum_{i=1}^q n_i=k-2n_0$. In particular $n_0$ is determined by $\{n_i\}_{1\leq i \leq q}$;
       \item $\sum_{i=1}^q(2g_i-2+n_i)=2g-2-m=2g-2g_0-k$. In particular, for large enough $g>1$, $2g-2g_0-k\geq 2g-2-\frac{4T}{2\pi}>g$ because $T=a\ln(g)$.
	\end{itemize}

Recall that as in Section \ref{section wp volume} for all $r\geq 1$,
$$
W_{r}=
\begin{cases}
V_{\frac{r}{2}+1}&\text{if $r$ is even},\\[5pt]
V_{\frac{r+1}{2},1}&\text{if $r$ is odd}.
\end{cases}
$$

\begin{proposition}\label{upp-orbit-gk}
Let $g_0\geq 0$ and $k\geq 1$ be fixed with $m=2g_0-2+k\geq 1$. Then we have that as $g\to \infty$,
\beqar 
\int_{\sM_g}\sum\limits_{\substack{Y\in \rm{Sub_T(X)}; \\ Y\cong S_{g_0,k}}} e^{-\frac{\ell(\partial Y)}{4}}\textbf{1}_{[0,2T]}(\ell(\partial Y)) dX \prec  e^{\frac{7}{2}T}W_{2g_0+k-2} W_{2g-2g_0-k}.  \nonumber
\eeqar
\end{proposition}
\bp
It follows by Proposition \ref{upp-orbit} that
\bear\label{total-i-1-0}
&&\int_{\sM_g}\sum\limits_{\substack{Y\in \rm{Sub_T(X)}; \\ Y\cong S_{g_0,k}}} e^{-\frac{\ell(\partial Y)}{4}}\textbf{1}_{[0,2T]}(\ell(\partial Y)) dX  \\
&&\prec \sum_q\sum_{n_1,\cdots,n_q}\sum_{g_1,\cdots,g_q}  e^{\frac{7}{2}T} \frac{1}{n_0!n_1!\cdots n_q!} V_{g_0,k} V_{g_1,n_1}\cdots V_{g_q,n_q} \nonumber
\eear
where the summation takes over all possible $q\geq 1$ and $(g_1,n_1),\cdots,(g_q,n_q)$ such that $\sum_{i=1}^q(2g_i-2+n_i)=2g-2-m=2g-2g_0-k$.

For fixed $q\geq 1$ and $\{n_i\}_{1\leq i \leq q} $, it follows by Lemma \ref{sum vol lemma} that
\bear\label{vol-pro-1} 
\sum_{g_1,\cdots,g_q} V_{g_1,n_1}\cdots V_{g_q,n_q}&\leq& c \left(\frac{D}{2g-2g_0-k}\right)^{q-1} W_{2g-2g_0-k}\\
&<&c \left(\frac{D}{g}\right)^{q-1} W_{2g-2g_0-k} \nonumber
\eear
where we apply $2g-2g_0-k>g$ for large enough $g$ in the last inequality. Then by \eqref{total-i-1-0} and \eqref{vol-pro-1} we have
\bear \label{upp-gk-1-0}
&&\int_{\sM_g}\sum\limits_{\substack{Y\in \rm{Sub_T(X)}; \\ Y\cong S_{g_0,k}}} e^{-\frac{\ell(\partial Y)}{4}}\textbf{1}_{[0,2T]}(\ell(\partial Y)) dX  \\
&& \prec e^{\frac{7}{2}T}V_{g_0,k} W_{2g-2g_0-k} \times \left(\sum_q\sum_{n_1,\cdots,n_q} \frac{1}{n_0!n_1!\cdots n_q!}  \left(\frac{D}{g}\right)^{q-1} \right). \nonumber
\eear

Recall that for fixed $k-2n_0 \geq 1$, we always have
\begin{equation*}
\sum_{n_1+..+n_q=k-2n_0,\ n_i\geq0} \frac{(k-2n_0)!}{n_1!\cdots n_q!}  ={\underbrace{(1+1+\cdots+1)}_{q\  \text{times}}}^{k-2n_0} = q^{k-2n_0}.
\end{equation*}
Since $0\leq n_0 \leq [\frac{k-1}{2}]$ and $\frac{q^{k-2n_0}}{(k-2n_0)!}<e^{q}$, we have that for large enough $g>1$,
\bear \label{upp-cons}
&& \sum_q\sum_{n_1,\cdots,n_q} \frac{1}{n_0!n_1!\cdots n_q!}  \left(\frac{D}{g}\right)^{q-1}\\
&=& \sum_q\sum_{0\leq n_0 \leq [\frac{k-1}{2}]} \frac{q^{k-2n_0}}{n_0!(k-2n_0)!}  \left(\frac{D}{g}\right)^{q-1}\nonumber \\
& <&\left(\sum_q e^q \cdot \left(\frac{D}{g}\right)^{q-1}\right) \left( \sum_{0\leq n_0 \leq [\frac{k-1}{2}]} \frac{1}{n_0!}\right) \nonumber \\
&<& \sum_q e \cdot e \cdot  \left(\frac{eD}{g}\right)^{q-1}\nonumber\\
&\prec & 1.\nonumber
\eear
Recall that Part $(1)$ of Lemma \ref{Wr-prop} states that $V_{g,n}\leq c W_{2g-2+n}$ for a universal constant $c>0$. Then we plug \eqref{upp-cons} into \eqref{upp-gk-1-0} to get the conclusion.
\ep

For any $Y\in \rm{Sub_T(X)}$ it is known that $|\chi(Y)|\leq \frac{4T}{2\pi}$. Now we bound the integral of $\sum \limits_{Y\in \rm{Sub_T(X)}} e^{-\frac{\ell(\partial Y)}{4}}\textbf{1}_{[0,2T]}(\ell(\partial Y))$
over $\sM_g$ when $m\leq|\chi(Y)|\leq [\frac{4T}{2\pi}]$ where $m\geq 1$ is a fixed integer. That is, we allow $g_0$ and $k$ in Assumption ($\star$) vary such that $m\leq 2g_0+k-2\leq [\frac{4T}{2\pi}]$.
\begin{proposition}\label{upp-orbit-m}
Let $m\geq 1$ be any fixed integer. Then we have that there exists a constant $c(m)>0$ only depending on $m$ such that as $g\to \infty$,
\beqar 
\int_{\sM_g}\sum\limits_{\substack{Y\in \rm{Sub_T(X)}; \\ m\leq|\chi(Y)|\leq [\frac{4T}{2\pi}]}} e^{-\frac{\ell(\partial Y)}{4}}\textbf{1}_{[0,2T]}(\ell(\partial Y))  dX  \prec  Te^{\frac{7}{2}T} c(m) \frac{V_g}{g^m}.  \nonumber
\eeqar
\end{proposition}
\bp
Let $Y \cong S_{g_0,k}\in \rm{Sub_T(X)}$ with $|\chi(Y)|=2g_0+k-2 \geq m$. In particular we have
\[1\leq k \leq [\frac{4T}{2\pi}]+2.\]
Then it follows by Proposition \ref{upp-orbit-gk} that
\bear \label{31-e-1}
&& \int_{\sM_g}\sum\limits_{\substack{Y\in \rm{Sub_T(X)}; \\ m\leq|\chi(Y)|\leq [\frac{4T}{2\pi}]}} e^{-\frac{\ell(\partial Y)}{4}}\textbf{1}_{[0,2T]}(\ell(\partial Y))  dX \\
&& = \int_{\sM_g} \sum \limits_{ m\leq|\chi(Y)|\leq [\frac{4T}{2\pi}]}\sum\limits_{\substack{Y\in \rm{Sub_T(X)}; \\ Y\cong S_{g_0,k}}}  e^{-\frac{\ell(\partial Y)}{4}}\textbf{1}_{[0,2T]}(\ell(\partial Y))  dX \nonumber\\
&& \prec \sum \limits_{1\leq k \leq [\frac{4T}{2\pi}]+2} \ \sum \limits_{g_0:m\leq 2g_0-2+k\leq [\frac{4T}{2\pi}]} e^{\frac{7}{2}T}W_{2g_0+k-2} W_{2g-2g_0-k}. \nonumber
\eear
By Part $(2)$ of Lemma \ref{Wr-prop} we know that for fixed $k\geq 1$,
\bear\label{31-e-2}
\sum \limits_{g_0: m\leq 2g_0-2+k\leq [\frac{4T}{2\pi}]}W_{2g_0+k-2} W_{2g-2g_0-k}\leq c(m) \frac{1}{(2g-2)^m}W_{2g-2}
\eear
for some constant $c(m)>0$ only depending on $m$. By definition we know that $W_{2g-2}=V_g$. Thus, it follows by \eqref{31-e-1} and \eqref{31-e-2} that
\bear
&&\int_{\sM_g}\sum\limits_{\substack{Y\in \rm{Sub_T(X)}; \\ m\leq|\chi(Y)|\leq [\frac{4T}{2\pi}]}} e^{-\frac{\ell(\partial Y)}{4}}\textbf{1}_{[0,2T]}(\ell(\partial Y))  dX \nonumber \\
&&\prec  e^{\frac{7}{2}T} \left( \sum \limits_{1\leq k \leq [\frac{4T}{2\pi}]+2}  \frac{c(m)}{(2g-2)^m}V_{g}\right) \nonumber \\
&& \prec  T e^{\frac{7}{2}T} c(m) \frac{V_g}{g^m}. \nonumber
\eear

The proof is complete.
\ep

Our aim is to show that when $T=4\ln(g)$, the expected value of the \rm{RHS} in Proposition \ref{V-upp} behaves like $o(g^{1+\eps})$ for any $\eps>0$ as $g\to \infty$. Proposition \ref{upp-orbit-m} implies that if $m\geq 17$, we truly have this property. More precisely, if $T=4\ln(g)$ we have
\[\small{\frac{T^2e^T}{V_g}\int_{\sM_g}\sum\limits_{\substack{Y\in \rm{Sub_T(X)}; \\ m\leq|\chi(Y)|\leq [\frac{4T}{2\pi}]}} e^{-\frac{\ell(\partial Y)}{4}}\textbf{1}_{[0,2T]}(\ell(\partial Y))  dX\prec (\ln(g))^3 g}.\]  
Next we will apply similar ideas in the proof of Proposition \ref{upp-orbit-gk} to show the expected value of the second term of the \rm{RHS} of \eqref{eq-v-upp} in Proposition \ref{V-upp} also has the same property even if $1\leq m \leq 16$.

%%%%%%%%%%%%%%%%%%%%%%%%%%%%%%%%%%%%%%%%%%%%%%%%%%%%%%%%%%%%%%%%%%%%%%%%%%%%%%%%%%%55
%%%%%%%%%%%%%%%%%%%%%%%%%%%%%%%%%%%%%
%%%%%%%%%%%%%%%%%%%%%%%%%%%%%5
%%%%%%%%%%%%%%%%%%%%%%%5

In the sequel, we always assume that $Y_0 \cong S_{g_0,k}\in \rm{Sub_T(X)}$ satisfies Assumption ($\star$) with an additional assumption that
\[1\leq 2g_0+k-2\leq 16.\]
Then $1\leq k\leq 18$ and $0\leq g_0\leq 8$. Actually there are 88 pairs of such integers $(g_0,k)$ satisfying the above inequality, even we will only apply the finiteness. Now we prove
\begin{proposition}\label{16-small-1}
Let $g_0\geq 0$ and $k\geq 1$ be fixed with $m=2g_0-2+k \in [1,16]$. For any $\eps_1>0$, then we have that as $g\to \infty$,
\[\int_{\sM_g}\sum\limits_{\substack{Y\in \rm{Sub_T(X)}; \\ Y\cong S_{g_0,k}}} e^{\frac{T}{2}-\frac{1-\eps_1}{2}\ell(\partial Y)} \textbf{1}_{[0,2T]}(\ell(\partial Y)) dX  \prec T^{66}e^{\frac{T}{2}+\eps_1T}\frac{V_g}{g^{m}}. \]
\end{proposition}

\bp
Let $Y_0 \cong S_{g_0,k}\in \rm{Sub_T(X)}$ satisfying Assumption ($\star$). Recall that the map $Y\mapsto \partial Y$ is injective for $Y\in \rm{Sub_T(X)}$ and large enough $g>1$. Now we apply the Integration Formula of Mirzakhani (see Theorem \ref{Mirz int formula}) to get
\bear \label{int-e-h-1}
&&\int_{\sM_g}\sum_{Y\in \Mod_g \cdot Y_0}e^{\frac{T}{2}-\frac{1-\eps_1}{2}\ell(\partial Y)} \textbf{1}_{[0,2T]}(\ell(\partial Y))   dX \\
&&= \int_{\sM_g}\sum_{\partial Y\in \Mod_g \cdot \partial Y_0}e^{\frac{T}{2}-\frac{1-\eps_1}{2}\ell(\partial Y)} \textbf{1}_{[0,2T]}(\ell(\partial Y)) dX \nonumber\\
&&\leq \frac{1}{|\Sym|} \int_{\R_{\geq0}^{k-n_0}}e^{\frac{T}{2}-\frac{1-\eps_1}{2}\left(\sum_{j=1}^{n_0}2x_j'+\sum_{i=1}^q(x_{i,1} + \cdots +x_{i,n_i})\right)} \nonumber \\
&&\mathbf 1_{[0,2T]}\left(\sum_{j=1}^{n_0}2x_j'+\sum_{i=1}^q(x_{i,1} + \cdots +x_{i,n_i})\right)  \nonumber \\
& & V_{g_0,k}(x_1',x_1',\cdots, x_{n_0}',x_{n_0}',x_{1,1},\cdots,x_{1,n_1}, x_{2,1}, \cdots,x_{q,n_q})\nonumber \\
& & V_{g_1,n_1}(x_{1,1},\cdots,x_{1,n_1})\cdots V_{g_q,n_q}(x_{q,1},\cdots,x_{q,n_q})\nonumber \\
& & x_1'\cdots x_{n_0}'\cdot x_{1,1}\cdots x_{q,n_q} dx_1'\cdots dx_{n_0}' dx_{1,1}\cdots dx_{q,n_q}.  \nonumber
\eear
Now we apply Theorem \ref{Mirz vol lemma 0} of Mirzakhani to $V_{g_0,k}(\cdots \cdots)$ to get that the volume satisfies that $V_{g_0,k}(x_1',x_1',\cdots, x_{n_0}',x_{n_0}',x_{1,1},\cdots,x_{1,n_1}, x_{2,1}, \cdots,x_{q,n_q})$ is a polynomial of degree $6g_0-6+2k$ with coefficients depending on $g_0$ and $k$. Thus there exists a constant $c_1>0$ only depending on $g_0$ and $k$ such that
\be \label{V-upper-1-1}
V_{g_0,k}(x_1',x_1',\cdots, x_{n_0}',x_{n_0}',x_{1,1},\cdots,x_{1,n_1}, x_{2,1}, \cdots,x_{q,n_q}) \leq c_1(1+T^{6g_0-6+2k}).
\ene
In our case since $1\leq 2g_0 -2 +k\leq 16$ holds for only finite $(g_0,k)$'s (actually there are $88$ solutions), one may take $c_1>0$ to be universal. It is clear that the symmetry satisfies
\begin{equation}\label{V-U-1-1-1}
|\Sym| \geq n_0!n_1!\cdots n_q!.
\end{equation}
By Lemma \ref{Mirz vol lemma 1}, we have
\begin{equation}\label{V-U-1-1-2}
V_{g,n}(x_1,\cdots,x_n) \leq e^{\frac{x_1+\cdots +x_n}{2}} V_{g,n}.
\end{equation}
Then we plug \eqref{V-upper-1-1}, \eqref{V-U-1-1-1} and \eqref{V-U-1-1-2} into \eqref{int-e-h-1}, and apply $\ell(\partial Y)\leq 2T$ to get
\bear \label{int-e-h-2-0-0}
&&  \int_{\sM_g}\sum_{Y\in \Mod_g \cdot Y_0}e^{\frac{T}{2}-\frac{1-\eps_1}{2}\ell(\partial Y)} \textbf{1}_{[0,2T]}(\ell(\partial Y))   dX \\
&&\leq \frac{c_1\cdot(1+T^{6g_0-6+2k})}{n_0!n_1!\cdots n_q!} \int_{\R_{\geq0}^{k-n_0}}e^{\frac{T}{2}-\frac{1-\eps_1}{2}\left(\sum_{j=1}^{n_0}2x_j'+\sum_{i=1}^q(x_{i,1} + \cdots +x_{i,n_i})\right)} \nonumber \\
&&\mathbf 1_{[0,2T]}\left(\sum_{j=1}^{n_0}2x_j'+\sum_{i=1}^q(x_{i,1} + \cdots +x_{i,n_i})\right)  \nonumber  \\
&& e^{\frac{\sum_{i=1}^q(x_{i,1} + \cdots +x_{i,n_i})}{2}}V_{g_1,n_1}\cdots V_{g_q, n_q} \nonumber \\
& & x_1'\cdots x_{n_0}'\cdot x_{1,1}\cdots x_{q,n_q} dx_1'\cdots dx_{n_0}' dx_{1,1}\cdots dx_{q,n_q} \nonumber \\
&& \prec  \frac{T^{6g_0-6+2k}}{n_0!n_1!\cdots n_q!}e^{\frac{T}{2}+\eps_1T} V_{g_1,n_1}\cdots V_{g_q, n_q} \nonumber \\
&& \times \int_{Cond} x_1'\cdots x_{n_0}'\cdot x_{1,1}\cdots x_{q,n_q} dx_1'\cdots dx_{n_0}' dx_{1,1}\cdots dx_{q,n_q} \nonumber
\eear
where
\begin{equation*}
Cond:=\{0\leq x_{j}', \ 0\leq x_{i,j}, \  \sum_{j=1}^{n_0}2x_j'+\sum_{i=1}^q\sum_{j=1}^{n_i} x_{i,j} \leq 2T  \}.
\end{equation*}

\noindent Recall that
\[\int_{x_i\geq 0,\ \sum_{j=1}^n x_i \leq s}x_1\cdots x_n dx_1\cdots dx_n=\frac{(s)^{2n}}{(2n)!}\]
which, together with the fact that $n_0+\sum_{i=1}^qn_i=k-n_0$, imply that
\bear \label{m-inte}
 && \int_{Cond} x_1'\cdots x_{n_0}'\cdot x_{1,1}\cdots x_{q,n_q} dx_1'\cdots dx_{n_0}' dx_{1,1}\cdots dx_{q,n_q}\\
&&=\frac{(2T)^{2(k-n_0)}}{2^{2n_0}(2(k-n_0))!} \nonumber\\
&& \prec  T^{2(k-n_0)}. \nonumber
\eear

\noindent Since $6g_0-6+4k= 3m+k \leq 48+18=66$, we plug \eqref{m-inte} into \eqref{int-e-h-2-0-0} to get
\bear \label{int-e-h-2}
&& \int_{\sM_g}\sum_{Y\in \Mod_g \cdot Y_0}e^{\frac{T}{2}-\frac{1-\eps_1}{2}\ell(\partial Y)} \textbf{1}_{[0,2T]}(\ell(\partial Y))   dX \\
&& \prec  \frac{T^{66}e^{\frac{T}{2}+\eps_1T}}{n_0!n_1!\cdots n_q!} V_{g_1,n_1}\cdots V_{g_q, n_q}.\nonumber
\eear
Now let $q\geq 1$ and $\{(g_i,n_i)\}_{1\leq i \leq q}$ vary, it follows by \eqref{int-e-h-2} that 
\bear \label{int-e-h-3} 
&&\int_{\sM_g}\sum\limits_{\substack{Y\in \rm{Sub_T(X)}; \\ Y\cong S_{g_0,k}}} e^{\frac{T}{2}-\frac{1-\eps_1}{2}\ell(\partial Y)} \textbf{1}_{[0,2T]}(\ell(\partial Y)) dX  \\
&& \prec T^{66}e^{\frac{T}{2}+\eps_1T} \sum_q\sum_{n_1,\cdots,n_q}\sum_{g_1,\cdots,g_q}   \frac{1}{n_0!n_1!\cdots n_q!} V_{g_1,n_1}\cdots V_{g_q,n_q}.  \nonumber
\eear
where the summation takes over all possible $q\geq 1$ and $(g_1,n_1),\cdots,(g_q,n_q)$ such that $\sum_{i=1}^q(2g_i-2+n_i)=2g-2-m=2g-2g_0-k$. For fixed $\{n_1, \cdots, n_q\}$, it follows by Proposition \ref{1 over gm} that
\begin{eqnarray}\label{vi-m}
\sum_{g_1,\cdots,g_q} V_{g_1,n_1}\cdots V_{g_q,n_q}\prec  \frac{V_g}{g^{m}}.
\end{eqnarray} 
Since $1\leq q\leq k \leq 18$, $n_i\geq 1$ and $n_1+..+n_q=k-2n_0\leq 18$, we have all $q$ and $n_i$'s are bounded from above by $18$. So we have
\bear \label{upp-cons-m}
\sum_q\sum_{n_1,\cdots,n_q} \frac{1}{n_0!n_1!\cdots n_q!}  \prec  1.
\eear
Combine \eqref{int-e-h-3}, \eqref{vi-m} and \eqref{upp-cons-m} we get
\bear \label{int-e-h-4} 
&&\int_{\sM_g}\sum\limits_{\substack{Y\in \rm{Sub_T(X)}; \\ Y\cong S_{g_0,k}}} e^{\frac{T}{2}-\frac{1-\eps_1}{2}\ell(\partial Y)} \textbf{1}_{[0,2T]}(\ell(\partial Y)) dX \prec T^{66}e^{\frac{T}{2}+\eps_1T}\frac{V_g}{g^{m}}  \nonumber
\eear
as desired.
\ep

As a direct consequence of Proposition \ref{16-small-1} we have
\begin{proposition}\label{16-small-2}
For any $\eps_1>0$, then we have that as $g\to \infty$,
\[\int_{\sM_g}\sum_{\substack{Y\in \rm{Sub_T(X)};\\ 1\leq |\chi(Y)|\leq 16}} e^{\frac{T}{2}-\frac{1-\eps_1}{2}\ell(\partial Y)} \textbf{1}_{[0,2T]}(\ell(\partial Y) dX  \prec T^{66}e^{\frac{T}{2}+\eps_1T}\frac{V_g}{g}. \]
\end{proposition}
\bp
Assume that $Y\cong S_{g_0,k}$. Since $|\chi(Y)|=m=2g_0+k-2\in [1,16]$, both $g_0$ and $k$ are uniformly bounded. By Proposition \ref{16-small-1} we have
\bear
&&\int_{\sM_g}\sum_{\substack{Y\in \rm{Sub_T(X)};\\ 1\leq |\chi(Y)|\leq 16}} e^{\frac{T}{2}-\frac{1-\eps_1}{2}\ell(\partial Y)} \textbf{1}_{[0,2T]}(\ell(\partial Y) dX  \nonumber\\
&& \prec \sum_{m=1}^{16}T^{66}e^{\frac{T}{2}+\eps_1T}\frac{V_g}{g^{m}} \nonumber\\
&& \prec T^{66}e^{\frac{T}{2}+\eps_1T}\frac{V_g}{g} \nonumber
\eear 
as desired.
\ep

\begin{rem*}
In the proofs of Theorem \ref{main} and \ref{main-2} we will apply $T=4\ln(g)$. Then the upper bound $\frac{T^{66}e^{\frac{T}{2}+\eps_1T}}{g}\asymp (\ln(g))^{66} g^{1+4\eps_1}$ as $g\to \infty$.
\end{rem*}

Now we are ready to prove the following effective bound for the integral of Term \rm{V} in \eqref{sel-split} over $\sM_g$.
\begin{theorem}\label{V}
Let $\phi_T$ be the function in Section \ref{section trace}. Then for any $\eps_1>0$, there exists a constant $c_1(\eps_1)>0$ only depending on $\eps_1$ such that as $g\to \infty$,
\[\frac{1}{V_g}\int_{\sM_g}\sum_{\gamma \in \mathcal{P}^{ns}(X)} \frac{\ell_{\gamma}(X)}{2\sinh \left(\frac{\ell_{\gamma}(X)}{2} \right)}\phi_T( \ell_{\gamma}(X))dX \prec \frac{T^3 e^{\frac{9}{2}T}}{g^{17}}+c_1(\eps_1)\frac{T^{67}e^{\frac{T}{2}+\eps_1T}}{g}.\]
\end{theorem}

\bp
The conclusion clearly follows by Proposition \ref{V-upp}, Proposition \ref{upp-orbit-m} for $m=17$ and Proposition \ref{16-small-2}.
\ep

%%%%%%%%%%%%%
\subsection{Endgame for the proof of Theorem \ref{main}}
Now we are ready to prove Theorem \ref{main}.
\bp[Proof of Theorem \ref{main}]
Let $\phi_T$ be the function in Section \ref{section trace} and $T=4\ln(g)$. For any $X\in \sM_g$, Equation \eqref{sel-split} says that 
\bear \label{sel-split-1}
&  \sum_{k=0}^{\infty}\widehat{\phi_T}(r_k(X))=\underbrace{(g-1)\int_{-\infty}^\infty r \widehat{\phi_T}(r)\tanh(\pi r)dr}_{\rm I}\\
&+\underbrace{\sum_{\gamma \in \mathcal P(X)}\sum_{k=2}^\infty\frac{\ell_{\gamma}(X)}{2\sinh \left(\frac{k \ell_{\gamma}(X)}{2} \right)}\phi_T(k \ell_{\gamma}(X))}_{\rm{II}} +\underbrace{\sum_{\gamma \in \mathcal{P}^{s}_{sep}(X)} \frac{\ell_{\gamma}(X)}{2\sinh \left(\frac{\ell_{\gamma}(X)}{2} \right)}\phi_T( \ell_{\gamma}(X))}_{\rm{III}}\nonumber \\
&+\underbrace{\sum_{\gamma \in \mathcal{P}^{s}_{nsep}(X)} \frac{\ell_{\gamma}(X)}{2\sinh \left(\frac{\ell_{\gamma}(X)}{2} \right)}\phi_T( \ell_{\gamma}(X))}_{\rm{IV}} +\underbrace{\sum_{\gamma \in \mathcal{P}^{ns}(X)} \frac{\ell_{\gamma}(X)}{2\sinh \left(\frac{\ell_{\gamma}(X)}{2} \right)}\phi_T( \ell_{\gamma}(X))}_{\rm{V}} \nonumber 
\eear

\noindent Recall that $\widehat{\phi_T}(\cdot)\geq 0$ on $\R \cup \textbf{i}\R$. For any $\eps>0$, we take an integral of \eqref{sel-split-1} over $\sM_g$ and only keep the first two terms in the \rm{LHS} to get
\bear \label{5-term}
 && \quad \widehat{\phi_T}(\frac{\textbf{i}}{2})+\frac{1}{V_g}\int_{\sM_g}\widehat{\phi_T}(r_1(X))dX\\
&&\leq \frac{1}{V_g}\cdot \left( \int_{\sM_g}\rm{I}dX+\int_{\sM_g}\rm{II}dX+\int_{\sM_g}\rm{III}dX+\int_{\sM_g}\rm{IV}dX+\int_{\sM_g}\rm{V}dX\right). \nonumber
\eear 
It follows by Lemma \ref{1+eps} that there exists a constant $C_\eps>0$ depending on $\eps$ and $\phi_0$ such that 
\be \label{r1-16}
\frac{\int_{\sM_g}\widehat{\phi_T}(r_1(X))dX}{V_g}>\Prob\left(X\in \M_g; \ \lmdx\leq\frac{3}{16}-\eps \right)\cdot \left( C_\eps g^{1+C_\eps}  \ln(g) \right)
\ene
which together with \eqref{5-term} imply that
\bear
&&\Prob\left(X\in \M_g; \ \lmdx\leq\frac{3}{16}-\eps \right)\\
&&<\frac{ \left( \int_{\sM_g}\rm{I}dX+\int_{\sM_g}\rm{II}dX+\int_{\sM_g}\rm{III}dX+\int_{\sM_g}\rm{V}dX \right)}{C_\eps g^{1+C_\eps}  \ln(g)V_g} \nonumber\\
&&+ \frac{\left|\frac{1}{V_g}\int_{\sM_g}\rm{IV}dX- \widehat{\phi_T}(\frac{\textbf{i}}{2})\right|}{C_\eps \ln(g) g^{1+C_\eps}}. \nonumber
\eear
For any $\eps_1>0$, let $c(\eps_1)>0$ be the constant in Theorem \ref{V}. Recall that $T=4\ln(g)$. Then it follows by Proposition \ref{I}, Proposition \ref{II}, Proposition \ref{III}, Proposition \ref{IV} and Theorem \ref{V} that

\bear \label{mtm-e-1}
&& \quad \ \Prob\left(X\in \M_g; \ \lmdx\leq\frac{3}{16}-\eps \right) \\
&&\prec \frac{\frac{g}{\ln(g)}+g (\ln(g))^2+g+ \left(g(\ln(g))^3+c_1(\eps_1)g^{1+4\eps_1}(\ln(g))^{67} \right) }{C_\eps \ln(g) g^{1+C_\eps}  } \nonumber\\
&& + \frac{g (\ln(g))^2}{C_\eps \ln(g)g^{1+C_\eps}} \nonumber\\
&&\prec \frac{g(\ln(g))^3+c_1(\eps_1)g^{1+4\eps_1}(\ln(g))^{67} }{C_\eps \ln(g) g^{1+C_\eps}  }. \nonumber
\eear
Now we choose $\eps_1>0$ such that $\eps_1<\frac{C_\epsilon}{4}$ and let $g\to \infty$ in \eqref{mtm-e-1} we get
\[\lim \limits_{g\to \infty}\Prob\left(X\in \M_g; \ \lmdx\leq\frac{3}{16}-\eps \right)=0\]
which clearly implies the conclusion. 

The proof is complete.
\ep
%%%%%%%%%

\subsection{Endgame for the proof of Theorem \ref{main-2}}
Now we are ready to prove Theorem \ref{main-2}, which is almost the same as the proof of Theorem \ref{main}.
\bp[Proof of Theorem \ref{main-2}]
For any $X\in \sM_g$, we let \[0=\lambda_0(X)<\lambda_1(X)\leq \cdots \leq\lambda_{s(X)}(X)\leq \frac{1}{4}\] denote the collection of all eigenvalues of $X$ at most $\frac{1}{4}$ counted with multiplicity. For each eigenvalue $\lambda_j(X)$ we write $\lambda_j(X)=s_j(X)(1-s_j(X))$ for some $s_j(X) \in [\frac{1}{2},1]$. In Selberg's trace formula Theorem \ref{sel}, the corresponding quantity satisfies that $$r_j(X)=\textbf{i}(s_j(X)-\frac{1}{2})$$ for each $0\leq j \leq s(X)$. Now we consider any eigenvalue $\lambda_j(X)$ with $\lambda_j(X)=s_j(X)(1-s_j(X))<\sigma(1-\sigma)$. First we have $s_j(X)>\sigma$. Recall that $N_\sigma(X)=\#\{1\leq j \leq s(X); \ \lambda_j(X)<\sigma(1-\sigma)\}$. Let $\phi_T$ be the function in Section \ref{section trace} and $T=4\ln(g)$. By Lemma \ref{1+eps} for any $\eps>0$ and large enough $g$,
\bear \label{sum-small-2}
\sum_{j=1}^{N_\sigma(X)} \widehat{\phi_T}(r_j(X))&=&\sum_{j=1}^{N_\sigma(X)} \widehat{\phi_T}\left(\textbf{i}\left(s_j(X)-\frac{1}{2}\right)\right)\\
&>& \sum_{j=1}^{N_\sigma(X)} C_\eps T e^{T(1-\eps)(\sigma-\frac{1}{2})} \nonumber \\
&=&  4C_\eps \ln(g) g^{4(1-\eps)(\sigma-\frac{1}{2})}N_\sigma(X) \nonumber \\
&\geq &  g^{4(1-\eps)(\sigma-\frac{1}{2})}N_\sigma(X). \nonumber
\eear 
Similar as in the proof of Theorem \ref{main}, we take an integral of \eqref{sel-split-1} over $\sM_g$ and keep more terms in the \rm{LHS} to get
\bear \label{5-term-2}
 && \quad \widehat{\phi_T}(\frac{\textbf{i}}{2})+\frac{1}{V_g}\int_{\sM_g}\left(\sum_{j=1}^{N_\sigma(X)} \widehat{\phi_T}(r_i(X))\right)dX \nonumber\\
&&\leq \frac{1}{V_g}\cdot \left( \int_{\sM_g}\rm{I}dX+\int_{\sM_g}\rm{II}dX+\int_{\sM_g}\rm{III}dX+\int_{\sM_g}\rm{IV}dX+\int_{\sM_g}\rm{V}dX\right). \nonumber
\eear
Which together with Proposition \ref{I}, Proposition \ref{II}, Proposition \ref{III}, Proposition \ref{IV} and Theorem \ref{V} imply that for any $\eps_1>0$ there exists a uniform constant $C>0$ such that for large enough $g>0$,
\bear \label{sum-exp-up-2}
&&\frac{1}{V_g}\int_{\sM_g}\left(\sum_{j=1}^{N_\sigma(X)} \widehat{\phi_T}(r_i(X))\right)dX  \\
&& \leq C\cdot\left( \frac{g}{\ln(g)}+g (\ln(g))^2+g+ \left(g(\ln(g))^3+c_1(\eps_1)g^{1+4\eps_1}(\ln(g))^{67}\right)\right) \nonumber \\
&&\leq g^{1+4\eps_1}(\ln(g))^{68} \nonumber
\eear
where $c_1(\eps_1)>0$ is the constant in Theorem \ref{V}. Now combine \eqref{sum-small-2} and \eqref{sum-exp-up-2}, by Markov's inquality we have that for large enough $g>0$,
\bear \label{eps1-eps}
\Prob\left( X\in \M_g; \ g^{4(1-\eps)(\sigma-\frac{1}{2})} N_\sigma(X) > g^{1+\eps}\right)\leq \frac{(\ln(g))^{68}}{g^{\eps-4\eps_1}}.
\eear
Recall that $\eps_1>0$ is arbitrary. Now one may choose $\eps_1>0$ with $4\eps_1<\eps$. Then by \eqref{eps1-eps} we have
\bear \label{eps1-eps-2}
\lim \limits_{g\to \infty}\Prob\left( X\in \M_g; \  N_\sigma(X) > g^{1+\eps-4(1-\eps)(\sigma-\frac{1}{2})}\right)=0.
\eear
Since $1+\eps-4(1-\eps)(\sigma-\frac{1}{2})\leq 3-4\sigma+3\eps$, we get
\[\lim \limits_{g\to \infty} \Prob\left( X\in \M_g; \ N_\sigma(X) \leq g^{3-4\sigma+\epsilon}\right)=1\]
as desired. 
\ep

\subsection{Proof of Theorem \ref{mt-diam}} Let $X\in \sM_g$ be a hyperbolic surface of genus $g$ and $\inj(X)$ be the injectivity radius of $X$. Magee in \cite{Magee-20} observes the following upper bound on diameter of $X$ in terms of $\lambda_1(X)$ and $\inj(X)$. More precisely,
\begin{proposition}[Magee] \label{up-b-d}
Let $X\in \sM_g$ satisfying $\lambda_1(X)\geq \frac{1-\delta^2}{4}$ and $\inj(X)\geq c_0$ for some $\delta \in (0,1)$ and $c_0>0$. Then there exists a universal constant $C>0$ such that
\[\diam(X)\leq 2c_0+\frac{2}{1-\delta}\left(\ln \left( \frac{C g}{\sinh (c_0/2)^2} \right) +2\ln \ln \left( \frac{C g}{\sinh (c_0/2)^2} \right)\right)\]
\end{proposition}

Now we are ready to prove Theorem \ref{mt-diam}.
\bp[Proof of Theorem \ref{mt-diam}]
First by \cite[Page 292]{Mirz13} we know that 
\be \label{lng-p}
\lim \limits_{g\to \infty} \Prob\left(X\in \M_g; \ \inj(X)\geq \frac{1}{\ln(g)}\right)=1.
\ene
For any $\eps>0$, we set
\[\mathcal{A}_g\overset{\text{def}}{=}\left\{X\in \M_g; \ \lmdx>\frac{3}{16}-\eps \ \text{and} \ \inj(X)\geq \frac{1}{\ln(g)}  \right\}.\]
By Theorem \ref{main} and \eqref{lng-p} we know that
\be \label{lng-p-u}
\lim \limits_{g\to \infty} \Prob\left(X\in \mathcal{A}_g\right)=1.
\ene
Now we complete the proof by showing that $\diam(X)<(4+\eps)\ln(g)$ for any $X\in \mathcal{A}_g$ and large enough $g$. Choose $\delta=\sqrt{\frac{1}{4}+4\eps}$ and $c_0=\frac{1}{\ln(g)}$. Recall that $\sinh(x)\geq x$ for $x\geq 0$. Then it follows by Proposition \ref{up-b-d} that for all $X\in \mathcal{A}_g$ and $g$ large enough, 
\beqar
\diam(X)
&\leq& \frac{2}{\ln(g)}+\frac{2}{1-\delta}\left( \ln\left( \frac{C g}{\sinh (\frac{1}{2\ln g})^2} \right) + 2\ln\ln\left( \frac{C g}{\sinh (\frac{1}{2\ln g})^2} \right) \right) \\
&\leq& \frac{2}{\ln(g)}+\frac{2}{1-\delta}\left( \ln\left(4Cg(\ln g)^2 \right) + 2\ln\ln\left( 4Cg(\ln g)^2 \right) \right) \\ 
&\leq& \frac{2}{\ln(g)}+\frac{2}{1-\delta}\left( \ln g + \ln(4C) + 2\ln\ln g + 2\ln\ln(g^2) \right) \\
&=& \frac{2}{\ln(g)}+\frac{2}{1-\delta}\left(\ln g+4\ln \ln g+\ln(16C) \right).
\eeqar 
Then the conclusion follows since $\eps>0$ is arbitrary.
\ep

%%%%%%%%%%%%%%%%%%%%%%%%%%%%

\section{A new counting result for filling closed geodesics}\label{section-new count}
In this section we prove Theorem \ref{sec-count}, which relies on the following technical result.
\begin{theorem}\label{thm filling curve decrease}
There exists a universal constant $L_0>0$ such that for any $0<\eps_1<1$, $m=2g-2+n\geq 1$, $\Delta\geq 0$ and $\sum_{i=1}^n x_i \geq \Delta +L_0\cdot\frac{m\cdot n}{\eps_1}$, the following holds: for any hyperbolic surface $X\in \T_{g,n}(x_1,\cdots,x_n)$, one can always find a new hyperbolic surface $Y\in \T_{g,n}(y_1,\cdots,y_n)$ satisfying
\begin{enumerate}
\item $y_i\leq x_i$ for all $1\leq i\leq n$;
\item $\sum_{i=1}^n x_i - \sum_{i=1}^n y_i = \Delta$;
\item for all filling curve $\eta$ in $S_{g,n}$, we have
$$\ell_\eta(X) - \ell_\eta(Y) \geq \frac{1}{2}(1-\eps_1)\Delta.$$ 
\end{enumerate}
\end{theorem}

Now we prove Theorem \ref{sec-count} assuming Theorem \ref{thm filling curve decrease}.
\bp[Proof of Theorem \ref{sec-count}]
Let $X\in  \T_{g,n}(x_1,\cdots,x_n)$ for certain $x_i$'s. 

If the total boundary length $\sum_{i=1}^n x_i \geq L_0\frac{mn}{\eps_1}$, by Theorem \ref{thm filling curve decrease} one may take $\Delta = \sum x_i - L_0\frac{mn}{\eps_1}$ and then get a hyperbolic surface $Y\in\T_{g,n}(y_1,\cdots,y_n)$ such that $\sum y_i = L_0\frac{mn}{\eps_1}$ and for any filling curve $\eta \subset S_{g,n}$ we have
$$\ell_\eta (X) - \ell_\eta(Y) \geq \frac{1-\eps_1}{2}\cdot \left(\sum_{i=1}^n x_i - L_0\frac{mn}{\eps_1}\right).$$
Thus 
$$\#_f(X,T) \leq \#_f\left(Y,T-\frac{1-\eps_1}{2}\left(\sum_{i=1}^n x_i - L_0\frac{mn}{\eps_1}\right)\right) $$
which together with Lemma \ref{count-3} and the fact that $1\leq n \leq m+2$ implies that
\begin{eqnarray}
\#_f(X,T) &\leq& \frac{1}{2}m \cdot e^{T-\frac{1-\eps_1}{2}(\sum x_i - L_0\frac{mn}{\eps_1}) +6} \\
&=& \frac{m}{2}\cdot e^{\frac{(1-\eps_1)L_0 mn}{2\eps_1}+6} \cdot e^{T-\frac{1-\eps_1}{2}\sum x_i}\nonumber\\
&\leq & c(\eps_1,m) \cdot e^{T-\frac{1-\eps_1}{2}\sum x_i} \nonumber
\end{eqnarray}
where
\[c(\eps_1,m)=\frac{m}{2}\cdot e^{\frac{(1-\eps_1)L_0 m(m+2)}{2\eps_1}+6}.\]

If the total boundary length $\sum_{i=1}^n x_i \leq L_0\frac{mn}{\eps_1}$, then by Lemma \ref{count-3} we have 
\begin{eqnarray}
\#_f(X,T) &\leq& \frac{1}{2}m \cdot e^{T+6} \\
&=& \frac{m}{2}\cdot e^{\frac{1-\eps_1}{2}L_0\frac{mn}{\eps_1}+6} \cdot e^{T-\frac{1-\eps_1}{2}L_0\frac{mn}{\eps_1}} \nonumber \\
&\leq& c(\eps_1,m) \cdot e^{T-\frac{1-\eps_1}{2}\sum x_i}. \nonumber
\end{eqnarray}

The proof is complete.
\ep

\begin{rem*}
The proof of Theorem \ref{thm filling curve decrease} requires some technical assumptions for the total boundary length, which is enough for us to prove Theorem \ref{main} and \ref{main-2}. It would be \emph{interesting} to know whether Theorem \ref{sec-count} holds for $\eps_1=0$ and a uniform lower bound of the total boundary length $\sum x_i$. More precisely, \begin{question}
	For any surface $X\in \T_{g,n}(x_1,...,x_n)$, does the following holds:
	\begin{equation*}
		\#_f(X,T)\leq \frac{1}{2}(2g-2+n) \cdot e^{T-\frac{1}{2}\sum_{i=1}^n x_i +6}?
	\end{equation*}
\end{question} 
\end{rem*}

The proof of Theorem \ref{thm filling curve decrease} is technical. We first briefly explain the strategy as follows.

\noindent \textbf{Strategy on the proof of Theorem \ref{thm filling curve decrease}.} 
For any given hyperbolic surface $X\in \T_{g,n}(x_1,...,x_n)$, we first find a special pair of pants $\mathcal P \subset X$ with one or two boundary closed geodesics in $\partial X$ (or three if $X=\mathcal P \cong S_{0,3}$). Then we reduce the length of these boundary curves in $\mathcal P$ by certain $\delta>0$ and do not change the remained part $X\setminus \mathcal P$ to obtain a new hyperbolic surface $X_\delta$. For any filling closed geodesic $\eta$ in $X$, it must intersect with the boundary of $\mathcal P$ (if $X \neq \mathcal{P}$). We fix the part of $\eta$ in $X\setminus \mathcal P$. And for each segment $J$ of $\eta\cap\mathcal P$, we replace $J$ to a geodesic segment in $X_\delta$ which has the same endpoints and is in the same homotopy class (with endpoints fixed) as $J$. Then we get a closed piecewise geodesic $\eta'$ in $X_\delta$ which is in the same homotopy class as $\eta$ (e.g. see Figure \ref{fig eta'}). Clearly, $\ell(\eta') \geq \ell_\eta(X_\delta)$. Then we will show that $\ell_\eta (X) - \ell(\eta') \geq \frac{1}{2}(1-\eps_1)\delta$ if the total boundary length of $X$ is large enough, and hence $\ell_\eta (X) - \ell_\eta(X_\delta) \geq \frac{1}{2}(1-\eps_1)\delta$. Then Theorem \ref{thm filling curve decrease} follows by repeating the process above by finite times.  \\

 We make the following notations throughout this section.

\noindent \textbf{Notations.} $(1)$. Denote  $m=2g-2+n=|\chi(S_{g,n})|\geq 1$.

\noindent $(2)$. Fix a constant $A>0$. Then for two quantities $h_1,h_2$, we say $h_1=O_A(h_2)$ if $|h_1|\leq A'|h_2|$ for some constant $A'>0$ only depending on $A$. 

\noindent $(3)$. We use the same letters for geodesics and their lengths. \\

Now we start the proof of Theorem \ref{thm filling curve decrease}, which will be split into several parts. 

For a hyperbolic surface $X\in \T_{g,n}(x_1,...,x_n)$, now we always assume the total boundary length 
\be 
\sum_{i=1}^n x_i \geq 2Amn.
\ene 
In particular, the longest boundary $2b$ of $X$ satisfies 
$$2b\geq 2Am.$$ 
Consider the maximal embedded half-collar of boundary $2b$ in $X$ with width $w$. A simple computation shows that the area of this half-collar is equal to $2b\sinh w$, and is bounded from above by $\area(X)=2\pi m$. Thus we have
\begin{equation} \label{w-small}
\sinh w\leq \frac{\pi m}{b}.
\end{equation}

\noindent The maximal embedded half-collar must be one of the following two types.
\begin{enumerate}
	\item \textbf{Type-1:} The maximal half-collar does not touch any other boundary of $X$.

	\item \textbf{Type-2:} The maximal half-collar touches another boundary of $X$.
\end{enumerate}

\begin{con*} [for $\mathcal P$]
 See Figure \ref{fig 2 type pants} for an illustration.
In Type-1, first one may choose a self-tangent point on the boundary of the maximal half-collar of $2b$, then take two-sides perpendiculars to $2b$ inside the half-collar, we obtain a geodesic segment $2w$ orthogonal to $2b$ at  both endpoints. Then the two curves $2b$ and $2w$ uniquely determine a pair of pants, denoted by $\mathcal{P}$. In Type-2, assume that the maximal half-collar touches another boundary $2c$ (if it simultaneously touches two boundary geodesics, we just pick one of them which is denoted by $2c$). Let $\alpha=w$ be the perpendicular between $2b$ and $2c$ inside the half-collar. Then the curves $w,2b$ and $2c$ determine a unique pair of pants, also denoted by $\mathcal{P}$. 
\begin{figure}[ht]
	\centering
	\includegraphics[width=12cm]{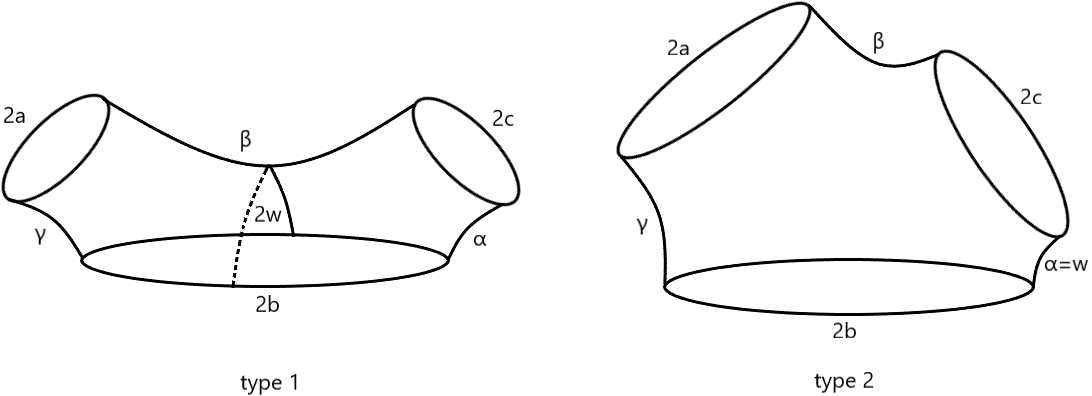}
	\caption{two types of $\mathcal{P}'s$}
	\label{fig 2 type pants}
\end{figure}
\end{con*}

One may cut the pair of pants in Figure \ref{fig 2 type pants} along the shortest perpendiculars between three closed boundary geodesics of $\mathcal{P}$ to get two right-angled hexagons as shown in Figure \ref{fig 2 type hexagon}.
\begin{figure}[ht]
	\centering
	\includegraphics[width=10cm]{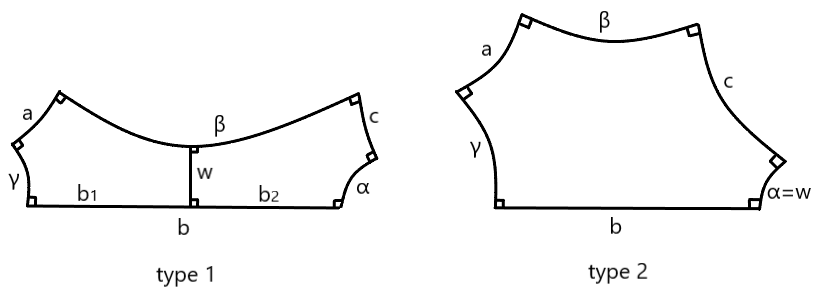}
	\caption{}
	\label{fig 2 type hexagon}
\end{figure}

\begin{rem*}
	In Type-1, $2b$ is a boundary curve of $X$, and $2a,2c$ may or may not be part of the boundary. In Type-2, $2b$ and $2c$ are two boundary curves of $X$ and $2a$ may or may not be.
\end{rem*}

\begin{con*}[for $X_\delta$]Now for certain $\delta>0$, we construct a new hyperbolic surface $X_\delta$ based on $X\in \T_{g,n}(x_1,...,x_n)$ with total boundary length reduced by $\delta$.
From the discussion above, one may first find a pair of pants $\mathcal P \sbs X$ with one boundary curve $2b$.

\noindent \emph{Case-a: $X\neq \mathcal{P}\cong S_{0,3}$}. If $\mathcal P$ is of Type-1 as shown in Figure \ref{fig 2 type pants} above, let $\mathcal P_\delta$ be the pair of pants with boundary lengths $(2a,2b-\delta,2c)$. If $\mathcal P$ is of Type-2 as shown in Figure \ref{fig 2 type pants} above, let $\mathcal P_\delta$ be the pair of pants with boundary lengths $(2a,2b-\frac{1}{2}\delta,2c-\frac{1}{2}\delta)$. Then we glue $\mathcal P_\delta$ and $X\setminus \mathcal P$ together along the same closed geodesics and twists as those when gluing $\mathcal P$ and $X\setminus \mathcal P$ back into $X$. Hence we get a new hyperbolic surface $X_\delta$ with total boundary length reduced by $\delta$. In other words, we fix the part $X\setminus \mathcal P$ and just decrease the length of boundary $2b$ or $(2b,2c)$ to get $X_\delta$. 

\noindent \emph{Case-b: $X= \mathcal{P}\cong S_{0,3}$}. We only decrease $2b$ to $2b-\delta$ in Type-1 and decrease $(2b,2c)$ to $(2b-\frac{1}{2}\delta,2c-\frac{1}{2}\delta)$ in Type-2 as above. 

That is,
$$X_\delta \in \T_{g,n}(x_1-\delta,x_2,\cdots,x_n)$$
if $\mathcal P$ is of Type-1 (assume $x_1=2b$) and 
$$X_\delta \in \T_{g,n}(x_1-\frac{1}{2}\delta,x_2-\frac{1}{2}\delta,x_3,\cdots,x_n)$$
if $\mathcal P$ is of Type-2 (assume $x_1=2b$ and $x_2=2c$).
\end{con*}

Our main task in this section is to show the following result.

\begin{proposition}\label{prop filling curve decrease a little}
There exists a constant $A_0>0$ only depending on $A$ such that the following hold: for any $0< \delta\leq A_0$ and $X\in\T_{g,n}(x_1,...,x_n)$ with $\sum x_i \geq n(2Am+\delta)$, the above construction for $X_\delta$ exists, and moreover for any filling closed curve $\eta$ in $S_{g,n}$, we have 
$$\ell_\eta (X) - \ell_\eta(X_\delta) \geq \frac{1}{2}\left(1-O_A(\frac{mn}{\sum x_i})\right)\delta.$$ 
\end{proposition}

Assuming Proposition \ref{prop filling curve decrease a little} first, we now finish the proof of Theorem \ref{thm filling curve decrease}. 
\begin{proof}[Proof of Theorem \ref{thm filling curve decrease}]
We choose $A=1$ in Proposition \ref{prop filling curve decrease a little}. Assume that the term $|O_A(\frac{mn}{\sum x_i})| \leq C \frac{mn}{\sum x_i}$ in Proposition \ref{prop filling curve decrease a little} for some constant $C>0$. Take 
$$L_0=\max\{C, 2A+A_0\}.$$

\noindent For $0\leq \Delta\leq \sum x_i -L_0\frac{mn}{\eps_1}$, assume $\Delta = k\delta$ where $k>0$ is an integer and $0< \delta\leq A_0$. One may repeat the construction above in $k$ times to find a sequence of pairs of pants, reduce total boundary length for $\delta$ each time, and obtain a sequence of hyperbolic surfaces $X^0,X^1,\cdots,X^k$ where $X^0=X$ and $X^{j+1}=(X^j)_{\delta}$ is the new hyperbolic surface obtained from $X^j$ by the construction above. Then $X^j$ has total boundary length equal to $\left(\sum x_i -j\delta\right)$. In particular $$X^k \in \T_{g,n}(y_1,\cdots,y_n)$$ for some $y_i\leq x_i$ with $\sum x_i - \sum y_i =k\delta= \Delta$. We remark here that one can always apply the construction above to $X^j$ for each $0\leq j\leq k$ since its total boundary length $\sum x_i -j\delta \geq L_0\frac{mn}{\eps_1} \geq (2A+A_0)mn$ which satisfies the assumption. Then it follows by Proposition \ref{prop filling curve decrease a little} that for any filling closed curve $\eta$ in $S_{g,n}$ and each $0\leq j \leq k-1$, 
$$\ell_\eta (X^j) - \ell_\eta(X^{j+1}) \geq \frac{1}{2} \left(1-C\frac{mn}{L_0 mn/\eps_1}\right) \delta \geq \frac{1}{2} (1-\eps_1) \delta$$ 
which implies that
$$\ell_\eta (X) - \ell_\eta(X^k) \geq \frac{1}{2}(1-\eps_1)\Delta.$$
Then the conclusion follows by choosing $Y=X^k$.
\end{proof}

Now we aim to prove Proposition \ref{prop filling curve decrease a little}. 

We first assume the existence of $X_\delta$. Normally it is hard to calculate the length $\ell_\eta(X_\delta)$. Instead, we construct a closed piecewise geodesic $\eta'$ in $X_\delta$ homotopic to $\eta$ and then show that 
$$\ell_\eta (X) - \ell(\eta') \geq \frac{1}{2}\left(1-O_A(\frac{mn}{\sum x_i})\right)\delta$$
implying that
\[\ell_\eta (X) - \ell_\eta(X_\delta) \geq \frac{1}{2}\left(1-O_A(\frac{mn}{\sum x_i})\right)\delta\]
where $\ell(\eta')$ is the length of the closed piecewise geodesic $\eta'$ in $X_\delta$, and clearly we have $\ell(\eta')\geq \ell_\eta(X_\delta)$.

Let $\eta$ be a closed filling curve in $S_{g,n}$ and still use $\eta$ to denote the corresponding closed geodesic representative in $X$. Since $\eta$ is filling, it must intersect with the pair of pants $\mathcal P$ for certain times if $X\neq \mathcal P\cong S_{0,3}$. Let's first assume that $X$ is not $S_{0,3}$. Consider the set of all the intersection points $\eta\cap\partial\mathcal P$ which separates $\eta$ into several segments. Assume that 
\be\label{J-j}
\eta=(\cup I_i)\cup(\cup J_j)
\ene 
where the $I_i$'s are geodesic segments in $X\setminus \mathcal P$ with endpoints on $\partial\mathcal P \cap \partial(X\setminus\mathcal P)$, and the $J_j$'s are geodesic segments in $\mathcal P$ with endpoints on $\partial\mathcal P \cap \partial(X\setminus\mathcal P)$. For each $J_j$, let $J'_j$ be the geodesic segment representative in $\mathcal P_\delta$ (defined in the construction above) which is homotopic to $J_j$ in $\mathcal{P}$ (as a topological space) relative to the two endpoints of $J_j$. In other words, we fix the endpoints of $J_j$ and shrink $J_j$ to a shortest geodesic segment $J'_j$ in $\mathcal P_\delta$. Since $X$ and $X_\delta$ have the same length and twist parameters for those simple closed geodesics $\partial\mathcal P \cap \partial(X\setminus\mathcal P)$, replacing $\mathcal P$ and $J_j$ by $\mathcal P_\delta$ and $J_j'$ respectively, we have that the closed curve
$$\eta'=(\cup I_i)\cup(\cup J'_j)\subset X_\delta$$
is a closed piecewise geodesic homotopic to $\eta$ and 
\begin{equation}\label{eqn eta-eta' = sum J-J'}
\ell_\eta(X)-\ell(\eta')=\sum \left(\ell(J_j)-\ell(J'_j)\right).
\end{equation}
See Figure \ref{fig eta'} for an example. If $X=\mathcal P\cong S_{0,3}$, we just denote $\eta'$ to be the closed geodesic in $X_\delta$ homotopic to $\eta$. We will show that for each $j$, 
\begin{equation}\label{eqn J-J' geq 1/2}
\ell(J_j)-\ell(J'_j) \geq \frac{1}{2}\left(1-O_A(\frac{mn}{\sum x_i})\right)\delta.
\end{equation}

\begin{figure}[ht]
	\centering
	\includegraphics[width=12cm]{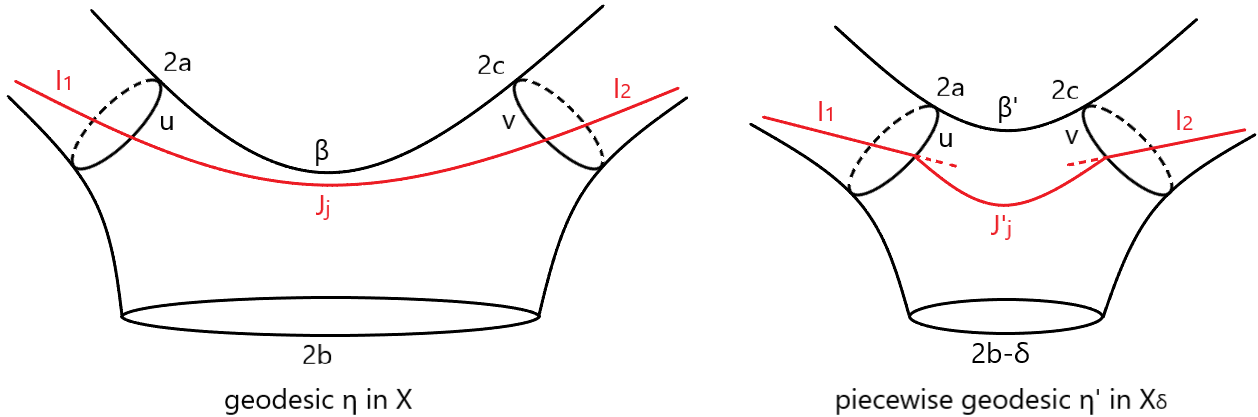}
	\caption{an example for $\eta$ and $\eta'$}
	\label{fig eta'}
\end{figure}

\begin{rem*}
As shown in Figure \ref{fig eta'}, we denote $u$ to be the part of geodesic $2a$ from the endpoint of $J_j$ to $\beta$ (the shortest geodesic between $2a$ and $2c$), and denote $v$ to be the part of geodesic $2c$ from the endpoint of $J_j$ to $\beta$. By saying fixing the endpoints of $J_j$, we just mean to keep the lengths of $u$ and $v$ to be unchanged.
\end{rem*}

\begin{remark}
The existence of $J_j$ is the only place where we apply the assumption that $\eta$ is filling. Actually if $\eta$ is an arbitrary closed curve in $S_{g,n}$, we will show that 
$$\ell_\eta(X)-\ell_\eta(X_\delta) \geq \frac{k}{2}\left(1-O_A(\frac{mn}{\sum x_i})\right)\delta$$
where $k$ is the number of components of $\eta\cap\mathring{\mathcal P}$ where $\mathring{\mathcal P}$ is the interior of $\mathcal P$. 
\end{remark}

To prove \eqref{eqn J-J' geq 1/2}, we separate $J_j$ into several pieces. In the universal covering space of the pair of pants $\mathcal P$, $J_j$ is lifted onto a simple geodesic segment with endpoints on the boundary. See Figure \ref{fig curve in pants} for an example. 
\begin{figure}[ht]
	\centering
	\includegraphics[width=12cm]{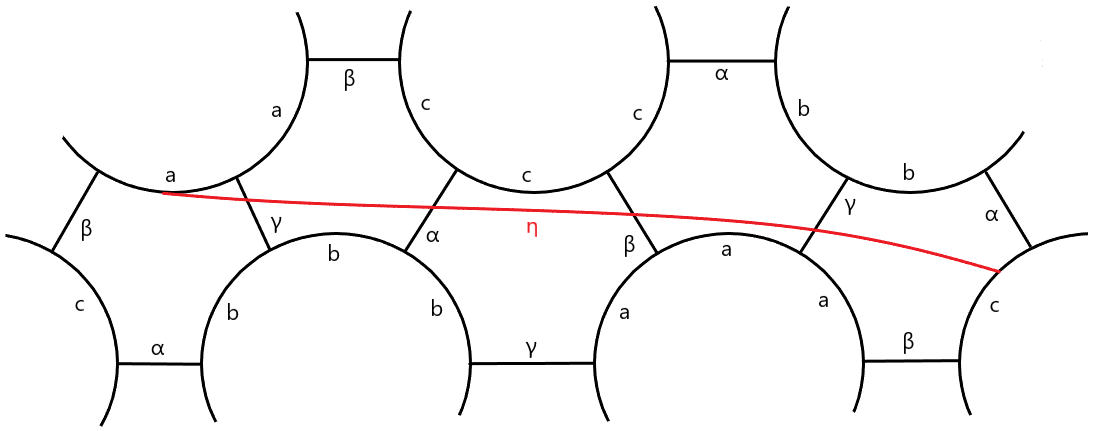}
	\caption{an example for $\eta$ in $\mathcal P$}
	\label{fig curve in pants}
\end{figure}

Now we split the proof of Proposition \ref{prop filling curve decrease a little} into the following subsections.

\subsection{A technical lemma for deforming pairs of pants}
First we recall the following formulas for hyperbolic distance, which one may refer to \cite[Formula Glossary and (2.3.2)]{Buser10}.

\begin{lemma}\label{lem hyp formula}
	In the right-angled hexagon in Figure \ref{fig 3 polygon}, we have
	$$\cosh c = \sinh a\sinh b\cosh \gamma - \cosh a\cosh b,$$
	$$\frac{\sinh\alpha}{\sinh a} = \frac{\sinh\beta}{\sinh b} = \frac{\sinh\gamma}{\sinh c}.$$
	
	In the right-angled pentagon in Figure \ref{fig 3 polygon}, we have
	$$\cosh c = \sinh a\sinh b.$$
	
	In the rectangle with right angle between $x,a$ and right angle between $y,a$ in Figure \ref{fig 3 polygon}, we have 
	$$\cosh \alpha = \cosh x\cosh y\cosh a - \sinh x\sinh y.$$
	\begin{figure}[ht]
		\centering
		\includegraphics[width=12cm]{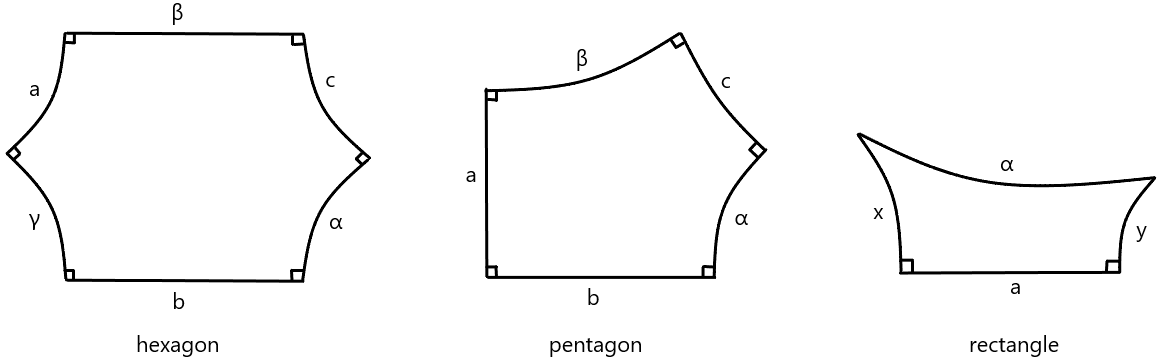}
		\caption{}
		\label{fig 3 polygon}
	\end{figure}
\end{lemma}

\begin{rem*}
	In the rectangle case above, the same formula still holds when $x$ or $y$ may be negative. Where for $x\cdot y<0$, it means that the edges $x$ and $y$ are on the different sides of $a$. In this case $\alpha$ will intersect with $a$.
\end{rem*}

In the rectangle case above, we provide the following technical lemma which will be repeatedly applied later. 

\begin{lemma}\label{lem derivative in rectangle}
	In the rectangle as shown in Figure \ref{fig 3 polygon} above, assume that $a,x,y$ and  $\alpha$ are smooth functions parametrized on the same domain by the variable $t$. Then we have
	\begin{eqnarray*}
		&& \left|\dot\alpha \tanh\alpha - \dot a\tanh a - \dot x\tanh x - \dot y\tanh y \right| \\
		&\leq& \frac{1}{\cosh a -1}\left(|\dot a|\tanh a + \frac{|\dot x|}{\cosh^2 x} + \frac{|\dot y|}{\cosh^2 y}\right).
	\end{eqnarray*}
In particular, if $\alpha\geq a\geq C$ for some constant $C>0$, then we have
\[\left| \dot\alpha-\dot a  \right| \leq C'_C \cdot \left( e^{-a}|\dot a| + |\dot x|+|\dot y| \right) \]
for some constant $C'_C>0$ only depending on $C$. 
\end{lemma}

\begin{proof}
	By Lemma \ref{lem hyp formula} we have 
	\begin{eqnarray}\label{cosha-l}
		\cosh \alpha &=& \cosh x\cosh y\cosh a - \sinh x\sinh y  \\
		&=& \cosh x\cosh y(\cosh a-1) + \cosh(x-y). \nonumber
	\end{eqnarray}

	Taking derivative we have
	\begin{eqnarray}
		\frac{\sinh \alpha}{\cosh \alpha} \dot \alpha 
		&=& \dot a \frac{\cosh x\cosh y\sinh a}{\cosh \alpha} \\
		&& +\dot x \frac{\sinh x\cosh y(\cosh a-1)}{\cosh \alpha} \nonumber\\
		&& +\dot y \frac{\cosh x\sinh y(\cosh a-1)}{\cosh \alpha} \nonumber\\
		&& +(\dot x-\dot y) \frac{\sinh(x-y)}{\cosh \alpha} \nonumber\\
		&=& \dot a \frac{\cosh x\cosh y\sinh a}{\cosh \alpha} \nonumber\\
		&& +\dot x \frac{\sinh x\cosh y\cosh a - \cosh x\sinh y}{\cosh \alpha} \nonumber\\
		&& +\dot y \frac{\cosh x\sinh y\cosh a - \sinh x\cosh y}{\cosh \alpha}.\nonumber
	\end{eqnarray}
	
	Now we estimate the difference. From \eqref{cosha-l}, we have 
	\begin{eqnarray}
		&& \left|\frac{\cosh x\cosh y\sinh a}{\cosh \alpha} -\frac{\sinh a}{\cosh a}\right| \\
		&=& \frac{\tanh a \sinh x\sinh y}{\cosh \alpha} \nonumber\\
		&\leq& \frac{\tanh a}{\cosh a -1}, \nonumber
	\end{eqnarray}
	
	\begin{eqnarray}
		&& \left|\frac{\sinh x\cosh y\cosh a - \cosh x\sinh y}{\cosh \alpha} -\frac{\sinh x}{\cosh x}\right| \\
		&=& \frac{\sinh y}{\cosh x\cosh \alpha}\nonumber \\
		&\leq& \frac{1}{\cosh^2 x(\cosh a -1)},\nonumber
	\end{eqnarray}

and
	
	\begin{eqnarray}
		&& \left|\frac{\cosh x\sinh y\cosh a - \sinh x\cosh y}{\cosh \alpha} -\frac{\sinh y}{\cosh y}\right| \\
		&=& \frac{\sinh x}{\cosh y\cosh \alpha}\nonumber \\
		&\leq& \frac{1}{\cosh^2 y(\cosh a -1)}.\nonumber
	\end{eqnarray}
	
	Put all the inequalities above together we get
	\begin{eqnarray*}
		&& |\dot\alpha \tanh\alpha - \dot a\tanh a - \dot x\tanh x - \dot y\tanh y | \\
		&\leq& \frac{1}{\cosh a -1}(|\dot a|\tanh a + \frac{|\dot x|}{\cosh^2 x} + \frac{|\dot y|}{\cosh^2 y})
	\end{eqnarray*}
which completes the proof of the first part.

The second part follows by the first part:
\begin{eqnarray*}
|\dot\alpha \tanh\alpha - \dot a\tanh a|
&\leq& |\dot\alpha \tanh\alpha - \dot a\tanh a - \dot x\tanh x - \dot y\tanh y | \\
&& + |\dot x\tanh x| + |\dot y\tanh y| \\
&\leq& \frac{|\dot a|+|\dot x|+|\dot y|}{\cosh a -1} + |\dot x|+|\dot y| \\
&=& \frac{|\dot a|}{\cosh a -1} + (|\dot x|+|\dot y|) \cdot \frac{\cosh a }{\cosh a -1} .
\end{eqnarray*}
Since $\alpha \geq a \geq C$, we have
\begin{eqnarray*}
|\dot\alpha - \dot a|
&\leq& \frac{1}{\tanh\alpha}|\dot\alpha \tanh\alpha - \dot a\tanh a| + |\dot a(1-\frac{\tanh a}{\tanh\alpha})| \\
&\leq& \frac{|\dot a|}{\tanh a(\cosh a -1)} + \frac{(|\dot x|+|\dot y|)\cosh a}{\tanh a(\cosh a -1)} + |\dot a|(1-\tanh a) \\
&\leq& C'_C \left( e^{-a}|\dot a| + |\dot x|+|\dot y| \right) 
\end{eqnarray*} 
for some constant $C'_C>0$ only depending on $C$. 
\end{proof}

Next we will prove Proposition \ref{prop filling curve decrease a little} for Type-1 and Type-2 separately in the following two subsections.
%%%%%%%%%%%%%%%
\subsection{Proof of Proposition \ref{prop filling curve decrease a little} for Type-1}
Now we consider a pair of pants in Figure \ref{fig 2 type pants} of Type-1 and the corresponding right-angled hexagon in Figure \ref{fig 2 type hexagon} of Type-1. 

In this subsection, we always use the notation $w$ to be the perpendicular between $b$ and $\beta$ in Figure \ref{fig 2 type hexagon} of Type-1. For the pants $\mathcal P$ of Type-1 we construct, \eqref{w-small} holds. But when reducing $b$, the number $w$ may increase such that \eqref{w-small} does not hold again. Instead, we always assume $w$ satisfies
\begin{equation}\label{w-small-1}
\sinh w \leq 2\pi\frac{m}{b}.
\end{equation}
Later in the proof of Proposition \ref{key-pro-1} we will show \eqref{w-small-1} always holds when $b$ does not reduce too much.

Let $b_1,b_2>0$ as shown in Figure \ref{fig 2 type hexagon}. By Lemma \ref{lem hyp formula} we have
$$\cosh a= \sinh b_1 \sinh w$$ and
$$\cosh c= \sinh b_2 \sinh w.$$
So we have 
\begin{eqnarray}\label{restriction type 1}
\cosh a\cosh c & = & \sinh^2 w \sinh b_1\sinh b_2  \\
& = & \frac{1}{2}\sinh^2 w (\cosh(b_1+b_2)-\cosh(b_1-b_2))  \nonumber\\
& \leq & \frac{1}{2}\sinh^2 w \cosh b. \nonumber
\end{eqnarray}
Now we consider reducing the length of boundary $2b$ of $\mathcal P$. Let $\{a(t),b(t),c(t)\}$ be the length functions in terms of a common parameter $t$. In our process from $\mathcal P$ to $\mathcal P_\delta$, we assume $a$ and $c$ are fixed and $b$ decreases with constant speed. More precisely, denote $\dot b = \frac{db}{dt}$ and assume
\be
\dot a =\dot c=0 \quad \emph{and} \quad \dot b =-1.
\ene

\noindent As shown in Figure \ref{fig 2 type pants}, in the pair of pants with boundary length $(2a,2b,2c)$, we always denote $\alpha$ to be the shortest perpendicular between $2b$ and $2c$, denote $\beta$ to be the shortest perpendicular between $2a$ and $2c$, and denote $\gamma$ to be the shortest perpendicular between $2a$ and $2b$. One may also see Figure \ref{fig 3 polygon} for the corresponding hexagon. Then by Lemma \ref{lem hyp formula}, we have
\be \label{alpha}
\cosh \alpha = \frac{\cosh a+\cosh b\cosh c}{\sinh b\sinh c},
\ene
\be \label{gamma}
\cosh \gamma = \frac{\cosh c+\cosh a\cosh b}{\sinh a\sinh b},
\ene
\be \label{beta}
\cosh \beta = \frac{\cosh b+\cosh a\cosh c}{\sinh a\sinh c},
\ene
and
\begin{equation}\label{sinh a sinh gamma =1 type 1}
\cosh w = \sinh c\sinh \alpha = \sinh a\sinh \gamma.
\end{equation}
A direct computation shows that 
\begin{eqnarray}\label{derivative alpha type 1}
\dot \alpha &=& \frac{1}{\sinh\alpha}\frac{\cosh c + \cosh a\cosh b}{\sinh^2 b \sinh c} \\
&=& \frac{1}{\cosh w}\frac{\cosh c + \cosh a\cosh b}{\sinh^2 b}, \nonumber
\end{eqnarray}
\begin{eqnarray}\label{derivative gamma type 1}
\dot \gamma &=& \frac{1}{\sinh\gamma}\frac{\cosh a + \cosh b\cosh c}{\sinh^2 b \sinh a} \\
&=& \frac{1}{\cosh w}\frac{\cosh a + \cosh b\cosh c}{\sinh^2 b},\nonumber
\end{eqnarray}
and
\begin{eqnarray}\label{derivative beta type 1}
\dot \beta &=& -\frac{1}{\sinh\beta}\frac{\sinh b}{\sinh a \sinh c} \\
&=& -\frac{\cosh \beta}{\sinh \beta}\frac{\sinh b}{\cosh b +\cosh a\cosh c}.\nonumber
\end{eqnarray}
A direct consequence of all the equations above is
\bl \label{pro-abg}
If $b\geq A m$ and \eqref{w-small-1} holds, then we have
\begin{enumerate}
\item $0<\dot\alpha= O_A\left(\frac{m^2}{b^2}\right)$ and $0<\dot\gamma= O_A\left(\frac{m^2}{b^2}\right)$.
\item $\cosh\beta \geq 1+\frac{1}{2\pi^2} \frac{b^2}{m^2}$, $\dot \beta<0$ and $|\dot\beta+1|= O_A\left(\frac{m^2}{b^2}\right)$.
\end{enumerate}
\el
\bp
For Part $(1)$, clearly we have $0<\dot\alpha$ and $0<\dot\gamma$. By \eqref{restriction type 1} and \eqref{derivative alpha type 1} we have
\begin{eqnarray*}
\dot \alpha &\leq&  \frac{2}{\cosh w}\frac{\cosh c\cosh a\cosh b}{\sinh^2 b}\\
&\leq &\frac{\sinh^2 w}{\cosh w} \frac{\cosh^2 b}{\sinh^2 b}\\
&\leq& \frac{1}{\tanh^2 A} 4\pi^2 \frac{m^2}{b^2} 
\end{eqnarray*}
where in the last inequality we apply $b\geq Am\geq A$, $\cosh w\ge 1$ and \eqref{w-small-1}. By a similar argument one may also show $$\dot\gamma\leq \frac{1}{\tanh^2 A} 4\pi^2 \frac{m^2}{b^2}.$$

For Part $(2)$, clearly we have $\dot \beta<0$. By \eqref{w-small-1}, \eqref{restriction type 1} and \eqref{beta} we have
\begin{eqnarray*}
\cosh \beta &\geq& \frac{\cosh b}{\cosh a \cosh c}+1\\
&\geq& \frac{2}{\sinh^2 w}+1\\
&\geq& 1+\frac{1}{2\pi^2} \frac{b^2}{m^2}.
\end{eqnarray*}
In particular $e^{\beta}>\frac{1}{2\pi^2} \frac{b^2}{m^2}\geq \frac{A}{2\pi^2}$. Then by \eqref{restriction type 1} and \eqref{derivative beta type 1} we have 
\begin{eqnarray*}
|\dot\beta +1|
&=& \left| 1-\frac{\tanh b}{\tanh\beta}\frac{1}{1+\frac{\cosh a \cosh c}{\cosh b}} \right| \\
&\leq& \left| 1-\frac{1}{1+\frac{\cosh a \cosh c}{\cosh b}} \right| + \left| 1-\frac{\tanh b}{\tanh\beta} \right| \frac{1}{1+\frac{\cosh a \cosh c}{\cosh b}} \\
&\leq& \frac{\cosh a \cosh c}{\cosh b} + \left| 1-\frac{\tanh b}{\tanh\beta} \right| \\
&\leq& \frac{1}{2}\sinh^2 w + \frac{|\sinh(\beta-b)|}{\sinh\beta \cosh b}. 
\end{eqnarray*}
Since $b\geq mA \geq A$, $\frac{|\sinh(\beta-b)|}{\sinh\beta \cosh b}= O_A\left(e^{-2\beta}+e^{-2b}\right)$. Thus, by  \eqref{w-small-1} we obtain
\begin{eqnarray*}
|\dot\beta +1|
 &\leq& 2\pi^2 \frac{m^2}{b^2} + O_A\left(e^{-2\beta}+e^{-2b}\right)  \\
&\leq& O_A\left(\frac{m^2}{b^2}\right)
\end{eqnarray*}
as desired.
\ep

Let $J_j \subset \mathcal{P}$ be a geodesic segment in \eqref{J-j}. We classify $J_j$ in terms of its possible intersections with $\alpha,\gamma,a$ and $c$ (without orientation). Actually we have
\bl\label{14}
There are 14 kinds of possible segments as shown in Figure \ref{fig 14 kinds of type 1} where each one is a fundamental domain for $\mathcal P$ (or half of a fundamental domain). More precisely, they are: from $\alpha$ to $\gamma$ (type 1.1, 1.2), from $\alpha$ to $\alpha$ (type 1.4), from $\gamma$ to $\gamma$ (type 1.3), from $a$ to $\alpha$ (type 1.5, 1.6), from $a$ to $\gamma$ (type 1.9, 1.10), from $c$ to $\gamma$ (type 1.7, 1.8), from $c$ to $\alpha$ (type 1.11, 1.12) and from $a$ to $c$ (type 1.13, 1.14).
\el

 We denote $d$ as the geodesic segment of $J_j$ in each kind.

If $J_j$ intersects with a piece of $\alpha$, we denote $p$ to be the intersection point and $x$ to be the half segment of $\alpha$ that goes from $p$ to $c$: for this case, see types $1.1, 1.2, 1.4, 1.5, 1.6, 1.11$ and $1.12$ as shown in Figure \ref{fig 14 kinds of type 1}. If $J_j$ intersects with a piece of $\gamma$, denote $q$ to be the intersection point and $y$ to be the half segment of $\gamma$ that goes from $q$ to $a$: see types $1.1, 1.2, 1.3, 1.7, 1.8, 1.9$ and $1.10$ as shown in Figure \ref{fig 14 kinds of type 1}. If $J_j$ does not intersect with any piece of $\alpha$ and $\gamma$, see types $1.13$ and $1.14$ as shown in Figure \ref{fig 14 kinds of type 1}. We fix these points $\{p,q\}$ as $b$ decreases, that is, we keep the length of $x$ and $y$ to be unchanged as $b$ decreases. We remark here that by Lemma \ref{pro-abg} both $\alpha$ and $\gamma$ increase as $b$ decreases. So the points $p$ and $q$ still lie in $\alpha$ and $\gamma$ respectively during the process. 

\begin{figure}[ht]
	\centering
	\includegraphics[width=12.6cm]{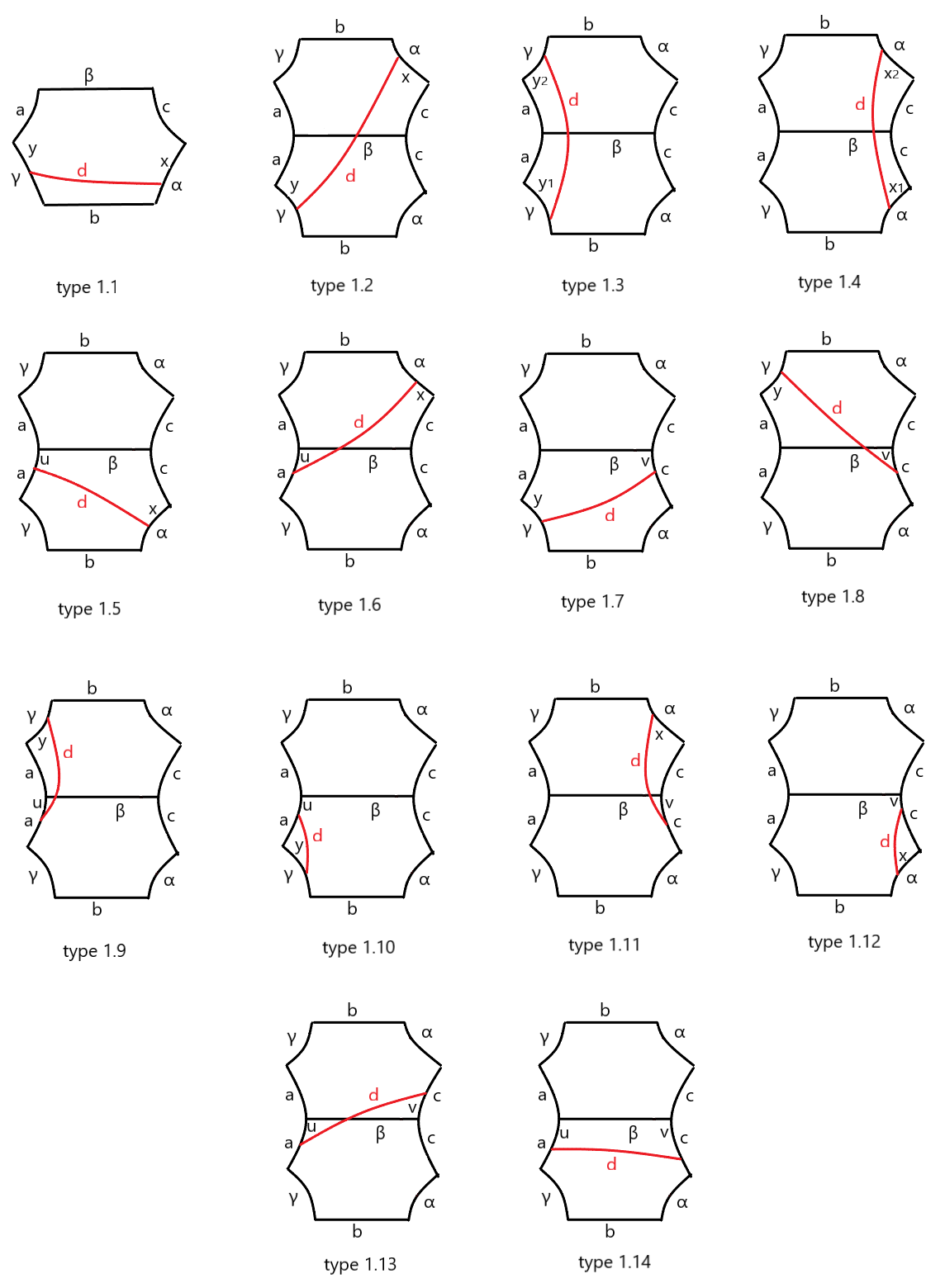}
	\caption{$14$ kinds of geodesic segments in Type-1}
	\label{fig 14 kinds of type 1}
\end{figure}

These intersection points separate $J_j$ into several segments. Fixing these points as $b$ is decreasing, then we get a piecewise geodesic $J''_j$ homotopic to $J_j$ in $\mathcal P_\delta$.

In types 1.3 and 1.4, the lengths $a,y_1,y_2,c,x_1$ and $x_2$ are unchanged during the process, so the corresponding $d$ is unchanged. In types 1.9, 1.10, 1.11 and 1.12, the endpoints of $J_j$ on boundary curves $2a$ and $2c$ are fixed, so the lengths $u$ and $v$, which are segments on $a$ and $c$ from the endpoint of $J_j$ to $\beta$ respectively (see Figure \ref{fig eta'}), in the figure are also unchanged. Also, since $a,y,c$ and $x$ are unchanged, the corresponding $d$ is also unchanged. 

For the remaining $8$ kinds, we will show that 
\be\label{dd-s}
|\dot d +1|=O_A\left(\frac{m^2}{b^2}\right)
\ene
and $J_j$ must contain at least one of those $8$ kinds.

\begin{lemma}\label{lem J contain one of 8 kinds type 1}
	The segment $J_j$ for Type-1 must contain at least one segment of types 1.1, 1.2, 1.5, 1.6, 1.7, 1.8, 1.13 and 1.14 as shown in Figure \ref{fig 14 kinds of type 1}.
\end{lemma}
\begin{proof}
	If $X=\mathcal{P}\cong S_{0,3}$, then $\eta$ is a single $J$ and does not intersect with $2a$ and $2c$. So it is a combination of types $1.1, 1.2, 1.3$, and $1.4$. Since $\eta$ is a filling closed curve, it must intersect with $\alpha$ and $\gamma$. So $\eta$ contains at least one segment from $\alpha$ to $\gamma$, that is of type 1.1 or 1.2.
	
	If $X$ is not $S_{0,3}$, then $\eta$ must intersect with at least one of $2a$ and $2c$, and thus $J_j$ must have an endpoint on $a$ or $c$. Suppose $J_j$ do not contain any of types $1.1, 1.2, 1.5, 1.6, 1.7, 1.8, 1.13$ and $1.14$. Then $J_j$ only consists of segments which may be from $a$ to $\gamma$, or from $\gamma$ to $\gamma$, or from $c$ to $\alpha$ and or from $\alpha$ to $\alpha$. So if starting at a point on $a$, the geodesic $J_j$ can only travel to $\gamma$, then go to next $\gamma$ several times, and finally return to $a$. But in this way $J_j$ would be homotopic to a piece of geodesic line $a$, which contradicts to the assumption that $J_j$ is part of a shortest closed geodesic representative for $\eta$ as shown in \eqref{J-j}. Similarly the geodesic $J_j$ can not start at a point in $c$. So $J_j$ must contain at least one segment of those 8 kinds.
\end{proof}

Next we prove \eqref{dd-s} for those $8$ types in Lemma \ref{lem J contain one of 8 kinds type 1} case by case.

\begin{lemma}[type 1.1]\label{lem type 1.1}
If $b\geq Am$ and \eqref{w-small-1} holds, then for type 1.1 in Figure \ref{fig 14 kinds of type 1} we have  
$$|\dot d +1|=O_A\left(\frac{m^2}{b^2}\right).$$
\end{lemma}

\begin{proof}
Recall that $\dot x=\dot y=0$ and $\dot b =-1$. Since $b\geq Am$ and \eqref{w-small-1} holds, it follows by Lemma \ref{pro-abg} that $|\dot\alpha|= O_A\left(\frac{m^2}{b^2}\right)$ and $|\dot\gamma|= O_A\left(\frac{m^2}{b^2}\right)$. It is clear that $e^{-b}\leq \frac{2}{b^2}$. Next we apply Lemma \ref{lem derivative in rectangle} to the rectangle $(\gamma-y,b,\alpha-x, d)$ to get
\begin{equation*}
|\dot d +1|=|\dot d-\dot b| \leq A'(e^{-b}+|\dot\alpha|+|\dot\gamma|) =O_A\left(\frac{m^2}{b^2}\right)
\end{equation*}
for some constant $A'>0$ only depending on $A$.
\end{proof}

\begin{lemma}[type 1.2]\label{lem type 1.2}
If $b\geq Am$ and \eqref{w-small-1} holds, then for type 1.2 in Figure \ref{fig 14 kinds of type 1} we have  
$$|\dot d +1|=O_A\left(\frac{m^2}{b^2}\right).$$
\end{lemma}

\begin{proof}
Let $h$ be the perpendicular between $\alpha$ and $\gamma$ in type 1.2. Let $e$ to be part of $\alpha$ between $h$ and $b$, and let $f$ to be part of $\gamma$ between $h$ and $a$ (see Figure \ref{fig type 1.2}). Then we have a rectangle $(\alpha-e-x,h,y-f,d)$ with $\angle(h,\alpha-e-x)=\angle(h,y-f)=\frac{\pi}{2}$. We remark here that the length $\alpha-e-x$ and $y-f$ may be negative. Now we compute the lengths of these edges and their derivatives.

\begin{figure}[ht]
	\centering
	\includegraphics[width=5.5cm]{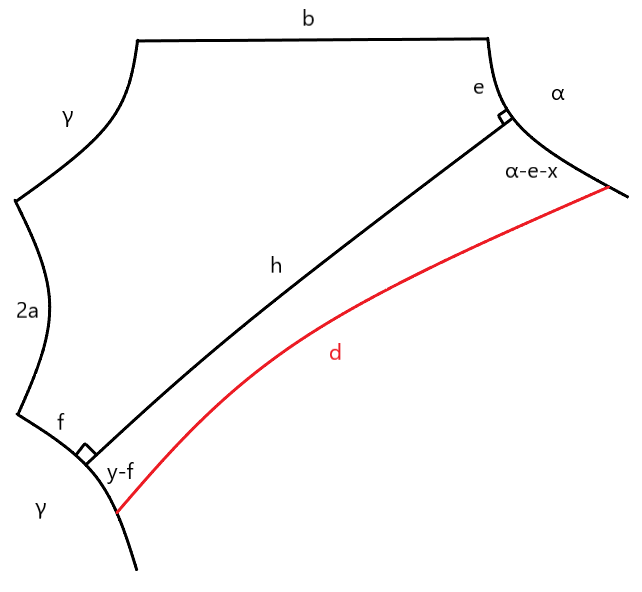}
	\caption{}
	\label{fig type 1.2}
\end{figure}

By Lemma \ref{lem hyp formula} and \eqref{gamma} we have 
\begin{eqnarray}\label{h-form}
\cosh h &=& \sinh 2a\sinh b\cosh \gamma - \cosh 2a\cosh b \\
&=& 2\sinh a\cosh a\sinh b\frac{\cosh c+\cosh a\cosh b}{\sinh a\sinh b} \nonumber\\
&& - (2\cosh^2 a -1)\cosh b \nonumber\\
&=& 2\cosh a\cosh c + \cosh b \nonumber
\end{eqnarray}
and 
\begin{eqnarray}\label{dot h type 1.2}
\dot h &=& \dot b \frac{\sinh b}{\sinh h} \\
&=& \dot b \frac{\cosh h}{\sinh h} \frac{\sinh b}{2\cosh a\cosh c + \cosh b}.\nonumber
\end{eqnarray}
So by \eqref{w-small-1}, \eqref{restriction type 1}, $h\geq b \geq A$ and $\dot b =-1$, we have
\begin{eqnarray*}
|\dot h+1| &=& \left|1- \frac{\cosh h}{\sinh h} \frac{\sinh b}{2\cosh a\cosh c + \cosh b}\right| \\
&\leq& \left|1- \frac{\sinh b}{2\cosh a\cosh c + \cosh b}\right| + \left|1- \frac{\cosh h}{\sinh h} \right|\frac{\sinh b}{2\cosh a\cosh c + \cosh b} \\
&\leq& \frac{2\frac{\cosh a\cosh c}{\cosh b} + \frac{e^{-b}}{\cosh b}}{2\frac{\cosh a\cosh c}{\cosh b} + 1} + \left(\frac{\cosh b}{\sinh b} -1\right) \\
&\leq & 4\pi^2 \frac{m^2}{b^2}+e^{-b}+\frac{e^{-b}}{\sinh A}\\
&=& O_A\left(\frac{m^2}{b^2}\right).
\end{eqnarray*}

By Lemma \ref{lem hyp formula}, \eqref{gamma}, \eqref{sinh a sinh gamma =1 type 1}, \eqref{derivative gamma type 1} and \eqref{dot h type 1.2} we have 
\begin{eqnarray}
\sinh e &=& \sinh 2a\frac{\sinh\gamma}{\sinh h}\\
&=& \frac{2\cosh w\cosh a}{\sinh h} \nonumber
\end{eqnarray}
and
\begin{eqnarray}
\quad \ \dot e &=& \frac{1}{\cosh e} \sinh 2a \bigg(\dot \gamma\frac{\cosh \gamma}{\sinh h} - \dot h \frac{\sinh \gamma\cosh h}{\sinh^2 h}\bigg) \\
&=& \frac{\sinh 2a}{\cosh e} \bigg[\frac{1}{\cosh w} \frac{(\cosh a+\cosh b\cosh c)}{\sinh^2 b} \frac{(\cosh c+\cosh a\cosh b)}{\sinh a\sinh b\sinh h} \nonumber\\
&& +\frac{\sinh b}{\sinh h} \frac{\cosh w}{\sinh a} \frac{\cosh h}{\sinh^2 h} \bigg] \nonumber\\
&=& \frac{1}{\cosh e} \bigg[\frac{2\cosh a}{\cosh w} \frac{(\cosh a+\cosh b\cosh c)}{\sinh^2 b} \frac{(\cosh c+\cosh a\cosh b)}{\sinh b\sinh h} \nonumber\\
&& +2\frac{\cosh a}{\sinh h} \frac{\sinh b}{\sinh h} \cosh w \frac{\cosh h}{\sinh h} \bigg]. \nonumber
\end{eqnarray}
Since $h>b\geq Am$, by \eqref{w-small-1} and \eqref{restriction type 1} it is not hard to see that $\frac{\cosh a}{\sinh h}= O_A\left(\frac{m^2}{b^2}\right)$ and $\frac{2\cosh a}{\cosh w} \frac{(\cosh a+\cosh b\cosh c)}{\sinh^2 b} \frac{(\cosh c+\cosh a\cosh b)}{\sinh b\sinh h}= O_A\left(\frac{m^4}{b^4}\right)$. Recall that by \eqref{w-small-1} $\cosh w=\sqrt{\sinh^2 w+1}\leq \sqrt{4\pi^2 \frac{m^2}{b^2}+1}\leq \sqrt{\frac{4\pi^2}{A^2}+1}$. So we have that $e\leq \sinh e= O_A\left(\frac{m^2}{b^2}\right)$ and $0\leq \dot e =O_A\left(\frac{m^2}{b^2}\right)$.

By Lemma \ref{lem hyp formula}, \eqref{h-form} and \eqref{dot h type 1.2} we have 
\begin{equation}
\sinh f = \sinh b\frac{\sinh\gamma}{\sinh h}
\end{equation}
and
\begin{eqnarray}
\quad \quad \dot f &=& \frac{1}{\cosh f} \bigg( \dot\gamma \cosh\gamma \frac{\sinh b}{\sinh h} + \dot b \sinh\gamma \frac{\cosh b}{\sinh h} - \dot h \sinh \gamma\frac{\sinh b\cosh h}{\sinh^2 h}\bigg) \\
&=& \dot\gamma \frac{\cosh\gamma}{\cosh f}\frac{\sinh b}{\sinh h} + \dot b \frac{\sinh \gamma}{\cosh f} \bigg( \frac{\cosh b}{\sinh h} - \frac{\sinh^2 b\cosh h}{\sinh^3 h} \bigg) \nonumber\\
&=& \dot\gamma \frac{\cosh\gamma}{\cosh f}\frac{\sinh b}{\sinh h} + \dot b \frac{\sinh \gamma}{\cosh f} \frac{(\cosh h -\cosh b)(1+\cosh b\cosh h)}{\sinh^3 h} \nonumber\\
&=& \dot\gamma \frac{\cosh\gamma}{\cosh f}\frac{\sinh b}{\sinh h} - \frac{\sinh \gamma}{\cosh f} \frac{2\cosh a\cosh c(1+\cosh b\cosh h)}{\sinh^3 h}.\nonumber
\end{eqnarray}
Since $\frac{\sinh \gamma}{\sinh f} = \frac{\sinh h}{\sinh b}\leq \frac{2\cosh a\cosh c+\cosh b}{\sinh b} = O_A(1)$ is bounded (by \eqref{restriction type 1} and \eqref{h-form}), we know that $\frac{\cosh \gamma}{\cosh f}=O_A(1)$ is also bounded. Apply \eqref{w-small-1} and \eqref{restriction type 1} and the fact that $h>b\geq Am$, we have that $\frac{2\cosh a\cosh c(1+\cosh b\cosh h)}{\sinh^3 h}= O_A\left(\frac{m^2}{b^2}\right)$. Thus together with Lemma \ref{pro-abg}, we have $|\dot f|= O_A\left(\frac{m^2}{b^2}\right)$. 

Now consider the rectangle $(\alpha-e-x,h,y-f,d)$. By Lemma \ref{pro-abg} and the discussion above we have $|\dot \alpha-\dot e|=O_A\left(\frac{m^2}{b^2}\right)$, $|\dot f|= O_A\left(\frac{m^2}{b^2}\right)$ and $|\dot h+1|=O_A\left(\frac{m^2}{b^2}\right)$. Recall that $\dot x=\dot y=0$ and $h>b\geq Am$. Then we apply Lemma \ref{lem derivative in rectangle} to the rectangle $(\alpha-e-x,h,y-f,d)$ to get $|\dot d+1|= O_A\left(\frac{m^2}{b^2}\right)$ as desired.
\end{proof}

\begin{lemma}[types 1.5--1.8]\label{lem type 1.5-1.8}
If $b\geq Am$ and \eqref{w-small-1} holds, then for types 1.5, 1.6, 1.7 and 1.8 in Figure \ref{fig 14 kinds of type 1} we have  
$$|\dot d +1|=O_A\left(\frac{m^2}{b^2}\right).$$
\end{lemma}

\begin{proof}
We only prove the lemma for type 1.5. For the other three types, the proofs are similar as the one for type 1.5.

Let $h$ be the perpendicular between $a$ and $\alpha$ in type 1.5. Let $e$ to be part of $\alpha$ between $h$ and $b$, and let $f$ to be part of $a$ between $h$ and $\gamma$ (see Figure \ref{fig type 1.5}). Then we have a rectangle $(f+u-a,h,e+x-\alpha,d)$ with $\angle(h,f+u-a)=\angle(h,e+x-\alpha)=\frac{\pi}{2}$. We remark here that the length $f+u-a$ and $e+x-\alpha$ may be negative. Now we compute the lengths of these edges and their derivatives.

\begin{figure}[ht]
	\centering
	\includegraphics[width=5cm]{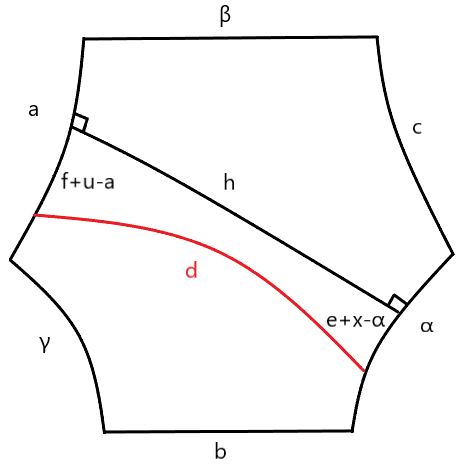}
	\caption{}
	\label{fig type 1.5}
\end{figure}

By Lemma \ref{lem hyp formula}, \eqref{gamma}, \eqref{sinh a sinh gamma =1 type 1} and \eqref{derivative gamma type 1}, we have 
\begin{eqnarray} \label{cosh h-1.5}
\cosh h &=& \sinh b \sinh \gamma \\
&=& \frac{\cosh w}{\sinh a} \sinh b \nonumber
\end{eqnarray}
and
\begin{eqnarray}
\quad \quad \dot h &=& \frac{1}{\sinh h} \bigg( \dot b \cosh b\sinh\gamma + \dot\gamma \sinh b\cosh\gamma \bigg)\\
&=& -\frac{\cosh h}{\sinh h}\frac{\cosh b}{\sinh b} \nonumber\\
&& + \frac{1}{\cosh w}\frac{(\cosh a +\cosh b\cosh c)}{\sinh^2 b}\frac{(\cosh c+\cosh a\cosh b)}{\sinh a\sinh b}\frac{\sinh b}{\sinh h} \nonumber\\
&=& -\frac{\cosh h}{\sinh h}\frac{\cosh b}{\sinh b} \nonumber\\
&&+ \frac{1}{\cosh^2 w} \frac{\cosh h}{\sinh h} \frac{(\cosh a +\cosh b\cosh c)(\cosh c+\cosh a\cosh b)}{\sinh^3 b}. \nonumber
\end{eqnarray}
Since $b\geq Am$, by \eqref{w-small-1}, \eqref{restriction type 1} and \eqref{cosh h-1.5} it is easy to see that $\frac{1}{\cosh h} = O_A\left(\frac{m^2}{b^2}\right)$. Which also implies that $|\frac{\cosh h}{\sinh h}\frac{\cosh b}{\sinh b}-1|=O_A\left(\frac{m^4}{b^4}\right)$. Then by \eqref{w-small-1} and \eqref{restriction type 1} it is not hard to see that
$$\frac{1}{\cosh^2 w} \frac{\cosh h}{\sinh h} \frac{(\cosh a +\cosh b\cosh c)(\cosh c+\cosh a\cosh b)}{\sinh^3 b}=O_A\left(\frac{m^2}{b^2}\right).$$ 
Thus, we have $$|\dot h +1|=O_A\left(\frac{m^2}{b^2}\right).$$

By Lemma \ref{lem hyp formula}, we have 
\begin{eqnarray}
\sinh (a-f) &=& \frac{\cosh c}{\sinh h}\\
&=& \frac{\cosh h}{\sinh h} \frac{1}{\cosh w} \frac{\sinh a\cosh c}{\sinh b}  \nonumber
\end{eqnarray}
and
\begin{eqnarray}
-\dot f &=& -\dot h \frac{1}{\cosh(a-f)} \frac{\cosh c\cosh h}{\sinh^2 h} \\
&=& -\dot h \frac{1}{\cosh(a-f)} \frac{\cosh h}{\sinh h} \sinh(a-f).\nonumber
\end{eqnarray}
Then again by \eqref{w-small-1}, \eqref{restriction type 1} and the bounds for $h$ and $\dot h$ above, it is easy to see that 
$$|a-f|=O_A\left(\frac{m^2}{b^2}\right)\quad \text{and} \quad |\dot f|=O_A\left(\frac{m^2}{b^2}\right).$$

By Lemma \ref{lem hyp formula}, we have 
\begin{equation}
\sinh e = \frac{\cosh f}{\sinh b} \leq \frac{\cosh a}{\sinh b}
\end{equation}
and
\begin{equation}
\dot e = \frac{1}{\cosh e}\bigg( \dot f \frac{\sinh f}{\sinh b} - \dot b \frac{\cosh f\cosh b}{\sinh^2 b} \bigg) \nonumber
\end{equation}
implying that
\begin{equation}
|\dot e| \leq |\dot f| \frac{\sinh a}{\sinh b} + \frac{\cosh a\cosh b}{\sinh^2 b}.
\end{equation}
Then by \eqref{w-small-1} and \eqref{restriction type 1} we have $e=O_A\left(\frac{m^2}{b^2}\right)$ and $|\dot e|=O_A\left(\frac{m^2}{b^2}\right).$

Recall that $\dot u=\dot a=\dot x=0$. Then we apply Lemma \ref{lem derivative in rectangle}, Lemma \ref{pro-abg} and all these bounds above to rectangle $(f+u-a,h,e+x-\alpha,d)$ to obtain  $|\dot d+1|=O_A\left(\frac{m^2}{b^2}\right)$ as desired.
\end{proof}

\begin{lemma}[types 1.13 and 1.14]\label{lem type 1.13-1.14}
If $b\geq Am$ and \eqref{w-small-1} holds, then for types 1.13 and 1.14 in Figure \ref{fig 14 kinds of type 1} we have  
$$|\dot d +1|=O_A\left(\frac{m^2}{b^2}\right).$$
\end{lemma}

\begin{proof}
Recall that $\dot u=\dot v=0$. For type 1.13, by Lemma \ref{pro-abg} and applying Lemma \ref{lem derivative in rectangle} to rectangle $(u,\beta,v,d)$, we obtain $|\dot d +1|=O_A\left(\frac{m^2}{b^2}\right)$ as desired. For type 1.14, the proof is similar.
\end{proof}

Now we are ready to prove Proposition \ref{prop filling curve decrease a little} for Type-1.

\begin{proposition}\label{key-pro-1}
Proposition \ref{prop filling curve decrease a little} holds for the case that $\mathcal P$ is of Type-1.
\end{proposition}
\bp
Let $2b$ to be the longest boundary geodesic of $X$, and $\mathcal P$ be the pants with three boundary geodesics $(2a,2b,2c)$. Since $\sum x_i \geq n(2Am+\delta)$, we know $2b\geq 2Am+\delta>\delta$. Hence both the pants $\mathcal{P}_\delta$ and the surface $X_\delta$ exist. 

As $b$ smoothly decreases to $b-\frac{1}{2}\delta$, we denote $b'$ and $w'$ to be the quantities corresponding to $b$ and $w$ respectively. The lengths $a$ and $c$ are unchanged. It is clear that
$$b-\frac{1}{2}\delta \leq b'\leq b.$$

So $b'\geq Am$ always holds. 

By Lemma \ref{lem hyp formula}, we have $\cosh w = \sinh c\sinh \alpha$ and then
\begin{eqnarray*}
\sinh^2 w &=& \cosh^2 w-1 \\
&=& \sinh^2 c \cosh^2 \alpha - \cosh^2 c \\
&=& \sinh^2 c \frac{(\cosh a +\cosh b\cosh c)^2}{\sinh^2 b\sinh^2 c} -\cosh^2 c \\
&=& \frac{\cosh^2 a +\cosh^2 c+2\cosh a\cosh b\cosh c}{\sinh^2 b}.
\end{eqnarray*}
Recall that $b-\frac{1}{2}\delta \leq b'\leq b.$ Then we have
\begin{eqnarray*}
\frac{\sinh w'}{\sinh w}
&=& \frac{\sinh b}{\sinh b'} \left( \frac{\cosh^2 a +\cosh^2 c+2\cosh a\cosh b'\cosh c}{\cosh^2 a +\cosh^2 c+2\cosh a\cosh b\cosh c} \right)^{\frac{1}{2}} \\
&\leq& \frac{\sinh b}{\sinh(b-\frac{1}{2}\delta)} \\
&=& \cosh\frac{1}{2}\delta + \frac{1}{\tanh(b-\frac{1}{2}\delta)} \sinh\frac{1}{2}\delta \\
&\leq& \cosh\frac{1}{2}\delta + \frac{1}{\tanh A} \sinh\frac{1}{2}\delta \\
&<& 2
\end{eqnarray*}
for $\delta \in (0, A_0]$ where $A_0$ is some constant only depending on $A$.
Recall that \eqref{w-small} says that $\sinh w\leq \frac{\pi m}{b}$. So we have that for $0<\delta\leq A_0$,  the inequality that $\sinh w'\leq \frac{2\pi m}{b}$ always holds, i.e, equation\eqref{w-small-1} always holds. Thus the assumptions in Lemma \ref{lem type 1.1}, \ref{lem type 1.2}, \ref{lem type 1.5-1.8} and \ref{lem type 1.13-1.14} are always satisfied.

As the boundary length $b$ reduces by $\frac{1}{2}\delta$, the length $d$ in types 1.3, 1,4, 1.9, 1.10, 1.11 and 1.12 are unchanged. For remaining types 1.1, 1.2, 1.5, 1.6, 1.7, 1.8, 1.13 and 1.14, it follows by Lemma \ref{lem type 1.1}, \ref{lem type 1.2}, \ref{lem type 1.5-1.8} and \ref{lem type 1.13-1.14} that the length $d$  decreases at least $\frac{1}{2} \left(1-O_A\left(\frac{m^2}{b^2}\right)\right) \delta$. By Lemma \ref{lem J contain one of 8 kinds type 1} $J_j$ must contain at least one of the decreasing types. So we have
$$\ell(J_j) - \ell(J'_j) \geq \frac{1}{2} \left(1-O_A\left(\frac{m^2}{b^2}\right)\right) \delta$$
which together with \eqref{eqn eta-eta' = sum J-J'} imply that
$$\ell_\eta (X) - \ell_\eta(X_\delta) \geq \frac{1}{2} \left(1-O_A\left(\frac{m^2}{b^2}\right)\right) \delta= \frac{1}{2} \left(1-O_A\left(\frac{m^2n^2}{(\sum x_i)^2}\right)\right) \delta$$ where in the last equation we apply $2b\geq\frac{\sum x_i}{n}$. 

The proof is complete.
\ep

\begin{rem*}

Moreover, if $g=n=1$, then the pair of pants $\mathcal P$ must be of Type-1, and the two boundary curves $2a=2c$ in $\mathcal P$. It is not hard to see that any closed filling curve $\eta$ contains at least two decreasing types. For this case, actually one can improve Proposition \ref{prop filling curve decrease a little} and Theorem \ref{sec-count} to be 
\begin{proposition}
With the same assumptions in Proposition \ref{prop filling curve decrease a little}, if $g=n=1$ and $X\in \T_{1,1}(x)$, we have
$$\ell_\eta (X) - \ell_\eta(X_\delta) \geq \left(1-O_A\left(\frac{1}{x^2}\right)\right) \delta.$$ 
\end{proposition}

\bt
For any $\eps_1>0$, there exists a constant $c(\eps_1)>0$ only depending on $\eps_1$ such that for all $T>0$ and any compact hyperbolic surface $X\cong S_{1,1}$ of geodesic boundary, we have
\begin{equation*}
	\#_f(X,T)\leq c(\eps_1) \cdot e^{T-(1-\eps_1)\ell(\partial X)}.
\end{equation*}
\et
In this paper the bound in Proposition \ref{prop filling curve decrease a little} is sufficient for us to prove Theorem \ref{main} and \ref{main-2}.
\end{rem*}
%%%%%%%%%%%%%%%%%%%%%%%%%%%%%
\subsection{Proof of Proposition \ref{prop filling curve decrease a little} for Type-2}
Now we consider a pair of pants in Figure \ref{fig 2 type pants} of Type-2 and the corresponding right-angled hexagon in Figure \ref{fig 2 type hexagon} of Type-2. Both $2b$ and $2c$ are two boundary geodesics of the surface.

In this subsection, we always use the notation $\alpha=w$ to denote the perpendicular between $b$ and $c$ as shown in Figure \ref{fig 2 type hexagon} of Type-2. For the pants $\mathcal P$ of Type-2 we construct, \eqref{w-small} always holds. But when $b$ and $c$ both decrease, \eqref{w-small} may not hold again. Instead, we always assume $\alpha=w$ satisfy
\begin{equation}\label{w-small-2}
\sinh w \leq 2\pi\frac{m}{b}.
\end{equation}
Later in the proof of Proposition \ref{key-pro-2} we will show \eqref{w-small-2} holds when $b$ and $c$ do not reduce too much. By Lemma \ref{lem hyp formula} we have
\begin{eqnarray}\label{restriction type 2}
\cosh a & = & \sinh b \sinh c \cosh w - \cosh b \cosh c \\
& = & 2\sinh^2\frac{w}{2} \sinh b\sinh c - (\cosh b \cosh c-\sinh b\sinh c)  \nonumber\\
& \leq & 2\sinh^2\frac{w}{2} \sinh b\sinh c \nonumber
\end{eqnarray}
Now we consider to simultaneously reduce the two boundary lengths $2b$ and $2c$ of $\mathcal P$. Let the length $\{a(t),b(t),c(t)\}$ be functions in terms of a common parameter $t$. In our process from $\mathcal P$ to $\mathcal P_\delta$, we assume $a$ is fixed and both $b$ and $c$ decrease with constant speed. More precisely, denote $\dot b = \frac{db}{dt}$ and $\dot c = \frac{dc}{dt}$, assume
\be
\dot a=0 \quad \emph{and} \quad \dot b =\dot c=-1.
\ene
Let $\alpha=w, \beta$ and $\gamma$ be as shown in type-2 of Figure \ref{fig 2 type pants}. Denote $h_a$ to be the shortest perpendicular between $a$ and $\alpha$, $h_b$ to be the shortest perpendicular between $b$ and $\beta$, $h_c$ to be the shortest perpendicular between $c$ and $\gamma$. Then by Lemma \ref{lem hyp formula} we have
\begin{equation}\label{cosh ha = in type 2}
\cosh h_a = \sinh b\sinh \gamma = \sinh c\sinh\beta,
\end{equation}
\begin{equation}\label{cosh hb = in type 2}
\cosh h_b = \sinh a\sinh \gamma = \sinh c\sinh\alpha,
\end{equation}
\begin{equation}\label{cosh hc = in type 2}
\cosh h_c = \sinh a\sinh \beta = \sinh b\sinh\alpha.
\end{equation}

\noindent Recall that by Lemma \ref{lem hyp formula}, we have 
$$\cosh \alpha = \frac{\cosh a+\cosh b\cosh c}{\sinh b\sinh c},$$
$$\cosh \beta = \frac{\cosh b+\cosh a\cosh c}{\sinh a\sinh c},$$
$$\cosh \gamma = \frac{\cosh c+\cosh a\cosh b}{\sinh a\sinh b}.$$
A direct computation shows that
\begin{eqnarray}\label{derivative alpha type 2}
\dot \alpha &=& \frac{1}{\sinh \alpha} \bigg( \frac{\cosh a(\cosh b\sinh c + \sinh b\cosh c)}{\sinh^2 b\sinh^2 c}\\
&& + \frac{1}{\sinh^2 b} \frac{\cosh c}{\sinh c} + \frac{1}{\sinh^2 c} \frac{\cosh b}{\sinh b} \bigg) \nonumber\\
&=& \frac{\cosh a}{\sinh w\sinh b\sinh c}\big( \frac{\cosh b}{\sinh b} + \frac{\cosh c}{\sinh c} \big) \nonumber\\
&& + \frac{1}{\cosh h_c \sinh b} \frac{\cosh c}{\sinh c} + \frac{1}{\cosh h_b\sinh c} \frac{\cosh b}{\sinh b}. \nonumber
\end{eqnarray}

\begin{eqnarray}\label{derivative beta type 2}
\dot \beta &=& \frac{-1}{\sinh \beta}\bigg( \frac{(\sinh b+\cosh a\sinh c)\sinh c}{\sinh a\sinh^2 c} \\
&& - \frac{\cosh c(\cosh b+\cosh a\cosh c)}{\sinh a\sinh^2 c}\bigg)  \nonumber \\
&=& \frac{\cosh a +\cosh (b-c)}{\cosh h_c \sinh^2 c} \nonumber\\
&=& \frac{\cosh a +\cosh (b-c)}{\sinh\alpha \sinh b \sinh^2 c} \nonumber\\
&=& \frac{\cosh a +\cosh (b-c)}{\cosh h_b \sinh b \sinh c},\nonumber
\end{eqnarray}
\begin{eqnarray}\label{derivative gamma type 2}
\dot \gamma &=& \frac{-1}{\sinh \gamma}\bigg(\frac{(\sinh c+\cosh a\sinh b)\sinh b}{\sinh a\sinh^2 b}\\
&&-\frac{\cosh b(\cosh c+\cosh a\cosh b)}{\sinh a\sinh^2 b}\bigg) \nonumber\\
&=& \frac{\cosh a +\cosh (b-c)}{\cosh h_b \sinh^2 b} \nonumber\\
&=& \frac{\cosh a +\cosh (b-c)}{\sinh\alpha \sinh c \sinh^2 b} \nonumber\\
&=& \frac{\cosh a +\cosh (b-c)}{\cosh h_c \sinh c \sinh b}. \nonumber
\end{eqnarray}
A direct consequence of all the equations above is
\bl \label{pro-abg-2}
If $b\geq Am$ and \eqref{w-small-2} holds, then we have
\begin{enumerate}
\item $\sinh c \geq \frac{1}{2\pi}\frac{b}{m}$.
\item $0<\dot\alpha= O_A\left(\frac{m}{b}\right), \ 0<\dot \beta= O_A\left(\frac{m^2}{b^2}\right) \ \emph{and} \ 0<\dot\gamma= O_A\left(\frac{m^2}{b^2}\right)$.
\end{enumerate}
\el
\bp
For Part (1), by \eqref{w-small-2} and \eqref{cosh hb = in type 2} we have $$\sinh c \geq \frac{1}{\sinh \alpha} = \frac{1}{\sinh w} \geq \frac{1}{2\pi}\frac{b}{m}.$$

For Part (2), by \eqref{derivative alpha type 2}, \eqref{derivative beta type 2} and \eqref{derivative gamma type 2} we clearly have $0<\dot\alpha$, $0<\dot\beta$ and $0<\dot\gamma$. By \eqref{restriction type 2} we have $\frac{\cosh a}{\sinh w\sinh b\sinh c}\leq \frac{2\sinh^2{\frac{w}{2}}}{\sinh w} \leq 2\sinh w \leq 4\pi\frac{m}{b}$ where we apply \eqref{w-small-2} in the last inequality. Since $b\geq A$ and $\sinh c\geq \frac{A}{2\pi}$, from \eqref{derivative alpha type 2} we have $\dot {\alpha}= O_A\left(\frac{m}{b}\right)$. For $\dot \beta$, by \eqref{w-small-2} and \eqref{restriction type 2} we have $\frac{\cosh a}{\cosh h_b\sinh b \sinh c}\leq 2\sinh^2{\frac{w}{2}}\leq 8\pi^2\frac{m^2}{b^2}$. Since $b\geq A$ and $\sinh c\geq \frac{A}{2\pi}$, we know $\frac{\cosh \left(b-c\right)}{\cosh h_b\sinh b \sinh c}= O_A\left(e^{-2b}+e^{-2c}\right)= O_A\left(\frac{m^2}{b^2}\right)$. So by \eqref{derivative beta type 2} we have $\dot \beta = O_A\left(\frac{m^2}{b^2}\right)$. By a similar argument one can show that $\dot \gamma= O_A\left(\frac{m^2}{b^2}\right)$.
\ep

Let $J_j \subset \mathcal{P}$ to be a geodesic segment in \eqref{J-j}. We classify $J_j$ in terms of its possible intersections with $\alpha,\beta, \gamma$ and $a$ (without orientation). Actually we have
\bl\label{type2-6}
There are 6 kinds of possible segments as shown in Figure \ref{fig 6 kinds of type 2}. More precisely, they are: from $\beta$ to $\gamma$ (type 2.1), from $\alpha$ to $\gamma$ (type 2.2), from $\alpha$ to $\beta$ (type 2.3), from $a$ to $\alpha$ (type 2.4), from $a$ to $\beta$ (type 2.5) and from $a$ to $\gamma$ (type 2.6).
\el

We denote $d$ as the geodesic segment of $J_j$ in each kind. 

\begin{figure}[ht]
	\centering
	\includegraphics[width=12cm]{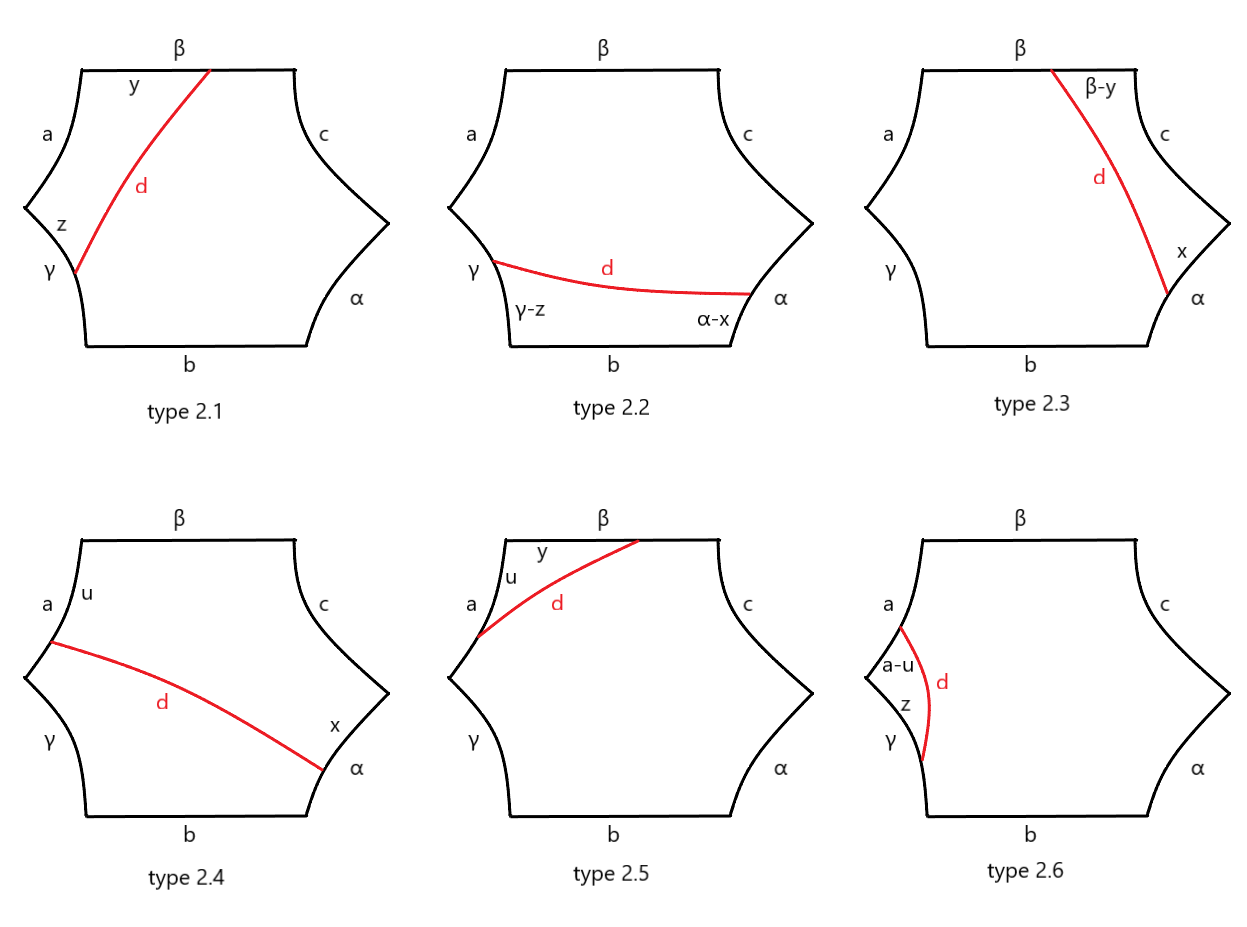}
	\caption{6 kinds of geodesic segments in Type-2}
	\label{fig 6 kinds of type 2}
\end{figure}

Similar to what we do for Type-1. Denote $p,q$ and $r$ to be intersection points of $J_j$ and $\alpha,\beta$ and $\gamma$ respectively if they exist. Denote $x$ to be the part of $\alpha$ that from $p$ to $c$ if it exists, and $y$ to be the part of $\beta$ that from $q$ to $a$ if it exists, and $z$ to be the part of $\gamma$ that from $r$ to $a$ if it exists. We fix these points $p,q$ and $r$ as $b$ and $c$ are simultaneously decreasing, more precisely, keep the length of each $x,y$ and $z$ to be unchanged. Since $\alpha,\beta$ and $\gamma$ are increasing as $b$ and $c$ are simultaneously decreasing by Lemma \ref{pro-abg-2},the points $p,q$ and $r$ will still lie in $\alpha,\beta$ and $\gamma$ respectively during the process. These intersection points separate $J_j$ into several segments. Fixing these points as $b$ and $c$ are simultaneously decreasing, we get a piecewise geodesic $J''_j$ homotopic to $J_j$ in $\mathcal P_\delta$. 

In type 2.1, the lengths $a,y$ and $z$ are all unchanged during the process, so the corresponding $d$ is also unchanged. In types 2.5 and 2.6, the endpoints of $J_j$ on boundary $2a$ are fixed, so the length $u$ (part of $a$ from the endpoint of $J_j$ to $\beta$, see Figure \ref{fig 6 kinds of type 2}) is also unchanged. Since $a,y$ and $z$ are all unchanged, the corresponding $d$ is also unchanged. 

For the remaining $3$ kinds, similar to Type-1 we will show that 
\be \label{dd-s-2}
|\dot d +1| = O_A\left(\frac{m}{b}\right)
\ene
and $J_j$ must contain at least two of those $3$ kinds.

\begin{lemma}\label{lem J contain two of 3 kinds type 2}
	The segment $J_j$ for Type-2 must contain at least two segments of types 2.2, 2.3 and 2.4 as shown in Figure \ref{fig 6 kinds of type 2}.
\end{lemma}
\begin{proof}
Since $\eta$ is filling, the segment $J_j$ is a geodesic in $\mathcal P$ with endpoints on $2a$ (if $X$ is not $S_{0,3}$) or just a closed geodesic in $\mathcal P$ (if $X$ is $S_{0,3}$). Moreover it must have at least one intersection point with $\alpha$ which is denoted by $o$. On both sides of $J_j$ at point $o$, there is at least one segment $d$ of type 2.2 or 2.3 or 2.4. Then the conclusion follows.
\end{proof}

Next we prove \eqref{dd-s-2} for those $3$ types in Lemma \ref{lem J contain two of 3 kinds type 2} case by case.

\begin{lemma}[types 2.2 and 2.3]\label{lem type 2.2 2.3}
If $b\geq Am$ and \eqref{w-small-2} holds, then for types 2.2 and 2.3 in Figure  \ref{fig 6 kinds of type 2} we have  
$$|\dot d +1|= O_A\left(\frac{m}{b}\right).$$
\end{lemma}

\begin{proof}
Recall that $\dot x=\dot y=\dot z=0$ and $\dot b=\dot c=-1$. For type 2.2, by Lemma \ref{pro-abg-2} and we apply Lemma \ref{lem derivative in rectangle} to rectangle $(\gamma-z,b,\alpha-x,d)$ to obtain that $|\dot d +1|= O_A\left(\frac{m}{b}\right)$. For type 2.3, the proof is similar.
\end{proof}

\begin{lemma}[type 2.4]\label{lem type 2.4}
If $b\geq Am$ and \eqref{w-small-2} holds, then for type 2.4 in Figure  \ref{fig 6 kinds of type 2} we have  
$$|\dot d +1|= O_A\left(\frac{m}{b}\right).$$
\end{lemma}

\begin{proof}
Let $h_a$ be the perpendicular between $a$ and $\alpha$ in type 2.4. Let $e$ to be part of $\alpha$ between $h_a$ and $c$, and let $f$ to be part of $a$ between $h_a$ and $\beta$ (see Figure \ref{fig type 2.4}). Then we have a rectangle $(u-f,h_a,x-e,d)$ with $\angle(h_a,u-f)=\angle(h_a,x-e)=\frac{\pi}{2}$. We remark here that the lengths $u-f$ and $x-e$ may be negative. Now we compute the lengths of these edges and their derivatives.

\begin{figure}[ht]
	\centering
	\includegraphics[width=5cm]{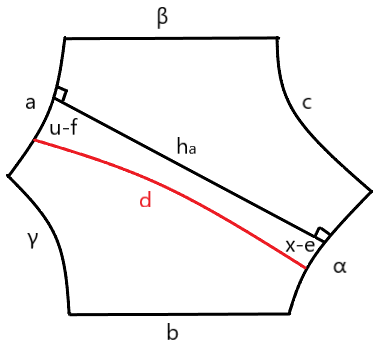}
	\caption{}
	\label{fig type 2.4}
\end{figure}

By Lemma \ref{lem hyp formula} we have
\begin{equation} \label{ha-be}
\cosh h_a = \sinh c\sinh\beta \ \rm{and} \ \cosh \beta=\sinh e\sinh h_a.
\end{equation}
By \eqref{derivative beta type 2} we have $\dot \beta=\frac{\cosh a +\cosh (b-c)}{\sinh\alpha \sinh b \sinh^2 c}$. A direct computation shows that
\begin{eqnarray}
\dot h_a &=& \frac{1}{\sinh h_a}\bigg( \dot c \cosh c\sinh\beta + \dot\beta \sinh c\cosh\beta \bigg) \\
&=& -\frac{\cosh h_a}{\sinh h_a}\frac{\cosh c}{\sinh c} + \frac{\cosh\beta}{\sinh h_a} \frac{(\cosh a+\cosh(b-c))}{\sinh\alpha \sinh b\sinh c} \nonumber\\
&=& -\frac{\cosh h_a}{\sinh h_a}\frac{\cosh c}{\sinh c} + \frac{\sinh e}{\sinh \alpha} \frac{(\cosh a+\cosh(b-c))}{\sinh b\sinh c}.\nonumber
\end{eqnarray}
Since $\sinh e\leq\sinh\alpha=\sinh w\leq 2\pi\frac{m}{b}$, we have 
$$\sinh h_a=\frac{\cosh \beta}{\sinh e}\geq \frac{1}{2\pi}\frac{b}{m}.$$ 
Since $e\leq\alpha=w$, by \eqref{w-small-2} and \eqref{restriction type 2} we have $ \frac{\sinh e}{\sinh \alpha} \frac{\cosh a}{\sinh b\sinh c}\leq 2\sinh^2{\frac{w}{2}}\leq 8\pi^2\frac{m^2}{b^2}$. By Lemma \ref{pro-abg-2} 
$\frac{\sinh e}{\sinh \alpha}\frac{\cosh \left(b-c\right)}{\sinh b \sinh c}= O_A\left(e^{-2b}+e^{-2c}\right) = O_A\left(\frac{m^2}{b^2}\right)$. Then we have 
\begin{eqnarray*}
|\dot h_a +1| 
&\leq& \left|\frac{\cosh h_a}{\sinh h_a}\frac{\cosh c}{\sinh c}-1\right| + O_A\left(\frac{m^2}{b^2}\right) \\
&=& O_A\left(e^{-2h_a} + e^{-2c} + \frac{m^2}{b^2}\right) \\
&=& O_A\left(\frac{m^2}{b^2}\right).
\end{eqnarray*}

By Lemma \ref{lem hyp formula} and \eqref{ha-be} we have
\begin{equation}
\sinh e = \frac{\cosh \beta}{\sinh h_a}
\end{equation}
and
\begin{eqnarray}
\dot e &=& \frac{1}{\cosh e} \bigg( \dot\beta \frac{\sinh\beta}{\sinh h_a} - \dot h_a \frac{\cosh\beta\cosh h_a}{\sinh^2 h_a}\bigg) \nonumber\\
&=& \frac{1}{\cosh e} \bigg( \dot\beta \frac{\cosh h_a}{\sinh h_a}\frac{1}{\sinh c} - \dot h_a \frac{\cosh h_a}{\sinh h_a} \sinh e \bigg).
\end{eqnarray}
By Lemma \ref{pro-abg-2}, the fact that $\sinh e\leq \sinh\alpha=\sinh w\leq 2\pi\frac{m}{b}$ and the above bounds, we have $|\dot e|= O_A\left(\frac{m}{b}\right)$.

By Lemma \ref{lem hyp formula} and \eqref{derivative beta type 2} we have
\begin{equation}
\sinh f = \frac{\cosh e}{\sinh \beta}
\end{equation}
and
\begin{eqnarray}
\quad \quad \quad \dot f &=& \frac{1}{\cosh f} \bigg( \dot e\frac{\sinh e}{\sinh\beta} - \dot\beta \frac{\cosh\beta \cosh e}{\sinh^2\beta} \bigg) \\
&=& \dot e \frac{\sinh f}{\cosh f} \frac{\sinh e}{\cosh e} - \frac{(\cosh a +\cosh(b-c))}{\sinh\alpha \sinh b\sinh^2 c} \frac{\cosh\beta\cosh e}{\sinh e\sinh c\sinh^2\beta} \nonumber\\
&=& \dot e \frac{\sinh f}{\cosh f} \frac{\sinh e}{\cosh e} - \frac{(\cosh a +\cosh(b-c))\sinh h_a \cosh e}{\sinh\alpha\sinh b\sinh c \cosh^2 h_a} \nonumber\\
&=& \dot e \frac{\sinh f}{\cosh f} \frac{\sinh e}{\cosh e} - \frac{\sinh e}{\sinh\alpha} \frac{\sinh^2 h_a}{\cosh^2 h_a} \frac{\cosh e}{\cosh\beta} \frac{(\cosh a +\cosh(b-c))}{\sinh b\sinh c}.\nonumber
\end{eqnarray}
Since $\sinh e\leq 2\pi\frac{m}{b}$ and $|\dot e|= O_A\left(\frac{m}{b}\right)$, we have $\left| \dot e \frac{\sinh f}{\cosh f} \frac{\sinh e}{\cosh e}\right|= O_A\left(\frac{m^2}{b^2}\right)$. By \eqref{w-small-2} and \eqref{restriction type 2} it is not hard to see that 
\begin{eqnarray*}
\left| \frac{\sinh e}{\sinh\alpha} \frac{\sinh^2 h_a}{\cosh^2 h_a} \frac{\cosh e}{\cosh\beta} \frac{(\cosh a +\cosh(b-c))}{\sinh b\sinh c}  \right|
&\leq& \cosh e \frac{\cosh a +\cosh(b-c)}{\sinh b\sinh c} \\
&=& O_A\left(\frac{m^2}{b^2}\right).
\end{eqnarray*} 
Thus, we have $|\dot f|= O_A\left(\frac{m^2}{b^2}\right)$.

Recall that $\dot x=\dot u=0$. By Lemma \ref{pro-abg-2} and all the bounds above, then we apply Lemma \ref{lem derivative in rectangle} to rectangle $(u-f,h_a,x-e,d)$ to obtain that $|\dot d+1|= O_A\left(\frac{m}{b}\right)$.
\end{proof}

Now we are ready to prove Proposition \ref{prop filling curve decrease a little} for Type-2.

\begin{proposition} \label{key-pro-2}
Proposition \ref{prop filling curve decrease a little} holds for the case that $\mathcal P$ is of Type-2.
\end{proposition}
\bp
Let $2b$ to be the longest boundary geodesic of $X$, and $\mathcal P$ be the pants with boundary geodesics $(2a,2b,2c)$ as in the construction. Since $\sum x_i \geq n(2Am+\delta)$, we know $2b\geq 2Am+\delta>\frac{1}{2}\delta$. By Lemma \ref{pro-abg-2}, we know $\sinh c\geq \frac{b}{2\pi m}\geq \frac{A}{2\pi}$. So if $0<\delta<4\arcsinh(\frac{A}{2\pi})$, we have $2c>\frac{1}{2}\delta$. Hence both the pants $\mathcal{P}_\delta$ and the surface $X_\delta$ exist. 

During the process that $b$ and $c$ reduce to $b-\frac{1}{4}\delta$ and $c-\frac{1}{4}\delta$ respectively, we denote $b',c',w'$ to be the quantities
corresponding to $b,c,w$ respectively. The length $a$ is unchanged. We have
$$b-\frac{1}{4}\delta\leq b'\leq b \ \ \rm{and} \ \  c-\frac{1}{4}\delta \leq c'\leq c.$$

So $b'\geq Am$ always holds. 

By Lemma \ref{lem hyp formula} $\cosh w = \frac{\cosh a + \cosh b\cosh c}{\sinh b\sinh c}$ and then 
\begin{eqnarray*}
\sinh^2 w &=& \cosh^2 w -1 \\
&=& \frac{\cosh^2 a + \sinh^2 b + \sinh^2 c + 1 + 2\cosh a\cosh b\cosh c}{\sinh^2 b\sinh^2 c}.
\end{eqnarray*}
So
\begin{eqnarray*}
\frac{\sinh w'}{\sinh w} 
&=& \frac{\sinh b\sinh c}{\sinh b'\sinh c'}  \\
& & \left(\frac{\cosh^2 a + \sinh^2 b' + \sinh^2 c' + 1 + 2\cosh a\cosh b'\cosh c'}{\cosh^2 a + \sinh^2 b + \sinh^2 c + 1 + 2\cosh a\cosh b\cosh c}\right)^{\frac{1}{2}} \\
&\leq& \frac{\sinh b}{\sinh(b-\frac{1}{4}\delta)}  \frac{\sinh c}{\sinh(c-\frac{1}{4}\delta)} \\
&=& \left(\cosh\frac{\delta}{4} + \frac{1}{\tanh(b-\frac{\delta}{4})} \sinh\frac{\delta}{4}\right)  \left(\cosh\frac{\delta}{4} + \frac{1}{\tanh(c-\frac{\delta}{4})} \sinh\frac{\delta}{4}\right) \\
&\leq& \left(\cosh\frac{\delta}{4} + \frac{\sinh\frac{\delta}{4}}{\tanh A} \right) \left(\cosh\frac{\delta}{4} + \frac{\sinh\frac{\delta}{4}}{\tanh (\arcsinh(\frac{A}{2\pi})-\frac{\delta}{4})} \right)\\
&<& 2
\end{eqnarray*}
for $\delta \in (0, A_0]$ where $A_0$ is some constant only depending on $A$.
Recall that \eqref{w-small} says that $\sinh w\leq \frac{\pi m}{b}$. So we have that for $0<\delta\leq A_0$,  the inequality that $\sinh w'\leq \frac{2\pi m}{b}$ always holds, i.e, equation\eqref{w-small-2} always holds. Thus the assumptions in Lemma \ref{lem type 2.2 2.3} and \ref{lem type 2.4} are always satisfied.

As the two boundary lengths $2b$ and $2c$ simultaneously reduce by $\frac{1}{2}\delta$,  the corresponding length $d$ in types 2.1, 2.5 and 2.6 are unchanged. For remaining types 2.2, 2.3 and 2.4, it follows by Lemma \ref{lem type 2.2 2.3} and \ref{lem type 2.4} that the length $d$ decreases at least $\left(1-O_A\left(\frac{m}{b}\right)\right) \frac{\delta}{4}$. By Lemma \ref{lem J contain two of 3 kinds type 2} $J_j$ must contain at least two of the decreasing types. So we have
$$\ell(J_j) - \ell(J'_j) \geq \frac{1}{2} \left(1-O_A\left(\frac{m}{b}\right)\right) \delta$$
which together with \eqref{eqn eta-eta' = sum J-J'} imply that
$$\ell_\eta (X) - \ell_\eta(X_\delta) \geq \frac{1}{2} \left(1-O_A\left(\frac{m}{b}\right)\right) \delta = \frac{1}{2} \left(1-O_A\left(\frac{mn}{\sum x_i}\right)\right) \delta$$
where in the last equation we apply $2b\geq\frac{\sum x_i}{n}$. 
 
The proof is complete.
\ep

\bp[Proof of Proposition \ref{prop filling curve decrease a little}]
By assumption, $\sum x_i \geq n(2Am+\delta)\geq 2Anm$. So we have 
\[\left(\frac{mn}{\sum x_i}\right)^2\leq \frac{1}{2A}\cdot \left(\frac{mn}{\sum x_i}\right).\]

Then the conclusion clearly follows by Proposition \ref{key-pro-1} and \ref{key-pro-2}.
\ep

\bibliographystyle{amsalpha}
\bibliography{ref}

\providecommand{\bysame}{\leavevmode\hbox to3em{\hrulefill}\thinspace}
\providecommand{\MR}{\relax\ifhmode\unskip\space\fi MR }
% \MRhref is called by the amsart/book/proc definition of \MR.
\providecommand{\MRhref}[2]{%
  \href{http://www.ams.org/mathscinet-getitem?mr=#1}{#2}
}
\providecommand{\href}[2]{#2}
\begin{thebibliography}{GLMST21}

\bibitem[BBD88]{BBD88}
Peter Buser, Marc Burger, and Jozef Dodziuk, \emph{Riemann surfaces of large
  genus and large {$\lambda_1$}}, Geometry and analysis on manifolds
  ({K}atata/{K}yoto, 1987), Lecture Notes in Math., vol. 1339, Springer,
  Berlin, 1988, pp.~54--63.

\bibitem[Ber16]{B-book}
Nicolas Bergeron, \emph{The spectrum of hyperbolic surfaces}, Universitext,
  Springer, Cham; EDP Sciences, Les Ulis, 2016, Appendix C by Valentin Blomer
  and Farrell Brumley, Translated from the 2011 French original by Brumley
  [2857626].

\bibitem[BM04]{BM04}
Robert Brooks and Eran Makover, \emph{Random construction of {R}iemann
  surfaces}, J. Differential Geom. \textbf{68} (2004), no.~1, 121--157.

\bibitem[Bus92]{Buser10}
Peter Buser, \emph{Geometry and spectra of compact {R}iemann surfaces},
  Progress in Mathematics, vol. 106, Birkh\"auser Boston, Inc., Boston, MA,
  1992.

\bibitem[Che75]{Cheng75}
Shiu~Yuen Cheng, \emph{Eigenvalue comparison theorems and its geometric
  applications}, Math. Z. \textbf{143} (1975), no.~3, 289--297.

\bibitem[DGZZ20]{DGZZ20-multi}
Vincent {Delecroix}, Elise {Goujard}, Peter {Zograf}, and Anton {Zorich},
  \emph{{Large genus asymptotic geometry of random square-tiled surfaces and of
  random multicurves}}, arXiv e-prints (2020), arXiv:2007.04740.

\bibitem[GJ78]{GJ78}
Stephen Gelbart and Herv\'{e} Jacquet, \emph{A relation between automorphic
  representations of {${\rm GL}(2)$} and {${\rm GL}(3)$}}, Ann. Sci. \'{E}cole
  Norm. Sup. (4) \textbf{11} (1978), no.~4, 471--542.

\bibitem[GLMST21]{GMST19}
Clifford Gilmore, Etienne Le~Masson, Tuomas Sahlsten, and Joe Thomas,
  \emph{Short geodesic loops and {$L^p$} norms of eigenfunctions on large genus
  random surfaces}, Geom. Funct. Anal. \textbf{31} (2021), no.~1, 62--110.

\bibitem[GPY11]{GPY11}
Larry Guth, Hugo Parlier, and Robert Young, \emph{Pants decompositions of
  random surfaces}, Geom. Funct. Anal. \textbf{21} (2011), no.~5, 1069--1090.

\bibitem[Hej76]{Hej76}
Dennis~A. Hejhal, \emph{The {S}elberg trace formula for {${\rm PSL}(2,R)$}.
  {V}ol. {I}}, Lecture Notes in Mathematics, Vol. 548, Springer-Verlag,
  Berlin-New York, 1976.

\bibitem[{Hid}21]{Hide21}
Will {Hide}, \emph{{Spectral gap for Weil-Petersson random surfaces with
  cusps}}, arXiv e-prints (2021), arXiv:2107.14555.

\bibitem[HM21]{HM21}
Will {Hide} and Michael {Magee}, \emph{{Near optimal spectral gaps for
  hyperbolic surfaces}}, arXiv e-prints (2021), arXiv:2107.05292.

\bibitem[Hub74]{Hub74}
Heinz Huber, \emph{\"{U}ber den ersten {E}igenwert des {L}aplace-{O}perators
  auf kompakten {R}iemannschen {F}l\"{a}chen}, Comment. Math. Helv. \textbf{49}
  (1974), 251--259.

\bibitem[Iwa89]{Iwa89}
H.~Iwaniec, \emph{Selberg's lower bound of the first eigenvalue for congruence
  groups}, Number theory, trace formulas and discrete groups ({O}slo, 1987),
  Academic Press, Boston, MA, 1989, pp.~371--375.

\bibitem[Iwa96]{Iwa96}
Henryk Iwaniec, \emph{The lowest eigenvalue for congruence groups}, Topics in
  geometry, Progr. Nonlinear Differential Equations Appl., vol.~20,
  Birkh\"{a}user Boston, Boston, MA, 1996, pp.~203--212.

\bibitem[Iwa02]{Iwa02}
\bysame, \emph{Spectral methods of automorphic forms}, second ed., Graduate
  Studies in Mathematics, vol.~53, American Mathematical Society, Providence,
  RI; Revista Matem\'{a}tica Iberoamericana, Madrid, 2002.

\bibitem[Jen84]{Jenni84}
Felix Jenni, \emph{\"{U}ber den ersten {E}igenwert des {L}aplace-{O}perators
  auf ausgew\"{a}hlten {B}eispielen kompakter {R}iemannscher {F}l\"{a}chen},
  Comment. Math. Helv. \textbf{59} (1984), no.~2, 193--203.

\bibitem[Kim03]{KS03}
Henry~H. Kim, \emph{Functoriality for the exterior square of {${\rm GL}_4$} and
  the symmetric fourth of {${\rm GL}_2$}}, J. Amer. Math. Soc. \textbf{16}
  (2003), no.~1, 139--183, With appendix 1 by Dinakar Ramakrishnan and appendix
  2 by Kim and Peter Sarnak.

\bibitem[KS02]{KS02}
Henry~H. Kim and Freydoon Shahidi, \emph{Functorial products for {${\rm
  GL}_2\times{\rm GL}_3$} and the symmetric cube for {${\rm GL}_2$}}, Ann. of
  Math. (2) \textbf{155} (2002), no.~3, 837--893, With an appendix by Colin J.
  Bushnell and Guy Henniart.

\bibitem[Lal89]{Lall89}
Steven~P. Lalley, \emph{Renewal theorems in symbolic dynamics, with
  applications to geodesic flows, non-{E}uclidean tessellations and their
  fractal limits}, Acta Math. \textbf{163} (1989), no.~1-2, 1--55.

\bibitem[LRS95]{LRS95}
W.~Luo, Z.~Rudnick, and P.~Sarnak, \emph{On {S}elberg's eigenvalue conjecture},
  Geom. Funct. Anal. \textbf{5} (1995), no.~2, 387--401.

\bibitem[LW21]{LW21}
Michael {Lipnowski} and Alex {Wright}, \emph{{Towards optimal spectral gaps in
  large genus}}, arXiv e-prints (2021), arXiv:2103.07496.

\bibitem[{Mag}20]{Magee-20}
Michael {Magee}, \emph{{Letter to {B}ram {P}etri}},
  \url{https://www.maths.dur.ac.uk/users/michael.r.magee/diameter.pdf}, 2020.

\bibitem[Mir07a]{Mirz07}
Maryam Mirzakhani, \emph{Simple geodesics and {W}eil-{P}etersson volumes of
  moduli spaces of bordered {R}iemann surfaces}, Invent. Math. \textbf{167}
  (2007), no.~1, 179--222.

\bibitem[Mir07b]{Mirz07-int}
\bysame, \emph{Weil-{P}etersson volumes and intersection theory on the moduli
  space of curves}, J. Amer. Math. Soc. \textbf{20} (2007), no.~1, 1--23.

\bibitem[Mir10]{Mirz10}
\bysame, \emph{On {W}eil-{P}etersson volumes and geometry of random hyperbolic
  surfaces}, Proceedings of the {I}nternational {C}ongress of {M}athematicians.
  {V}olume {II}, Hindustan Book Agency, New Delhi, 2010, pp.~1126--1145.

\bibitem[Mir13]{Mirz13}
\bysame, \emph{Growth of {W}eil-{P}etersson volumes and random hyperbolic
  surfaces of large genus}, J. Differential Geom. \textbf{94} (2013), no.~2,
  267--300.

\bibitem[MNP20]{MNP20}
Michael {Magee}, Fr{\'e}d{\'e}ric {Naud}, and Doron {Puder}, \emph{{A random
  cover of a compact hyperbolic surface has relative spectral gap
  $\frac{3}{16}-\epsilon$}}, arXiv e-prints (2020), arXiv:2003.10911.

\bibitem[Mon15]{Mondal15}
Sugata Mondal, \emph{On largeness and multiplicity of the first eigenvalue of
  finite area hyperbolic surfaces}, Math. Z. \textbf{281} (2015), no.~1-2,
  333--348.

\bibitem[{Mon}20]{Monk20}
Laura {Monk}, \emph{{Benjamini-Schramm convergence and spectrum of random
  hyperbolic surfaces of high genus}}, Analysis \& PDE (2020), to appear.

\bibitem[MP19]{MP19}
Maryam Mirzakhani and Bram Petri, \emph{Lengths of closed geodesics on random
  surfaces of large genus}, Comment. Math. Helv. \textbf{94} (2019), no.~4,
  869--889.

\bibitem[MZ15]{MZ15}
Maryam Mirzakhani and Peter Zograf, \emph{Towards large genus asymptotics of
  intersection numbers on moduli spaces of curves}, Geom. Funct. Anal.
  \textbf{25} (2015), no.~4, 1258--1289.

\bibitem[NWX20]{NWX20}
Xin {Nie}, Yunhui {Wu}, and Yuhao {Xue}, \emph{{Large genus asymptotics for
  lengths of separating closed geodesics on random surfaces}}, arXiv e-prints
  (2020), arXiv:2009.07538.

\bibitem[PWX21]{PWX20}
Hugo {Parlier}, Yunhui {Wu}, and Yuhao {Xue}, \emph{{The simple separating
  systole for hyperbolic surfaces of large genus}}, Journal of the Institute of
  Mathematics of Jussieu (2021), to appear.

\bibitem[Sar95]{Sar95-survey}
Peter Sarnak, \emph{Selberg's eigenvalue conjecture}, Notices Amer. Math. Soc.
  \textbf{42} (1995), no.~11, 1272--1277.

\bibitem[Sel56]{Selb56}
A.~Selberg, \emph{Harmonic analysis and discontinuous groups in weakly
  symmetric {R}iemannian spaces with applications to {D}irichlet series}, J.
  Indian Math. Soc. (N.S.) \textbf{20} (1956), 47--87.

\bibitem[Sel65]{Sel65}
Atle Selberg, \emph{On the estimation of {F}ourier coefficients of modular
  forms}, Proc. {S}ympos. {P}ure {M}ath., {V}ol. {VIII}, Amer. Math. Soc.,
  Providence, R.I., 1965, pp.~1--15.

\bibitem[Wol82]{Wolpert82}
Scott Wolpert, \emph{The {F}enchel-{N}ielsen deformation}, Ann. of Math. (2)
  \textbf{115} (1982), no.~3, 501--528.

\bibitem[Wol10]{Wolpert-book}
Scott~A. Wolpert, \emph{Families of {R}iemann surfaces and {W}eil-{P}etersson
  geometry}, CBMS Regional Conference Series in Mathematics, vol. 113,
  Published for the Conference Board of the Mathematical Sciences, Washington,
  DC; by the American Mathematical Society, Providence, RI, 2010.

\bibitem[Wri20]{Wright-tour}
Alex Wright, \emph{A tour through {M}irzakhani's work on moduli spaces of
  {R}iemann surfaces}, Bull. Amer. Math. Soc. (N.S.) \textbf{57} (2020), no.~3,
  359--408.

\bibitem[WX21]{WX18}
Yunhui {Wu} and Yuhao {Xue}, \emph{{Small eigenvalues of closed Riemann
  surfaces for large genus}}, Transactions of the American Mathematical Society
  (2021), to appear.

\end{thebibliography}

\end{document}